\documentclass[a4paper,10pt]{amsart}
\usepackage[english]{babel}
\usepackage{amsmath,tikz-cd}
\usepackage{amssymb}
\usepackage[T1]{fontenc} 
\usepackage[utf8]{inputenc} 
\usepackage{lmodern} 
\usepackage{mathrsfs}
\usepackage{enumerate}
\usepackage{mathtools,yfonts}
\usepackage{tikz,tkz-euclide}
\usepackage[font=scriptsize]{caption}
\usepackage{todonotes}
\usepackage{pgfplots}
\usetikzlibrary{arrows}
\usepackage[shortlabels]{enumitem}
\usetikzlibrary{positioning,calc,arrows.meta}
\usepackage{makecell}

\usepackage{xcolor} 
\definecolor{DarkRed}{RGB}{173,0,0}
\definecolor{LightRed}{RGB}{201,0,0}
\usepackage[
 colorlinks=true,
 linkcolor=DarkRed,
 urlcolor=LightRed,
 citecolor=LightRed
]{hyperref}

\usepackage{tcolorbox}
\usepackage[]{algorithm2e}

\newtheorem{thm}{Theorem}[section]
\newtheorem{Lemma}[thm]{Lemma}
\newtheorem{Proposition}[thm]{Proposition}
\newtheorem{Corollary}[thm]{Corollary}

\newtheorem*{mainthm}{Theorem}

\theoremstyle{definition}
\newtheorem{Notation}[thm]{Notation}

\newtheorem{Construction}[thm]{Construction}
\newtheorem{Definition}[thm]{Definition}
\newtheorem{Remark}[thm]{Remark}

\newtheorem{Assumption}[thm]{Assumptions}

\definecolor{wwwwww}{rgb}{0.4,0.4,0.4}

\newcommand{\PP}{\mathbb{P}}
\newcommand{\ZZ}{\mathbb{Z}}
\newcommand{\QQ}{\mathbb{Q}}
\newcommand{\RR}{\mathbb{R}}

\newcommand{\kk}{k}

\newcommand{\OO}{\mathcal{O}} 

\DeclareMathOperator{\Spec}{Spec}
\DeclareMathOperator{\Pic}{Pic}

\DeclareMathOperator{\Aut}{Aut}
\DeclareMathOperator{\PGL}{PGL}

\DeclareMathOperator{\Center}{Center}

\DeclareMathOperator{\codim}{codim}

\newcommand{\A}{\mathbb{A}}
\newcommand{\F}{\mathbb{F}}

\newcommand{\I}{\mathcal{I}}

\newcommand{\Q}{\mathbb{Q}} 

\DeclareMathOperator{\Bl}{Bl}
\DeclareMathOperator{\Cl}{Cl}
\DeclareMathOperator{\NE}{NE}
\DeclareMathOperator{\NEbar}{\overline{NE}}

\DeclareMathOperator{\Hilb}{Hilb}

\DeclareMathOperator{\Eff}{Eff}

\DeclareMathOperator{\mult}{mult}

\DeclareMathOperator{\Ext}{Ext}

\DeclareMathOperator{\Exc}{Exc}

\DeclareMathOperator{\Proj}{Proj}

\DeclareMathOperator{\Diff}{Diff}
\DeclareMathOperator{\Sing}{Sing}

\DeclareMathOperator{\WCl}{WCl}

\DeclareMathOperator{\Vol}{Vol}

\DeclareMathOperator{\Indet}{Indet}
\DeclareMathOperator{\ord}{ord}

\DeclareMathOperator{\Cone}{Cone}
\DeclareMathOperator{\Nef}{Nef}
\DeclareMathOperator{\Mov}{Mov}

\newcommand\Span[1]{\langle{#1}\rangle}
\newcommand{\quot}{/\hspace{-1.2mm}/}

\hypersetup{pdfpagemode=UseNone}
\hypersetup{pdfstartview=FitH}

\providecommand{\mP}{\mathbb{P}}
\providecommand{\mQ}{\mathbb{Q}}

\providecommand{\sC}{\mathcal{C}}

\providecommand{\sE}{\mathcal{E}}
\providecommand{\sF}{\mathcal{F}}

\providecommand{\sH}{\mathcal{H}}

\providecommand{\sO}{\mathcal{O}}

\providecommand{\sQ}{\mathcal{Q}}

\providecommand{\sU}{\mathcal{U}}

\begin{document}

\title[Genus $8$ Fano $3$-folds: rationality and K-stability]{The moduli space of genus $8$ Fano $3$-folds with two $\frac12(1,1,1)$-points: rationality and K-stability}

\author[Joseph Malbon]{Joseph Malbon}
\address{School of Mathematics, The University of Edinburgh, Edinburgh, UK}
\email{j.malbon@sms.ed.ac.uk}

\author[Alex Massarenti]{Alex Massarenti}
\address{Dipartimento di Matematica e Informatica, Universit\`a di Ferrara, Via Machiavelli 30, 44121 Ferrara, Italy}
\email{msslxa@unife.it}

\author[Francesco Zucconi]{Francesco Zucconi}
\address{Dipartimento di Scienze Matematiche, Informatiche e Fisiche, Universit\`a di Udine, Via delle Scienze 206, 33100 Udine, Italy}
\email{francesco.zucconi@uniud.it}

\begin{abstract}
Let $\mathcal F_{8,2\times\frac12(1,1,1)}$ be the moduli space of genus $8$ Fano $3$-folds with two terminal quotient singularities of type $\frac12(1,1,1)$. We prove that $\mathcal F_{8,2\times\frac12(1,1,1)}$ is rational. Moreover, we show that a general member of this family is K-stable.
\end{abstract}

\date{\today}

\maketitle
\tableofcontents

\section*{Introduction}

The aim of this paper is to study a natural family of singular Fano $3$-folds from two complementary points of view: K-stability and the birational geometry of its moduli space.

Let $X$ be a $\QQ$-Fano $3$-fold, that is, a normal projective $3$-fold with klt singularities and ample $\QQ$-Cartier anticanonical divisor $-K_X$. We are interested in the genus $8$ family whose general member has precisely two terminal quotient singularities of type $\frac12(1,1,1)$. We denote the corresponding moduli space by $\mathcal F_{8,2\times\frac12(1,1,1)}$.

This family appears naturally in the study of primary $\QQ$-Fano $3$-folds with non-Gorenstein points. It is a singular analogue of the classical families of smooth prime Fano $3$-folds, which have been studied extensively since the work of Iskovskikh, Mori, Mukai, and many others \cite{Iskovskikh77,Iskovskikh78,MoriMukai81,Mukai89}. In the singular setting considered here, the relevant $3$-folds admit an explicit construction from two twisted cubics on a smooth quadric $3$-fold.

When $X$ is a smooth prime Fano $3$-fold, the even integer $(-K_X)^3$ is called the degree of $X$, and one writes
$$
        g(X):=\frac12(-K_X)^3+1.
$$
The integer $g(X)$ is the genus of $X$. It is one of the fundamental numerical invariants in the classification of prime Fano $3$-folds. The possible genera are very restricted: one has $2\leq g\leq 10$ or $g=12$. Thus the genus provides a natural way to organize the deformation families of smooth prime Fano $3$-folds.

The geometry of these varieties is strongly reflected in their curves of lines. For a general prime Fano $3$-fold $X$ of any genus, the Hilbert scheme of lines contained in $X$ is a smooth curve \cite[Theorem~4.2.7]{AGV99}. In maximal genus $g=12$, Mukai proved that a general prime Fano $3$-fold can be reconstructed from the pair $(C,\theta)$, where $C$ is a smooth plane quartic parametrizing the lines on $X$, and $\theta$ is an even theta-characteristic on $C$ \cite{Muk04}. An analogous reconstruction result in genus $g=8$ was obtained in \cite{FS10}. In particular, the moduli space $\mathcal F_{12}$ of genus $12$ prime Fano $3$-folds is birational to the moduli space $\mathcal S_3^+$ of genus $3$ spin curves with an even theta-characteristic. Since $\mathcal S_3^+$ is rational, this gives the rationality of $\mathcal F_{12}$.

We now turn to the singular setting. Recall that the index of a $\QQ$-Fano $3$-fold $X$ is
$$
        \iota(X):=
        \max\{\,n\in\mathbb N \mid -K_X\sim nH
        \text{ for some ample Cartier divisor }H\,\},
$$
whereas the $\QQ$-index is
$$
        \iota_{\QQ}(X):=
        \max\{\,n\in\mathbb N \mid -K_X\sim_{\QQ} nH
        \text{ for some ample Weil divisor }H\,\}.
$$
If $X$ has only Gorenstein singularities, then $\iota(X)=\iota_{\QQ}(X)$. A $\QQ$-Fano $3$-fold is called primary if $\iota_{\QQ}(X)=1$. Takagi proved that if $X$ is a primary $\QQ$-Fano $3$-fold with non-Gorenstein cyclic quotient terminal singularities and if the anticanonical system $|-K_X|$ contains a K3 surface with only rational double points, then $g(X)\leq 8$. Moreover, if $g(X)=8$, then $X$ has at most two non-Gorenstein points, and they are of type $\frac12(1,1,1)$ \cite[Theorem~1.5]{Tak06}.

Therefore the moduli space $\mathcal F_{8,2\times\frac12(1,1,1)}$ of genus $8$ Fano $3$-folds with two singularities of type $\frac12(1,1,1)$ is the natural non-Gorenstein analogue of the maximal genus moduli space $\mathcal F_{12}$ in the smooth theory. In both cases, one is at the extremal end of the relevant genus range: $g=12$ for smooth prime Fano $3$-folds, and $g=8$ for primary $\QQ$-Fano $3$-folds with non-Gorenstein cyclic quotient terminal singularities and anticanonical K3 sections. This makes $\mathcal F_{8,2\times\frac12(1,1,1)}$ a particularly natural testing ground for extending the birational and modular geometry of smooth prime Fano $3$-folds to the terminal non-Gorenstein setting.

Our main focus is K-stability. This was introduced as an algebro-geometric stability condition characterizing the existence of K\"ahler--Einstein metrics on Fano varieties. The differential-geometric problem goes back to the search for canonical metrics, and in the Fano case it led to the Yau--Tian--Donaldson picture. Tian introduced analytic stability conditions for Fano manifolds \cite{Tia97}, Donaldson formulated K-stability in terms of test configurations \cite{Don02}, and the Yau--Tian--Donaldson theorem for smooth Fano manifolds was proved by Chen--Donaldson--Sun and Tian \cite{CDS15a,CDS15b,CDS15c,Tian15}. In the singular setting, K-stability is now understood as the correct stability condition for constructing moduli spaces of $\QQ$-Fano varieties \cite{Odaka13,LWX21,Xu20}.

A key feature of the modern theory is the valuative criterion. To a prime divisor $F$ over a $\QQ$-Fano variety $X$ one associates the log discrepancy $A_X(F)$ and the expected vanishing order $S_X(F)$. K-stability can then be tested by comparing these two quantities \cite{Fujita16,Li17,BJ20}. This point of view is particularly well suited to singular varieties, since destabilizing valuations may have centers at singular points, along curves, or along divisors.

K-stability has been studied very intensively for smooth Fano $3$-folds, notably in the work on the Calabi problem for Fano $3$-folds \cite{ACCFKMSV21}. The smooth genus $12$ case is especially relevant for the present paper. Smooth prime Fano $3$-folds of genus $12$, usually denoted by $V_{22}$, form the maximal-genus family in the smooth theory. From the point of view of K-stability this family is already subtle: a general smooth $V_{22}$ is K-stable, but the family also contains strictly K-semistable members, and a complete description of the K-polystable and K-semistable loci is still not known \cite{ACCFKMSV21,DenisovaKaloghiros25,KLPZ26}. 

Let us also place our result in the context of recent work on K-stability of singular Fano threefolds. In the Gorenstein direction, the closest examples are the one-nodal prime Fano threefolds of genus \(12\). Denisova and Kaloghiros proved that a general one-nodal prime Fano threefold of genus \(12\) is K-polystable; more precisely, they study the four one-nodal families appearing in Prokhorov's classification and prove the existence of K-stable members in two of them and of a K-polystable member in a third one \cite{DenisovaKaloghiros25}. These varieties are terminal Gorenstein and have an ordinary double point. The K-moduli picture of genus \(12\) prime Fano threefolds has also been studied recently by Kaloghiros--Liu--Petracci--Zhao, who describe the boundary of the K-moduli stack of \(V_{22}\)'s as a purely divisorial boundary with four irreducible components corresponding to Prokhorov's one-nodal families \cite{KLPZ26}. Thus the genus \(12\) Gorenstein nodal case is closely related to the boundary of the K-moduli space of smooth prime Fano threefolds.

There are also several explicit results for non-Gorenstein Fano threefolds, especially in the weighted hypersurface setting. Kim--Okada--Won proved K-stability for birationally superrigid quasi-smooth Fano \(3\)-fold weighted hypersurfaces of index \(1\) \cite{KOW23}. Sano--Tasin studied K-stability of Fano weighted hypersurfaces in broader dimension and index ranges, proving K-stability under suitable numerical hypotheses and obtaining applications in dimension at most \(3\) \cite{ST24}. More recently, Campo--Okada settled the K-stability problem for quasi-smooth Fano \(3\)-fold hypersurfaces of index \(1\), using lower bounds for delta invariants obtained via the Abban--Zhuang method \cite{CO24}. These works show that terminal quotient singularities naturally occur in explicit K-stability problems. 

The family considered in this paper is of a different birational nature: it is not a weighted hypersurface family, but a primary genus \(8\) family constructed from pairs of twisted cubics on a smooth quadric threefold, whose general member has two terminal non-Gorenstein points of type \(\frac12(1,1,1)\). This creates some additional difficulties. First, the anticanonical divisor is not Cartier: the sheaf $\OO_X(-K_X)$ must be treated as a rank-one reflexive sheaf, and the anticanonical linear system is a Weil linear system rather than an ordinary Cartier linear system. Second, the most delicate potential destabilizing center is the curve joining the two non-Gorenstein points. Blowing up this curve produces a singular model, so the volume computations have to be carried out on normal $\QQ$-factorial varieties using $\QQ$-Cartier divisors and reflexive divisorial sheaves. Finally, the relevant birational geometry is not visible from a single projective embedding. We compare the blow-up of the curve with an auxiliary weak Fano model obtained by changing the order of the birational operations; the two models are related by a small crepant birational map, and this comparison is essential for computing the volume function needed in the valuative criterion.

The second theme of the paper is rationality of the moduli space. Rationality questions for moduli spaces are classical and remain central in birational geometry: a rational moduli space admits, at least birationally, an explicit parametrization. For Fano $3$-folds, such parametrizations are often closely related to the projective models and auxiliary curves naturally associated with the varieties. In our case, the same explicit construction that produces the singular Fano $3$-folds also gives a birational description of the moduli space.

\begin{mainthm}
The moduli space $\mathcal F_{8,2\times\frac12(1,1,1)}$ is rational. Moreover, its general point parametrizes a K-stable Fano $3$-fold.
\end{mainthm}

Let us briefly explain the proof. We start with a smooth quadric $3$-fold $Q\subset \PP^4$ and two smooth twisted cubics $\Gamma_1,\Gamma_2\subset Q$ satisfying suitable incidence and transversality conditions. Blowing up $Q$ along $\Gamma_1\sqcup\Gamma_2$ gives a weak Fano $3$-fold. There are seven distinguished flopping curves: one comes from the conic cut out by the two spans of the twisted cubics, and six come from certain bisecant lines. After flopping these curves, two surfaces become planes and can be contracted. The result is a Fano $3$-fold with two terminal quotient points of type $\frac12(1,1,1)$.

For K-stability, we choose a particularly symmetric member of this family. It carries an action of the dihedral group $D_6$ of order $12$, and this action has no fixed points. The equivariant valuative criterion of Zhuang \cite{Zhuang21} reduces the proof of K-polystability to the exclusion of $D_6$-invariant destabilizing centers. Since the automorphism group of this member is finite, K-polystability will then imply K-stability.

The most delicate invariant center is a curve $l$ joining the two singular points. To study it, we blow up $l$ and denote the exceptional divisor by $E$. We compute the volume function of the divisors $f^*(-K_X)-uE$ by comparing this blow-up with a second weak Fano model obtained by performing the same birational operations in a different order. The two models are related by a small crepant birational map, given by six flops. This comparison gives explicit formulas for the relevant volumes and shows that the divisorial valuation defined by $E$ itself is not destabilizing. To exclude the curve $l$ as a center, we then apply the Abban--Zhuang adjunction method \cite{AZ20,AZZ23}. This reduces the problem to a computation on a surface obtained from $E$ by resolution, where the relevant refined linear series is controlled by explicit Zariski decompositions.

All remaining invariant centers are excluded by simpler arguments. Points are impossible because the $D_6$-action is fixed-point-free. Divisorial centers are ruled out using the computation of the class group. Curves different from $l$ are controlled by very general anticanonical K3 sections and by the orbit structure of the seven flopping curves. Thus every invariant possible destabilizing center has local stability threshold strictly larger than $1$. This proves K-stability of the symmetric member. Finally, K-stability is open in $\QQ$-Gorenstein families of $\QQ$-Fano varieties \cite{BLX22}, and hence the general member of $\mathcal F_{8,2\times\frac12(1,1,1)}$ is K-stable.

We now describe the rationality argument. The construction above associates to a suitable unordered pair of twisted cubics on a smooth quadric $3$-fold a point of $\mathcal F_{8,2\times\frac12(1,1,1)}$. Conversely, a general $3$-fold in the family allows one to recover the initial quadric and the unordered pair of twisted cubics. Therefore the moduli space is birational to a quotient of a natural parameter space of pairs of twisted cubics by the automorphism group of the quadric. This parameter space is described explicitly, and the quotient is shown to be rational by applying the slice method.

\subsection*{Summary of the paper}

In Section~\ref{Sec1} we recall the geometry of Hilbert schemes of rational curves on a smooth quadric $3$-fold. In Sections~\ref{Sec2} and~\ref{Sec3} we develop the construction from pairs of twisted cubics and prove the rationality of $\mathcal F_{8,2\times\frac12(1,1,1)}$ by reconstructing the initial data from a general member and applying the slice method to the resulting quotient. In Section~\ref{Sec4} we carry out the birational and valuative analysis needed for K-stability: we study the blow-up of the invariant curve, construct the auxiliary weak Fano model, compute the relevant volumes, and apply Abban--Zhuang adjunction. Finally, in Section~\ref{Sec5} we specialize to the $D_6$-symmetric member, exclude all invariant destabilizing centers, prove its K-stability, and deduce K-stability of the general member by openness.

\subsection*{Acknowledgements}

Malbon and Zucconi thank the University of Ferrara for its hospitality, and Massarenti thanks the University of Edinburgh for its hospitality, during visits in which part of this work was carried out. The authors acknowledge the support of the PRIN 2022 project ``Birational geometry of moduli spaces and special varieties'' funded by the European Union under Next Generation EU and by the Italian Ministry of University and Research. The authors thank Ivan Cheltsov and Hamid Abban for discussions on K-stability, and Hiromichi Takagi for sharing his expertise on singular Fano 3-folds.

\section{Hilbert schemes of rational curves in the quadric $3$-fold}\label{Sec1}

Let $\Hilb_d^Q$ be the Hilbert scheme of degree $d$ rational curves in a smooth quadric $3$-fold $Q\subset\mathbb{P}^4$. By \cite{Per02} $\Hilb_d^Q$ is smooth and irreducible.

\begin{Proposition}\label{H2}
The Hilbert scheme $\Hilb_2^Q$ is isomorphic to the Grassmannian $Gr(3,5)$ of subvector spaces of dimension three in a $5$-dimensional vector space. Furthermore, the universal family $\mathcal{U}_{\Hilb_2^Q}$ of $\Hilb_2^Q$ embeds in the projectivization of the universal bundle $\mathcal{U}_{Gr(3,5)}$ over $Gr(3,5)$.
\end{Proposition} 
\begin{proof}
Let $\Pi\subset\mathbb{P}^4$ be a plane. Since $Q$ does not contain any plane $\Pi\cap Q$ is either a smooth conic or the union of two distinct and intersecting lines or a double line. In any case we get a $1$-dimensional scheme with Hilbert polynomial $2t+1$ and hence a point of $\Hilb_2^Q$. This yields a morphism
$$
\begin{array}{cccc}
f: & Gr(3,5) & \rightarrow & \Hilb_2^Q\\
   & \Pi & \mapsto & \Pi\cap Q
\end{array}
$$
and since $Gr(3,5)$ has Picard rank one and $\Hilb_2^Q$ is smooth the above morphism is an isomorphism. Let $f^{-1}:\Hilb_2^Q\rightarrow Gr(3,5)$ be its inverse. By the universal property of the Grassmannian there exists a morphism $\overline{f}^{-1}:\mathcal{U}_{\Hilb_2^Q}\rightarrow \mathbb{P}(\mathcal{U}_{Gr(3,5)})$ making the following diagram commutative
$$
\begin{tikzcd}
\mathcal{U}_{\Hilb_2^Q} \arrow[rr, "\overline{f}^{-1}"] \arrow[d] &  & {\mathbb{P}(\mathcal{U}_{Gr(3,5)})} \arrow[d] \\
\Hilb_2^Q \arrow[rr, "f^{-1}"]                                    &  & {Gr(3,5)}                                    
\end{tikzcd}
$$
and concluding the proof.
\end{proof}

\begin{Lemma}\label{TanG}
Let $F_Q$ be the subscheme of $Gr(3,5)$ whose general point represents a plane which is contained in a tangent space of $Q$. Then $F_Q\subset Gr(3,5)\subset\mathbb{P}^9$ is a divisor cut out on $Gr(3,5)$ by a hypersurface of degree five. In particular $\deg(F_Q) = 25$. 
\end{Lemma}
\begin{proof}
Write a plane $\Pi$ in $\mathbb{P}^4$ as 
$$
\Pi = \left\lbrace
\begin{array}{l}
x_4-\alpha_0x_0-\alpha_1x_1-\alpha_2x_2 = 0;\\ 
x_3-\beta_0x_0-\beta_1x_1-\beta_2x_2 = 0.
\end{array}\right.
$$
Note that $\Pi$ is contained in a tangent space of $Q$ if and only if the conic $C_{\Pi} = Q\cap \Pi$ is singular. Without loss of generality we may assume that $Q = \{x_0x_1-x_2x_3+x_4^2 = 0\}\subset\mathbb{P}^4$. Then the matrix of $C_{\Pi}$ takes the following form
$$
M_{C_{\Pi}} = \left(
\begin{array}{ccc}
\alpha_0^2 & \frac{2\alpha_0\alpha_1 + 1}{2} & \frac{2\alpha_0\alpha_2 - \beta_0}{2} \\ 
\frac{2\alpha_0\alpha_1 + 1}{2} & \alpha_1^2 & \frac{2\alpha_1\alpha_2 - \beta_1}{2} \\ 
\frac{2\alpha_0\alpha_2 - \beta_0}{2} & \frac{2\alpha_1\alpha_2 - \beta_1}{2} & \alpha_2^2
\end{array}\right) 
$$
and $F_Q$ is the closure in $\mathbb{P}^9$ of the image of $W = \{\det(M_{C_{\Pi}}) = 0\}\subset\mathbb{A}^6$ via the Pl\"ucker embedding
$$
pl: Gr(3,5)\rightarrow\mathbb{A}^9                                                              
$$                                                                                                 
mapping 
$$
\left(\begin{array}{ccccc} 
1 & 0 & 0 & b_0 & a_0 \\  
0 & 1 & 0 & b_1 & a_1 \\  
0 & 0 & 1 & b_2 & a_2 
\end{array}\right) \rightarrow (1,b_2,a_2,-b_1,-a_1,b_0,a_0,b_1a_2-a_1b_2,a_0b_2-a_2b_0,b_0a_1-a_0b_1).
$$
Then, fixed homogeneous coordinates $[x_0:\dots : x_9]$ on $\mathbb{P}^9$ we have that the hypersurface defined by the following polynomial
$$
\begin{array}{l}
x_0^3x_2^2 + x_0^3x_3^2 - 4x_0^2x_1x_4^2 + x_0^3x_3x_5 - 4x_0^2x_1x_4x_6 - x_0x_3^2x_6^2 + 4x_1x_4^2x_6^2 + 4x_0^3x_4x_7 + 4x_0^3x_4x_8  \\ 
- 4x_0x_4^2x_6x_8 - 2x_0^3x_2x_9 + 4x_0x_2x_4x_6x_9 + x_0^3x_9^2
\end{array} 
$$
cuts out scheme-theoretically $F_Q$ in $Gr(3,5)$.
\end{proof}

The $10$-dimensional group ${\rm{Aut}}(Q) = SO(5,\mathbb C)$ acts on $Gr(3,5)$ and $F_{Q}$ is an invariant subscheme. Let 
$[q]\in\sH^{Q}_{2}= Gr(3,5)\setminus F_{Q}$ be a general element and denote by $[\langle q\rangle]\in Gr(3,5)$ the point corresponding to the plane 
$\langle q\rangle$ generated by $q$. Let ${\rm{Stab}}(q)$ be the 
stabilizer of $[\langle q\rangle]\in Gr(3,5)$ under the 
$SO(5,\mathbb C)$-action.

\begin{Lemma}\label{stabiliserone} 
Let $[q]\in\sH^{Q}_{2}$ be an element such that $[\langle q\rangle]\not\in F_{Q}$. The stabilizer of the ${\rm{Aut}}(Q)$-action on ${\rm{Stab}}(q)$ is isomorphic to $SO(3,\mathbb C)\times SO(2,\mathbb C)$.
\end{Lemma}
\begin{proof} 
Let $\mP(W) = \langle q\rangle$. To stabilize $q = Q\cap \mP(W)$ under the $SO(5,\mathbb C)$-action is equivalent to stabilize $\mP(W)$. On the other hand, to stabilize $\mP(W)$ is equivalent to stabilize $\mP(W)$ and $\mP(W^{\perp})$. 
\end{proof}

\begin{Corollary}\label{charactofHilb2} 
The open subscheme $\sH^{Q}_{2}\hookrightarrow{\rm{Hilb}}^{Q}_{2}$ is isomorphic to the open subscheme $\mathbb G(3,5)\setminus F_{Q}$. Moreover ${\rm{Aut}}(Q)$ acts transitively on $\sH^{Q}_{2}$ with stabilizer isomorphic to $SO(3,\mathbb C)\times SO(2,\mathbb C)$.
\end{Corollary}
\begin{proof} It follows from Proposition \ref{H2} and Lemma \ref{stabiliserone}.
\end{proof}
    
Let $\sE'\to\mathbb G(3,5)$ be the natural rank three vector bundle on $\mathbb G(3,5)$ such that $\mP(\sE')= \sU_{\mathbb G(3,5)}$ is the universal $\mP^2$-bundle over $\mathbb G(3,5)$, that is the incidence variety $\{[\Pi],p)\in \mathbb G(3,5)\times \mP^{4} |\,  p\in \Pi\}$. 

\begin{Lemma}\label{relativehyperquadric} 
The universal family $\sU_{{\rm{Hilb}}^{Q}_{2}}\to{\rm{Hilb}}^{Q}_{2}$ is embedded as a relative conic of $\mP(\sE')$.
\end{Lemma}
\begin{proof} 
We stress that 
 ${\rm{Pic}} \,(\mP(\sE'))=[T]\cdot\mathbb Z\oplus {\rm{Pic}} (Gr(3,5))$ where $T$ is the tautological divisor. By Proposition \ref{H2} we can identify ${\rm{Hilb}}^{Q}_{2}$ with $Gr(3,5)$. By universality it follows that the restiction of the universal family $\sU_{\sH^{Q}_{2}}\to\sH^{Q}_{2}$ is embedded in $\sU_{Gr(3,5)}\to Gr(3,5)$ as an open subscheme of a relative conic. 
\end{proof}

By Lemma \ref{relativehyperquadric} $\sU_{{\rm{Hilb}}^{Q}_{2}}\to{\rm{Hilb}}^{Q}_{2}$ is embedded as a relative conic of $\mP(\sE')$. However, it is possible to present the restricted universal family $\sU_{\sH^{Q}_{2}}\rightarrow\sH^{Q}_{2}$ as a $\mP^{1}$-bundle.

\begin{Proposition}\label{relativehyperquadricbis} 
The universal family $\pi_{\sH^{Q}_{2}}\colon \sU_{\sH^{Q}_{2}}\rightarrow\sH^{Q}_{2}$ is a $\mP^{1}$-bundle over $\sH^{Q}_{2}$.
\end{Proposition}
\begin{proof} 
By Corollary \ref{charactofHilb2} $\sH^{Q}_{2}\subset {\rm{Hilb}}^{Q}_{2}$ is an irreducible and reduced open subvariety. Computing the normal bundle of a conic inside $Q$ one sees that $\sH^{Q}_{2}$ is a smooth. By definition of the universal family $\pi_{\sH^{Q}_{2}}\colon \sU_{\sH^{Q}_{2}}\rightarrow\sH^{Q}_{2}$ is a proper surjective morphism between normal schemes and every fiber is a smooth rational curve. Then by \cite[Theorem II.2.2.8]{Kol96} $\pi_{\sH^{Q}_{2}}\colon \sU_{\sH^{Q}_{2}}\rightarrow\sH^{Q}_{2}$ is a $\mP^{1}$-bundle.
\end{proof}

\section{Twisted cubics in quadric $3$-folds}\label{Sec2}
Recall that a normal and $\mathbb{Q}$-factorial projective variety $X$ is
\begin{itemize}
\item[-] \textit{weak Fano} if $-K_X$ is nef and big;
\item[-] \textit{log Fano} if there exists an effective divisor $D\subset X$ such that $-(K_X+D)$ is ample and the pair $(X,D)$ is Kawamata log terminal.
\end{itemize}

Recall also that a normal projective $\QQ$-factorial variety $X$ is a \emph{Mori dream space} if the following conditions hold:
\begin{itemize}
\item[-] $\Pic(X)_\QQ=N^1(X)_\QQ$;
\item[-] the nef cone $\Nef(X)$ is generated by finitely many semiample divisor classes;
\item[-] there exist finitely many small $\QQ$-factorial modifications $f_i\colon X\dashrightarrow X_i$ such that each $X_i$ is normal and $\QQ$-factorial, and
$$
        \Mov(X)=\bigcup_i f_i^*\Nef(X_i).
$$
\end{itemize}
Equivalently, $X$ is a Mori dream space if its Cox ring is finitely generated.

Clearly, Fano implies weak Fano which in turn implies log Fano. Moreover, if $X$ and $Y$ are normal and $\mathbb{Q}$-factorial projective varieties which are isomorphic in codimension one then $X$ is log Fano if and only if $Y$ is so \cite{GOST15}. Finally, by \cite[Corollary 1.3.2]{BCHM} if $X$ is log Fano then it is a Mori dream space.

Let $Q\subset\mP^{4}$ be a smooth quadric hypersurface, $\Gamma_1,\Gamma_2\subset Q$ two twisted cubics, $H_1 = \left\langle \Gamma_1 \right\rangle, H_2 = \left\langle \Gamma_2 \right\rangle$ the hyperplanes spanned by them, $Q_1 = Q\cap H_1, Q_2 = Q\cap H_2$ the corresponding hyperplane sections, and $C = Q \cap H_{1}\cap H_{2}$.

\subsubsection{Generality assumptions}\label{gen_ass}
We will assume that 
\begin{itemize}
\item[(i)] $Q_1,Q_2$ are smooth quadric surfaces;
\item[(ii)] $C$ is a smooth conic and $Q_1\cap Q_2 = C$;
\item[(iii)] $Q_2$ intersects $\Gamma_1$ transverally in three distinct points $p_1,p_2,p_3$, $Q_1$ intersects $\Gamma_2$ transverally in three distinct points $q_1,q_2,q_3$, and $\Gamma_1\cap C = \{p_1,p_2,p_3\}$, $\Gamma_2\cap C = \{q_1,q_2,q_3\}$.
\end{itemize}
Note that through $p_i$ there is a unique line $L_i\subset Q_2$ intersecting $\Gamma_2$ in two points, and similarly through $q_i$ there is a unique line $R_i\subset Q_1$ intersecting $\Gamma_1$ in two points. 
\begin{itemize}
\item[(iv)] We will assume that $L_i$ is not tangent to $\Gamma_2$ and $R_i$ is not tangent to $\Gamma_1$ for $i = 1,2,3$.
\end{itemize}

\begin{Lemma}\label{Mori_Cone}
Let $\Theta: W\rightarrow Q$ be the blow-up of $Q$ along $\Gamma_{1},\Gamma_{2}$ with exceptional divisors $E_1,E_2$. Then $E_i = \mathbb{F}_1$ for $i = 1,2$, where $\mathbb{F}_1$ is the first Hirzebruck surface. Furthermore, let $f_i$ be the class of a fiber of $E_i$, $h_i$ the $(-1)$-curve of $E_i$, and $h$ the pull-back of a line in $Q$. Then the Mori cone $\NE(W)$ of $W$ is generated by $h_1,h_2,f_1,f_2,h-2f_1-f_2,h-f_1-2f_2$.
\end{Lemma}
\begin{proof}
Since $N_{\Gamma_i/Q_i} = \mathcal{O}_{Q_i}(1,2)_{|\Gamma_i}$ and $N_{Q_i/Q} = \mathcal{O}_{Q_i}(1,1)_{|\Gamma_i}$ pulling-back the exact sequence
$$0\rightarrow N_{\Gamma_i/Q_i}\rightarrow N_{\Gamma_i/Q}\rightarrow N_{Q_i/Q}\rightarrow 0$$
via the embedding $\nu:\mathbb{P}^1\rightarrow \Gamma_i$ we get 
$$0\rightarrow \mathcal{O}_{\mathbb{P}^1}(4)\rightarrow \nu^{*}N_{\Gamma_i/Q}\rightarrow \mathcal{O}_{\mathbb{P}^1}(3)\rightarrow 0$$
so that $\nu^{*}N_{\Gamma_i/Q} = \mathcal{O}_{\mathbb{P}^1}(4)\oplus \mathcal{O}_{\mathbb{P}^1}(3)$ and hence $E_i = \mathbb{F}_1$ for $i = 1,2$ where $\mathbb{F}_1$ is the first Hirzebruck surface. Moreover, we have $E_{i|E_i} \simeq -h_i+3f_i$.

Now, let $\widetilde{C}\subset W$ be an irreducible curve. If $\widetilde{C}$ is contained in $E_i$ then it is numerically equivalent to a combination with non-negative coefficients of $h_i$ and $f_i$. If  $\widetilde{C}$ is not contained in $E_1\cup E_2$ then $C = \Theta(\widetilde{C})$ is a curve different from $\Gamma_1$ and $\Gamma_2$. We can then write
$$
\widetilde{C} \sim d h - m_1f_1 - m_2 f_2
$$
where $d$ is the degree of $C$ and $m_i$ is the number of intersection points counted with multiplicity of $C$ and $\Gamma_i$. 

If $m_i > d$ then $C\subset Q_i$. Therefore, $C$ is numerically equivalent to a linear combination with non-negative coefficients of the two rulings of $Q_i$, and hence $\widetilde{C}$ is numerically equivalent to a linear combination with non-negative coefficients of $h-2f_i-f_j$ and $h-f_j$. To conclude it is enough to note that $h-f_j\sim (h-f_i-2f_j) + f_i + f_j$. 

Finally, consider the case $m_1,m_2\leq d$. We may assume that $m_1\geq m_2$ and write
$$
\widetilde{C} \sim (d-m_1) h + (m_1-m_2)(h-f_1) + m_2 (h-f_1-f_2).
$$
Since $h-f_1\sim (h-f_1-2f_2) + f_1 + f_2$, $h-f_1-f_2\sim (h-2f_1-f_2) + f_1$ we get the claim.
\end{proof}

\begin{Lemma}\label{Eff_Cone}
Let $\Theta: W\rightarrow Q$ be the blow-up of $Q$ along $\Gamma_{1},\Gamma_{2}$ with exceptional divisors $E_1,E_2$, and denote by $\widetilde{Q}_1,\widetilde{Q}_2$ the strict transforms in $W$ of the quadrics $Q_1,Q_2$. The effective cone $\Eff(W)$ of $W$ is generated by $E_1,E_2,\widetilde{Q}_1,\widetilde{Q}_2$.
\end{Lemma}
\begin{proof}
Let $D\subset W$ be an irreducible effective divisor. If $D$ is contracted by $\Theta$ then is must be either a multiple of $E_1$ or a multiple of $E_2$. If $D$ is not contracted by $\Theta$ then we may write
$$
D \sim dH-m_1E_1-m_2E_2
$$ 
where $H$ is the pullback of the hyperplane section of $\mathbb{P}^4$, and hence $D$ is the strict transform in $W$ of a hypersurface of degree $d$ in $\mathbb{P}^4$ having multiplicity $m_i$ along $\Gamma_i$ for $i = 1,2$. 

The intersection product of $D$ with the strict transform of a secant line of $\Gamma_i$ is given by
$$
D\cdot (h-2f_i) = (dH-m_1E_1-m_2E_2)\cdot (h-2f_i) = d-2m_i.
$$
So, if $d-2m_i < 0$ since $\widetilde{Q}_i$ is covered by curves of class $h-2f_i$ and $D$ is irreducible $D$ must be a multiple of $\widetilde{Q}_i$. Now, assume that $d-2m_i \geq 0$ for $i = 1,2$. Then, since $H\sim \widetilde{Q}_1 + E_1$, we may write
$$
\begin{small}
\begin{array}{lcl}
D & \sim & (\frac{d}{2}-m_1)H + m_1(H-E_1) + \frac{d}{2}H - m_2 E_2 \sim (\frac{d}{2}-m_1)H + m_1(H-E_1) + (\frac{d}{2}-m_2) H +  m_2 (H-E_2)  \\ 
 & \sim & (\frac{d}{2}-m_1)(\widetilde{Q}_1 + E_1) + m_1\widetilde{Q}_1 + (\frac{d}{2}-m_2)(\widetilde{Q}_1 + E_1) +  m_2 \widetilde{Q}_2
\end{array} 
\end{small}
$$ 
concluding the proof. The case $d$ odd follows immediately from the case $d$ even.
\end{proof}

\begin{Proposition}\label{MDS}
Let $\Theta: W\rightarrow Q$ be the blow-up of $Q$ along $\Gamma_{1},\Gamma_{2}$. The strict transforms $\widetilde{L}_i,\widetilde{R}_i$ of the $L_i,R_i$ and $\widetilde{C}$ of $C$ become flopping curves in $W$. Furthermore, $W$ is weak Fano and hence a Mori dream space.
\end{Proposition}
\begin{proof}
We have that $Q_i\simeq\mP^1\times\mP^1$ and $C\in |\mathcal{O}_{Q_i}(1,1)|$ for $i=1,2$. We can assume without loss of generality that $\Gamma_i \in |\mathcal{O}_{Q_i}(1,2)|$ for $i=1,2$.   

Let $\widetilde{L}_i,\widetilde{R}_i$ be the strict transforms of the $L_i,R_i$. We have $L_i\sim h-2f_1-f_2,R_i\sim h-f_1-2f_2\subset Q$ and $-K_{W} = 3H - E_1 - E_2$. Since
$$
-K_{W}\cdot \widetilde{L}_i = -K_{W}\cdot \widetilde{R}_i = -K_{W}\cdot \widetilde{C} = 0
$$ 
we get that $\widetilde{L}_i,\widetilde{R}_i,\widetilde{C}\subset W$ are flopping curves. Note that
$$
\widetilde{C}\sim 2h-3f_1-3f_2 \sim (h-2f_1-f_2) + (h-f_1-2f_2).
$$
Moreover, $-K_{W}\cdot f_i = 1$ and $-K_{W}\cdot h_i = 5$ for $i = 1,2$. Therefore, Lemma \ref{Mori_Cone} yields that $-K_{W}$ is nef. Finally, since $(-K_{W})^3 = 14$ we get that $-K_{W}$ is nef and big. So $W$ is weak Fano and \cite[Corollary 1.3.2]{BCHM} implies that $W$ is a Mori dream space. 
\end{proof}

\begin{thm}\label{prop:otherreconst}
Let $\Theta: W\rightarrow Q$ be the blow-up of $Q$ along $\Gamma_{1}\cup \Gamma_{2}$, and $W\dashrightarrow V$ the flop of $\widetilde{L}_i,\widetilde{R}_i$ and $\widetilde{C}$. The strict transforms $F_i$ of $Q_i$ on $V$ are projective planes which can be contracted to two $\frac{1}{2}(1,1,1)$-singularities:
$$
\begin{tikzcd}
  & W \arrow[ld, "\Theta"'] \arrow[rr, "\Psi", dashed] &  & V \arrow[rd, "\Phi"] &   \\
Q &                                                   &  &                       & X
\end{tikzcd} 
$$
The target of the contraction is a $\mQ$-Fano $3$-fold $X$ such that $[X]\in \sF_{8,2\times \frac{1}{2}(1,1,1)}$. 
\end{thm}
\begin{proof} 
Let $\Psi: W\dasharrow V$ be the composition of the seven flops of $\widetilde{L}_i,\widetilde{R}_i$ for $i = 1,2,3$ and $\widetilde{C}$, and $\widetilde{Q}_i$ the strict transform in $W$ of $Q_i$. Let us analyze the effect of the seven flops on $\widetilde{Q}_i$. First we blow-up the $q_i$ getting a surface of Picard rank five. In such surface the $\widetilde{L}_i$ and $\widetilde{C}$ have self-intersection $-1$ and so they can be blown-down obtaining a minimal rational surface of Picard rank one and with no curve of self-intersection zero. Hence $F_i\cong\mathbb{P}^2$.

We denote by $h_{\mP^2}$ the class of a line. Note that $E_{1\mid Q_1^{W}}$ is $\Gamma_1$ which is a rational curve. The image of such curve in $F_1$ has three double points coming from the contraction of the $\widetilde{L}_i$ and one triple point coming from the contraction of $\widetilde{C}$, and since it is rational we must have that $E^V_{1\mid F_{1}}\simeq 5h_{\mP^2}$ in $V$. 

Moreover, since the $Q_i$ are hyperplane sections of $Q$ containing $\Gamma_i$ we have $H^V\simeq E_i^V+F_i$ for $i=1,2$, and by adjunction it follows that
$$
K_{F_1} \simeq (K_V+F_1)_{\mid F_{1}} \simeq (-3H^V+E_1^V+E_2^V+F_{1})_{\mid F_{1}} \sim (-3H^V + H^V - F_2 + H^V)_{| F_1} \sim -H^V_{|F_1} 
$$
so that $H^V_{|F_1} \sim 3h_{\mP^2}$, and
$$K_{F_1} \simeq (K_V+F_1)_{\mid F_{1}} \simeq (-3H^V+E_1^V+H^V - F_2 +F_{1})_{\mid F_{1}}\simeq  -6h_{\mP^2} + 5h_{\mP^2} + F_{1|F_1}
$$
so that $F_{1|F_1}\sim -2h_{\mP^2}$. Furthermore, $E_{2|F_1}\sim H^V_{|F_1} \sim 3h_{\mP^2}$ yields
$$
(8H^V - 3E_1^V - 3E_2^V)_{|F_1}\sim 24h_{\mP^2} - 15h_{\mP^2} - 9 h_{\mP^2} \sim 0.
$$
Let $\overline{L}_i,\overline{R}_i,\overline{C}\subset V$ be the flopped curves relative to $\widetilde{L}_i,\widetilde{R}_i,\widetilde{C}\subset V$ respectively. Note that 
$$(8H^V - 3E_1^V - 3E_2^V)\cdot \overline{L}_i = (8H^V - 3E_1^V - 3E_2^V)\cdot \overline{R}_i = 1$$ 
and $(8H^V - 3E_1^V - 3E_2^V)\cdot \overline{C} = 2$.
 
Therefore, $8H^{V}-3E_1^V-3E_2^V$ is nef, and since by Proposition \ref{MDS} $V$ is a Mori dream space we get that $8H^{V}-3E_1^V-3E_2^V$ is semi-ample. Furthermore, as we will see $(8H^{V}-3E_1^V-3E_2^V)^3 = 120$, and hence $8H^{V}-3E_1^V-3E_2^V$ is big. Hence, a big enough multiple of $8H^{V}-3E_1^V-3E_2^V$ induces a birational morphism $\Phi:V\rightarrow X$ contracting $F_1, F_2$. 

Since $N_{F_i/V} = \mathcal{O}_{F_i}(-2)$ the points $\Phi(F_1),\Phi(F_2)\in V$ are two $\frac{1}{2}(1,1,1)$-singularities. So $X$ is terminal, $\mathbb Q$-factorial with Picard rank one and locally the contractions of $F_1$ and $F_2$ are weighted blow-ups. By standard formulas for weighted blow-up we have:
 $$
 \Phi^{\star}(-K_X)-\frac{1}{2}F_1-\frac{1}{2}F_2\simeq 3H^{V}-E_1^V-E_2^V
 $$
 that is
 $$
 \Phi^{\star}(-2K_X)\simeq F_1+F_2+6H^{V}-2E_1^V-2E_2^V\simeq 8H^{V}-3E_1^V-3E_2^V.
 $$
From $E_i^2 = -h_i+3f_i$ we get that
$$
E_i^3 = (-h_i+3f_i)^2 = -7,\quad H\cdot E_i^2 = 3f_i\cdot(-h_i+3f_i) = -3
$$
for $i = 1,2$, and hence
$$
(8H^W-3E_1^W-3E_2^W)^3 = 106.
$$
By \cite[Proposition 4.5]{TZ12} it holds that 
$$
106=(8H^W-3E_1^W-3E_2^W)^3=(8H^{V}-3E_1^{V}-3E_2^{V})^3-14.
$$ 
yielding $(8H^{V}-3E_1^{V}-3E_2^{V})^3=120$, that is $(-K_X)^3=15$. By \cite[Theorem 0.3]{Tak02} we have that $X$ is a $\mathbb Q$-Fano $3$-fold with two quotient singularities of type $\frac{1}{2}(1,1,1)$. In particular, $X\hookrightarrow\mathbb P(1^{10}, 2^2)$ where $\sO_{X}(-K_X)=\sO_{X}(1)$.
\end{proof}

\begin{Remark}\label{And}
Since $W$ is the blow-up of $Q$ along two disjoint curves, $\Psi:W\dasharrow V$ is small, and $\Phi:V\rightarrow X$ is the blow-down of two divisors we have that
$$
\rho(X) = \rho(V)-2 = \rho(W)-2 = 1.
$$
Furthermore, $E_i^3 = -7$ and $HE_i^2 = H_{|E_i}E_i = 3f_i(-h_i+3f_i) = -3f_ih_i = -3$ yield
$$
(-K_W)^3 = 27H^3+9HE_1^2+9HE_2^2-E_1^3-E_2^3 = 14.
$$
Hence $(-K_V) = 14$, and \cite[Lemma 2.3]{Tak06} implies that $(-K_X)^3 = 15$.

Since $\chi(\mathcal{O}_X) = 1$ and $\chi(\mathcal{O}_X(m K_X)) = h^0(X,\mathcal{O}_X(m K_X))$ for $m < 0$ \cite[Corollary 10.3]{Re87} yields
$$
h^0(X,\mathcal{O}_X(m K_X)) = \begin{cases}
1-2m-\frac{5}{4}m(m-1)(2m-1)+\frac{m}{4}, & \text{if } m \text{ is even}, \\[2mm]
1-2m-\frac{5}{4}m(m-1)(2m-1)+\frac{m-1}{4}, & \text{if } m \text{ is odd}.
\end{cases}
$$
for $m < 0$. In particular
$$
\begin{array}{l}
h^0(X,\mathcal{O}_X(-K_X)) = 10, h^0(X,\mathcal{O}_X(-2K_X)) = 42, h^0(X,\mathcal{O}_X(-3K_X)) = 111, \\ 
h^0(X,\mathcal{O}_X(-4K_X)) = 233, h^0(X,\mathcal{O}_X(-5K_X)) = 422. 
\end{array}  
$$
\end{Remark}

\subsubsection{The reconstruction theorem}
By \cite[Theorem~1.2]{CHY}, the Hilbert scheme ${\rm Hilb}^Q_3$ of twisted cubics in $Q$ is isomorphic to a projectivized rank $6$ bundle over $Gr(3,5)$.

By Theorem \ref{prop:otherreconst} we can build a rational map 
$$
\begin{array}{cccc}
\kappa: &\mathfrak{S}^{2}{\rm{Hilb}}^{Q}_{3}\quot{\rm{Aut}}(Q)& \dashrightarrow & \sF_{8,2\times \frac{1}{2}(1,1,1)}\\
      & [\Gamma_1,\Gamma_2] & \mapsto & [X(\Gamma_1, \Gamma_2)]
\end{array}
$$
where $X(\Gamma_1, \Gamma_2)$ is the Fano $3$-fold obtained in Theorem \ref{prop:otherreconst} from the twisted cubics $\Gamma_1,\Gamma_2$ and $\mathfrak{S}^{2}{\rm{Hilb}}^{Q}_{3}\quot{\rm{Aut}}(Q)$ is the quotient of the second symmetric product of ${\rm{Hilb}}^{Q}_{3}$ by the natural diagonal action on ${\rm{Hilb}}^{Q}_{3}\times{\rm{Hilb}}^{Q}_{3}$ of the automorphism group of $Q$. Let $\sH\hookrightarrow {\rm{Hilb}}^{Q}_{3}$ be the locus of smooth twisted cubics. We define
$$
 \sC':=\{ (\Gamma_1, \Gamma_2)\in \sH\times\sH \text{ such that } \Gamma_1, \Gamma_2 \text{ satisfy the assumptions in \ref{gen_ass}}\}
$$
and let $\sC$ be its image in $\mathfrak{S}^{2}{\rm{Hilb}}^{Q}_{3}$. Theorem \ref{prop:otherreconst} yields a morphism 
$$
\begin{array}{cccc}
\kappa: &\sC& \rightarrow & \sF_{8,2\times \frac{1}{2}(1,1,1)}\\
      & [\Gamma_1,\Gamma_2] & \mapsto & [X(\Gamma_1, \Gamma_2)].
\end{array}
$$

\begin{Lemma}\label{lem_a}
Let $[\Gamma_1,\Gamma_2]\in \mathfrak{S}^{2}{\rm{Hilb}}^{Q}_{3}$ and define
$$
\Aut(Q,[\Gamma_1,\Gamma_2]) = \{\varphi\in \Aut(Q)\: | \: \varphi(\Gamma_i) = \Gamma_j,\: i,j\in\{1,2\}\}.
$$
If $[\Gamma_1,\Gamma_2]\in \mathfrak{S}^{2}{\rm{Hilb}}^{Q}_{3}$ is general then $\Aut(Q,[\Gamma_1,\Gamma_2])$ is trivial.
\end{Lemma}
\begin{proof}
Recall that any automorphism $\varphi:Q\rightarrow Q$ of $Q\subset\mathbb{P}^4$ is induced by an automorphism of the ambient projective space that we will keep denoting by $\varphi:\mathbb{P}^4\rightarrow\mathbb{P}^4$. 

Now, let $\varphi\in \Aut(Q,[\Gamma_1,\Gamma_2])$. Then $\varphi(H_i) = H_j$ for $i,j\in\{1,2\}$. Hence $\varphi(\Pi) = \Pi$ where $\Pi = \left\langle C\right\rangle$ and $\varphi(C) = C$. 

Therefore, $\varphi$ restricts to an automorphism of $\Pi\cong\mathbb{P}^2$ stabilizing $C\cong\mathbb{P}^1$. Let $\Aut(\Pi,C)\subset\Aut(\Pi)$ be the subgroup of the automorphisms of $\Pi$ stabilizing $C$. We have a restriction map $\Aut(\Pi,C)\rightarrow\Aut(C)\cong PGL(2)$ which is surjective since via the Veronese embedding $\mathbb{P}^1\rightarrow C\subset\Pi$ any automorphism of $\mathbb{P}^1$ yields an automorphism of $\Pi$. Now, if an automorphism in $\Aut(\Pi,C)$ restricts to the identity on $C$ then it fixes five points of $\Pi$ in general position and hence it is the identity. Therefore, $\Aut(\Pi,C)\cong PGL(2)$. 

Assume that $\varphi(\{p_1,p_2,p_3\}) = \{q_1,q_2,q_3\}$ and $\varphi(\{q_1,q_2,q_3\}) = \{p_1,p_2,p_3\}$. Then $\varphi$ yields and automorphism of $\mathbb{P}^1$ switching two sets of three general points which does not exist. So $\varphi(\{p_1,p_2,p_3\}) = \{p_1,p_2,p_3\}$ and $\varphi(\{q_1,q_2,q_3\}) = \{q_1,q_2,q_3\}$. On the other hand, there is a unique automorphism of $\mathbb{P}^1$ mapping $\{p_1,p_2,p_3\}$ to itself and since the $q_i$ are general such automorphism will not stabilize $\{q_1,q_2,q_3\}$ unless it is the identity. Hence, $\varphi(p_i) = p_i$ and $\varphi(q_i) = q_i$ for $i = 1,2,3$.

Now, both $\Gamma_1,\Gamma_2$ are isomorphic to $\mathbb{P}^1$ and $\varphi$ fixes three general points on each of them. So $\varphi$ restricts to the identity on $\Gamma_i$ for $i = 1,2$. Taking five general points on $\Gamma_i$ we see that $\varphi$ restricts to the identity on $H_i$ for $i = 1,2$.

Finally, taking four points in $H_1$ and two points in $H_2$ such that the six points are in general position we conclude that $\varphi$ must be the identity.    
\end{proof}

\begin{Proposition}\label{prop_map}
The rational map
$$
\begin{array}{cccc}
\kappa: &\mathfrak{S}^{2}{\rm{Hilb}}^{Q}_{3}\quot{\rm{Aut}}(Q)& \dashrightarrow & \sF_{8,2\times \frac{1}{2}(1,1,1)}\\
      & [\Gamma_1,\Gamma_2] & \mapsto & [X(\Gamma_1, \Gamma_2)]
\end{array}
$$
is dominant and generically injective.
\end{Proposition}
\begin{proof}
Let $\mathcal{U}'\subset\mathcal{C}'$ be the locus of pairs $(\Gamma_1,\Gamma_2)$ such that $\Aut(Q,[\Gamma_1,\Gamma_2])$ is trivial, and $\mathcal{U}\subset\mathcal{C}$ its image in $\mathfrak{S}^{2}{\rm{Hilb}}^{Q}_{3}$. By Lemma \ref{lem_a} we have that $\mathcal{U}\neq\emptyset$, and Theorem \ref{prop:otherreconst} yields a morphism 
$$
\begin{array}{cccc}
\kappa: &\mathcal{U}\quot{\rm{Aut}}(Q)& \rightarrow & \sF_{8,2\times \frac{1}{2}(1,1,1)}\\
      & [\Gamma_1,\Gamma_2] & \mapsto & [X(\Gamma_1, \Gamma_2)].
\end{array}
$$
Now, let $[\Gamma_1,\Gamma_2]\in \mathcal{U}$ be a general point and assume that there is $[\Gamma_1,\Gamma_2]\in \mathcal{U}$ such that $[X(\Gamma_1, \Gamma_2)] = [X(\Gamma_1', \Gamma_2')]$. Then there is an isomorphism $\rho:X(\Gamma_1, \Gamma_2)\rightarrow X(\Gamma_1', \Gamma_2')$. First note that $\rho$ must map the two singular points of  $X(\Gamma_1, \Gamma_2)$ to those of $X(\Gamma_1', \Gamma_2')$, and hence $\rho$ lifts to an isomorphism $\overline{\rho}:V\rightarrow V'$ which in turn yields a small birational map $\overline{\rho}:W\dashrightarrow W'$.

Now, $\overline{\rho}$ maps the extremal rays of $\Eff(W)$ to those of $\Eff(W')$, see for instance \cite[Lemma 7.1]{Mas20}. Note that by Lemma \ref{Mori_Cone} and the proof of Theorem \ref{prop:otherreconst} all the curves in the indeterminacy locus of $\overline{\rho}$ are contained in $\widetilde{Q}_1\cup \widetilde{Q}_2$. Hence, by Lemma \ref{Eff_Cone} $\overline{\rho}$ maps $\widetilde{Q}_1\cup \widetilde{Q}_2$ to $\widetilde{Q}_1'\cup \widetilde{Q}_2'$ and $E_1\cup E_2$ to $E_1'\cup E_2'$. Therefore, $\overline{\rho}$ descends to a small birational map $\widehat{\rho}:Q\dasharrow Q'$ mapping $\Gamma_1\cup\Gamma_2$ to $\Gamma_1'\cup\Gamma_2'$. Now, it is enough to note that since $Q$ and $Q'$ have Picard number one $\widehat{\rho}$ extends to a biregular isomorphism mapping $\Gamma_1\cup\Gamma_2$ to $\Gamma_1'\cup\Gamma_2'$, see for instance \cite[Proposition 7.2]{Mas20}. Finally, since $\dim(\mathfrak{S}^{2}{\rm{Hilb}}^{Q}_{3}\quot{\rm{Aut}}(Q)) = \dim(\sF_{8,2\times \frac{1}{2}(1,1,1)}) = 8$ we conclude that $\kappa$ is dominant. 
\end{proof}

\section{The rationality of $\sF_{8,2\times \frac{1}{2}(1,1,1)}$}\label{Sec3}
We recall the slice method. Let $G$ be a group acting on a variety $X$, and $H\subset G$ a subgroup of $G$.
\begin{Definition}
A $(G; H)$-section of $X$ is an irreducible subvariety
$Y\hookrightarrow X$ such that 
\begin{itemize}
\item[-] for every $h\in H$ it holds that $h(Y)\hookrightarrow Y$;
\item[-] $\bigcup_{g\in G}gY$ is dense in $X$; 
\item[-] for every $g\in G$ such that there exist $y_1,y_2\in Y$ with $y_1\neq y_2$ and $g(y_1)=y_2$ we have $g\in H$.
\end{itemize}
\end{Definition}
If $Y$ is a $(G;H)$-section then by restricting rational
functions on $X$ to $Y$ we obtain an isomorphism of invariant 
functions; that is $Y\quot H$ is birational to $X\quot G$. We sum up the above results into the following:

\begin{Lemma}\label{slicemethod} 
Let $Y\hookrightarrow X$ be a $(G;H)$-section of a variety $X$ with an action of $G$. Then $Y\quot H$ is birational to $X\quot G$. In particular, if $Y\quot H$ is rational so is $X\quot G$.
\end{Lemma}

For any $S$-scheme $X\to S$ denote by $X^{m}:=X\times_{S}\cdots \times_{S} X$ the $m$-fiber product over $S$. Let $X^{m}/\mathfrak{S}_m\to S$ be the relative $m$-symmetric $S$-product scheme and set $\mathfrak{S}^{m}_{S}X:=X^{m}/\mathfrak{S}_m$. If no danger of confusion arises we omit the subscript $S$ in the notation $\mathfrak{S}^{m}_{S}X$.

By Corollary \ref{charactofHilb2} $\mathfrak{S}^{2}\mathfrak{S}^{3}\sU_{\sH^{Q}_{2}}$ is irreducible and by definition ${\rm{Aut}}(Q)$ acts equivariantly on $\mathfrak{S}^{2}\mathfrak{S}^{3}\sU_{\sH^{Q}_{2}}\to\sH^{Q}_{2}$. Indeed, a general point $x\in\mathfrak{S}^{2}\mathfrak{S}^{3}\sU_{\sH^{Q}_{2}}$ is 
given by an unordered couple of unordered triples, that is we can write $x=[[a_{1},a_{2},a_{3}], [b_{1},b_{2},b_{3}]]$, where 
$a_{i},b_{j}\in q$, $i,j=1,2,3$ and $[q]\in\sH^{Q}_{2}$. 

\begin{Proposition}\label{firstforrat} 
There exists an ${\rm{Aut}}(Q)$-invariant open subset $W\hookrightarrow \mathfrak{S}^{2}\mathfrak{S}^{3}\sU_{\sH^{Q}_{2}}$ on which ${\rm{Aut}}(Q)$ acts geometrically and $W\quot{\rm{Aut}}(Q)$ is birational to $\mP^{3}$.
\end{Proposition}
\begin{proof} 
We apply the $(G,H)$-section method. We define $W$ as the open subset such that for every $[q]\in \sH^{Q}_{2}$ the fiber of $W\to\sH^{Q}_{2}$ over $[q]$ is given by the scheme 
$$
W_{q}:=\{[[(a_{1},a_{2},a_{3})], [(b_{1},b_{2},b_{3})]]\in\mathfrak{S}^{2}\mathfrak{S}^{3}q\mid\, a_{i}\neq a_{j},b_{i}\neq b_{j},\, i\neq j, i,j=1,2,3\}.
$$ 
By Corollary \ref{charactofHilb2} ${\rm{Aut}}(Q)$ acts transitively on $\sH^{Q}_{2}$, and the action commutes with the two symmetric actions. Indeed, it is the diagonal action on the corresponding ordered sets, and Lemma \ref{slicemethod} yields that $W\quot {\rm{Aut}}(Q)$ is birational to $W_{q}\quot SO(3,\mathbb C)$. Clearly, $W_{q}\quot SO(3,\mathbb C)$ is birational to $\mathfrak{S}^{2}\mathfrak{S}^{3}\mP^{1}\quot PGL(2)$ which in turn is birational to $\mP^{3}$.
\end{proof}

Let $q\subset Q$ be a general conic, and $L_{q}\subset \check\mP^{4}$ the pencil of hyperplanes containing $q$: 
$$L_{q}:=\{[H]\in \check\mP^{4}\: |\: q\subset H\}.$$
Note that $L_{q}$ is the pencil 
of hyperplanes containing the plane generated by $q$. Set 
$$
{\check{\sQ}} = \{(H,R) \: | \: R \text{ is a ruling of } Q\cap H\}
$$
and denote by $\rm{rul}:{\check{\sQ}}\rightarrow\check\mP^{4}$ the double covering, branched over the dual of $Q$, associating to a ruling of $Q\cap H$ the hyperplane $H$. Since $q$ is general the inverse image 
${\check{\sQ}}_{q}:={\rm{rul}}^{-1}L_{q}$ of the pencil $L_{q}$ is still a projective line and the induced morphism ${\rm{rul}}_{q}\colon{\check{\sQ}}_{q}\to L_{q}$ is branched over the two points given by $L_{q}\cap\check{Q}$. Considering the base 
change induced by the inclusion 
${\check{\sQ}}_{q}\hookrightarrow{\check{\sQ}}$ we get a 
$\mP^{5}$-bundle ${\overline\pi_{ {\check{\sQ}}_{q}}}\colon{{\overline{\sH}_{3,q}}\to {\check{\sQ}}_{q_{}}}$ 
where $\sH^{}_{3,q_{}} = {\overline\pi_{{\check{\sQ}}}}^{-1}({\check{\sQ}}_{ q_{} }) \cap \sH^{Q}_{3}$. 
Furthermore, we have the following diagram
\stepcounter{thm}
\begin{equation}\label{eq:2prod}
\begin{tikzcd}
\mathfrak{S}^{2}(\sH^{Q}_{3}\setminus g^{-1}{\check{Q}}) \setminus\Delta_{\sH^{Q}_{3}} \arrow[d, "\mathfrak{S}^{2}g"'] \arrow[rr] &  & \mathfrak{S}^{2}\mathfrak{S}^{3}     \sU_{\sH^{Q}_{2}} \arrow[d, "\mathfrak{S}^{2}\mathfrak{S}^{3}\pi_{\sH^{Q}_{2}}"] \\
\mathfrak{S}^{2}(\check\mP^{4}\setminus  {\check{Q}})\setminus\Delta_{\check\mP^{4}} \arrow[rr]                                   &  & \sH^{Q}_{2}                                                                                                          
\end{tikzcd}
\end{equation}
where 
$\iota\colon\mathfrak{S}^{2}(\check\mP^{4}\setminus 
{\check{Q}})\setminus\Delta_{\check\mP^{4}}\to\sH^{Q}_{2}$ is the morphism that to an unordered couple $[([H_{1}],[H_{2}])]$ of hyperplanes of $\mP^4$ associates the conic $q=H_{1}\cap H_{2}\cap Q$. The morphism 
$\nu\colon \mathfrak{S}^{2}(\sH^{Q}_{3}\setminus g^{-1}{\check{Q}})\setminus \Delta_{\sH^{Q}_{3}}\to\mathfrak{S}^{2}\mathfrak{S}^{3}\sU_{\sH^{Q}_{2}}$ is definded by 
$[(\Gamma_{1},\Gamma_{2})]\to [([\Gamma_{1}\cap 
q],[\Gamma_{2}\cap q])]$. Note that if 
$[([H_{1}],[H_{2}])]\in \mathfrak{S}^{2}(\check\mP^{4}\setminus 
{\check{Q}})\setminus\Delta_{\check\mP^{4}}$ then 
$(\mathfrak{S}^{2}g)^{-1}([([H_{1}],[H_{2}])])$ is birational to the disjoint union of four copies of $\mP^{5}\times\mP^{5}$. 

The closure in $\mathfrak{S}^{2}\check\mP^{4}$ of the fiber $\iota^{-1}[q]$ is a projective plane. Indeed, it is isomorphic to the quotient of $\mP^1\times\mP^1$ by the automorphism exchanging the two copies of $\mathbb{P}^1$. 

\begin{Proposition}\label{2symmetric} 
The quotient $\mathfrak{S}^{2}\sH^{Q}_{3}\quot{\rm{Aut}}(Q)$ is a rational variety.
\end{Proposition}
\begin{proof}
In this proof we will denote with $H$ both the hyperplane section and the subgroups used in the slice method, what we are talking about every time will be clear from the context. The diagramm (\ref{eq:2prod}) is ${\rm{Aut}}(Q)$-equivariant with respect to the induced diagonal action on the fibers. Take a point $[q]\in \sH^{Q}_{2}$. We will apply the slice method to the pair $(G;H)=({\rm{Aut}}(Q);{\rm{Aut}}(Q,q))$. By Corollary \ref{charactofHilb2}, ${\rm{Aut}}(Q)$ acts transitively on $\sH^{Q}_{2}$. Restricting the diagram (\ref{eq:2prod}) over the point $[q]$ and considering the Stein factorization of $\mathfrak{S}^{2}(\sH^{Q}_{3}\setminus g^{-1}{\check{Q}})\setminus\Delta_{\sH^{Q}_{3}}\to \mathfrak{S}^{2}(\check\mP^{4}\setminus 
{\check{Q}})\setminus\Delta_{\check\mP^{4}}$ we get the following diagram   
\stepcounter{thm}
\begin{equation}\label{eq:2bissymdiag}
\begin{tikzcd}
(\mathfrak{S}^{2}\sH_{3,q})\setminus\Delta_{\sH_{3,q}} \arrow[rr] \arrow[d, "{(\mathfrak{S}^{2}{\overline\pi_{{\check{\sQ}}_{q}}}) _{|(\mathfrak{S}^{2}\sH_{3,q})\setminus\Delta_{\sH_{3,q}}}}"'] &  & \mathfrak{S}^{2}\mathfrak{S}^{3}     q \arrow[d, "(\mathfrak{S}^{2}\mathfrak{S}^{3}\pi_{\sH^{Q}_{2}})_{|\mathfrak{S}^{2}\mathfrak{S}^{3}q}"] \\
(\mathfrak{S}^{2}{\check{\sQ}}_{q})\setminus\Delta_{{\check{\sQ}}_{q}} \arrow[rr]                                                                                                                              &  & {[q]}                                                                                                                                            
\end{tikzcd}
\end{equation}
and by the slice method we are left to show that  
$((\mathfrak{S}^{2}\sH^{}_{3,q})\setminus\Delta_{\sH^{}_{3,q}})/{\rm{Aut}}(Q,q)$
is a rational variety. By costruction ${\check{\sQ}}_{q}\to L_{q}$ is 
given, in a suitable coordinate system, by $z\mapsto z^{2}$. Then 
$\mathfrak{S}^{2}{\check{\sQ}}_{q}$ is a projective plane that 
we denote by $\mP^{2}_{q}$. By Corollary \ref{charactofHilb2} we have ${\rm{Aut}}(Q,q)=SO(3,\mathbb C)\times SO(2,\mathbb C)$. Consider the 
morphism $\pi_{q}\colon\sH^{}_{3,q}\to{\check{\sQ}}_{q}$ and the diagram 
\stepcounter{thm}
\begin{equation}\label{eq:2bisproddiag}
\begin{tikzcd}
{\sH^{}_{3,q}\times \sH^{}_{3,q}} \arrow[rr] \arrow[d] &  & {\mathfrak{S}^{2} \sH_{3,q}} \arrow[d] \\
{\check{\sQ}}_{q}\times {\check{\sQ}}_{q} \arrow[rr]   &  & \mathfrak{S}^{2}{\check{\sQ}}_{q}     
\end{tikzcd} 
\end{equation}
We apply the slice method once more to the pair $(G;H) = ( {\rm{Aut}}(Q,q);SO(3,\mathbb C))$. Note that $H$ acts as the identity on the bottom row of diagram (\ref{eq:2bisproddiag}) since $H$ does not exchange the rulings.

Take the image $C$ in $\mathfrak{S}^{2}{\check{\sQ}}_{q}$ of the curve $C_{x}:=\{x\}\times{\check{\sQ}}_{q}$, where $x\in{\check{\sQ}}_{q}$ is a general point. Note that the translates of $C$ by $G$ cover a dense subset of $\mathfrak{S}^{2}{\check{\sQ}}_{q}$. Consider 
$Y\subset 
(\mathfrak{S}^{2}\sH^{}_{3,q})\setminus\Delta_{\sH^{}_{3,q}}$ that is the inverse image of $C\cap (\mathfrak{S}^{2}{\check{\sQ}}_{q})\setminus\Delta_{{\check{\sQ}}_{q}}$. In other words, set $Q_H:=Q\cap H$
$$
Y\simeq \{(\Gamma,\Gamma')\in \sH^{}_{3,q}\times {\sH^{}_{3,q}}\mid |(1,2)|_{Q_H}\ni \Gamma,Q\cap\Gamma\cap\Gamma'=q\}.
$$
By construction $Y$ is $H$-invariant, its $G$ 
translates cover a dense subset of $\mathfrak{S}^{2}\sH^{}_{3,q}$ of dimension $11$. By the slice method we are reduced to study the $SO(3,\mathbb C)$-action on $Y$.

We stress that the $SO(3,\mathbb C)$-action is trivial on the bottom side of 
the diagramm (\ref{eq:2bissymdiag}) and by Proposition \ref{firstforrat} $\mathfrak{S}^{2}\mathfrak{S}^{3}q\quot SO(3,\mathbb C)$ is birational to $\mP^3$. Considering the following diagram
\stepcounter{thm}
\begin{equation}\label{eq:2bisproddiagsimp}
\begin{tikzcd}
Y \arrow[rr] \arrow[d] &  & \mathfrak{S}^{2}\mathfrak{S}^{3}q \arrow[d] \\
C \arrow[rr]           &  & q                                          
\end{tikzcd}
\end{equation}
we see that if we can show that $Y\quot SO(3,\mathbb C)\to C\times \mathfrak{S}^{2}\mathfrak{S}^{3}q\quot SO(3,\mathbb C)$ is 
birational to a $\mP^{2}\times\mP^{2}$-bundle 
over $C\times \mathfrak{S}^{2}\mathfrak{S}^{3}q\quot SO(3,\mathbb C)$, then
$Y\quot SO(3,\mathbb C)$ is birational to $\mP^8$ since $C$ is a rational curve. The choice of $x\in 
{\check{\sQ}}_{q}$ yields that in $\sH_{3,q}$ we 
have selected a $\mP^{5}$ which is a complete linear system of twisted 
cubics in $Q\cap H$, where $x\mapsto[H]$ by the morphism 
${\check{\sQ}}_{q}\to L_q\subset \check\mP^{4}$. Let $Y_{x}\to C_{x}$ be the 
morphism obtained restricting the diagram (\ref{eq:2bisproddiag}) 
along the morphism $C_{x}\to C$. Then $Y_{x}\to C_{x}$ corresponds to $\mP^5\times\sH_{q,3}\to C_{x}$.

Let $Z_{0}$, $Z_{1}$ be two general points of 
$\mathfrak{S}^{3}q$, and $\mP_{Z_{0}}^{2}$ the $2$-dimensional 
sublinear system of $\mP^{5}$ given by the twisted cubics contained in $Q\cap H$, that is $[\Gamma]\in\mP^{5}$, such that 
$\Gamma_{|q} = Z_{0}$. Here we use again that $x$ determines $H$ and a ruling of $Q\cap H$. Let $[\Gamma_{0}]\in\mP_{Z_{0}}^{2}$ be a general element, and $G_{\Gamma_{0}}=\{g\in 
SO(3,\mathbb C)\mid g(\Gamma_{0})=\Gamma_{0}\}$. Then 
$G_{\Gamma_{0}}$ is isomorphic to $G_{Z_{0}}\simeq \mathfrak{S}_{3}$, actually to the subgroup of ${\rm{Aut}}(q)$ which stabilizes the 
support of $Z_{0}$. Hence the morphism 
$$
\begin{array}{ccc}
 \Phi_{Z_{0}}\colon 
SO(3,\mathbb C)\times\mP_{Z_{0}}^{2} & \rightarrow & \mP^{5}\\
 (g,[\Gamma]) & \mapsto & [g(\Gamma)]
\end{array}
$$
is dominant. Now, take $x'\neq x$, $x'\in {\check{\sQ}}_{q}$. Over the general point $(x,x')\in C_{x}$ the fiber of $Y_{x}\to C_{x}$ is isomorphic to $\mP^{5}\times\mP^{5'}$, where $\mP^{5'}$ is a complete linear system of twisted cubics in $Q\cap H'$ and $x'\mapsto[H']$ via the morphism ${\check{\sQ}}_{q}\to\check\mP^{4}$. Let $\mP_{Z_{1}}^{2}$ be the $2$-dimensional sublinear system of 
$\mP^{5'}$ of the twisted cubics $\Gamma'$ such that 
$\Gamma'_{|q} = Z_{1}$. Consider the morphism 
$$
\begin{array}{ccc}
 \Phi\colon  SO(2,\mathbb C)\times 
SO(3,\mathbb C)\times SO(3,\mathbb 
C)\times \mP_{Z_{0}}^{2}\times\mP_{Z_{1}}^{2} & \rightarrow & Y_{x}\\
 (h,g_{0},g_{1}, [\Gamma_{0}],[\Gamma_{1}]) & \mapsto & ([g_{0}(\Gamma_{0})], 
[hg_{1}(\Gamma_{1})])
\end{array}
$$
where $hg_{1}(\Gamma_{1})$ is the twisted 
cubic over the point $h(x')$ obtained from $[g_{1}(\Gamma_{1})]\in \mP^{5'}$. We have that $hg_{1}(\Gamma_{1})=g_{1}h(\Gamma_{1})$. By construction $\Phi$ is surjective and if 
$$
\Phi(h,g_{0},g_{1}, [\Gamma_{0}],[\Gamma_{1}])=\Phi(h',g'_{0},g'_{1}, [\Gamma'_{0}],[\Gamma'_{1}])
$$
then $h=h'$ since $SO(2,\mathbb C)$ acts with trivial stabilizer on ${\check{\sQ}}_{q}$. Then $g_{0}^{-1}\circ g'_{0}\in G_{Z_0}\simeq \mathfrak{S}_{3}$ and $g_{1}^{-1}\circ g'_{1}\in G_{Z_1}\simeq \mathfrak{S}_{3}$ since $g_{0}^{-1}\circ g'_{0}$ permutes $Z_0\hookrightarrow q$ and similarly for $g_{1}^{-1}\circ g'_{1}$. This means that for the general element $([\Gamma],[\Gamma'])\in Y_{x}$ we have
$$
\Phi^{-1}([\Gamma],[\Gamma'])\simeq \times G_{Z_0}\cdot \{[\Gamma_0]\}\times 
G_{Z_{1}}\cdot \{[\Gamma_1]\}.
$$ 
Note that if we fix three general points of $q$ then we can define a representation $SO(3,\mathbb C)\to {\rm{Aut}}(\mathfrak{S}^{3}q)$ whose kernel is isomorphic to $\mathfrak{S}_{3}$. We have that
$\mathfrak{S}^{3}q$ is birational to $SO(3,\mathbb C)/\mathfrak{S}_{3}$. Consider the diagonal action on $\mathfrak{S}_{3}\times\mathfrak{S}_{3}$ over $(SO(3,\mathbb C)\times \mP_{Z_{0}}^{2})\times (SO(3,\mathbb C)\times \mP_{Z_{1}}^{2})$. Then $Y_{x}$ is birational to $SO(2,\mathbb C)\times \mathfrak{S}^{3}q\times\mathfrak{S}^{3}q\times\mP_{Z_{0}}^{2}\times\mP_{Z_{1}}^{2}$. 

Furthermore, as an $SO(3,\mathbb C)$-variety we see that $(Y_{x},SO(3,\mathbb C))$ is birational to $(SO(2,\mathbb 
C)\times \mathfrak{S}^{3}q\times\mathfrak{S}^{3}q\times
\mP_{Z_{0}}^{2}\times\mP_{Z_{1}}^{2},SO(3,\mathbb C))$ where 
$SO(3,\mathbb C)$ acts trivially on $SO(2,\mathbb C)\times
\mP_{Z_{0}}^{2}\times\mP_{Z_{1}}^{2}$ and diagonally on 
$\mathfrak{S}^{3}q\times\mathfrak{S}^{3}q$. Then $Y_{x}/SO(3,\mathbb C)$ is birational to $SO(2,\mathbb C)\times 
\left((\mathfrak{S}^{3}q\times\mathfrak{S}^{3}q)/SO(3,\mathbb C)\right)\times
\mP_{Z_{0}}^{2}\times\mP_{Z_{1}}^{2}$. 

Since the $\mathfrak{S}_{2}$-action on the diagram \ref{eq:2bisproddiag} commutes with the diagonal $SO(3,\mathbb C)$-action it follows that $Y\to C$ is birational to $SO(2,\mathbb C)\times 
(\mathfrak{S}_{2}\mathfrak{S}^{3}q/SO(3,\mathbb C)\times
\mP_{Z_{0}}^{2}\times\mP_{Z_{1}}^{2}$. Hence, the claim follows since $\mathfrak{S}_{2}\mathfrak{S}^{3}q/SO(3,\mathbb C)$ is rational.
\end{proof}

\begin{Corollary}\label{rationality} The moduli space $\sF_{8,2\times \frac{1}{2}(1,1,1)}$ of Fano $3$-folds of genus $8$ with two singular points of type $\frac{1}{2}(1,1,1)$ is rational.
\end{Corollary}
\begin{proof}
The claim follows from Propositions \ref{prop_map} and \ref{2symmetric}. 
\end{proof}

\section{K-stability: birational setup and the invariant curve}\label{Sec4}

The preceding sections establish the rationality of the moduli space by using an explicit construction from pairs of twisted cubics on a smooth quadric threefold. We now use the same construction to study K-stability. The first part of the argument is carried out for triples satisfying the general assumptions below; later we specialize to a highly symmetric member with a fixed-point-free $D_6$-action.

For the reader's convenience, we briefly recall the main hypotheses and constructions from the previous sections. The initial geometric setup was introduced in the previous sections, where we considered pairs of twisted cubics $\Gamma_1,\Gamma_2$ contained in a smooth quadric threefold $Q\subset \PP^4$ and imposed the transversality and secant-line conditions needed for the birational construction. The construction of the threefold $X$ was then carried out by blowing up $\Gamma_1\sqcup\Gamma_2$, flopping the seven distinguished curves, and contracting the two resulting planes. In the following assumption and construction we collect this data in the precise form needed for the K-stability argument.

\begin{Assumption}\label{ass:general-setting}
Let $Q \subset \PP^4$ be a smooth quadric threefold, and let $\Gamma_1,\Gamma_2 \subset Q$ be two smooth twisted cubics such that the following conditions hold:
\begin{enumerate}[(i)]
\item if $H_i := \Span{\Gamma_i} \subset \PP^4$, then $Q_i := Q \cap H_i$ is a smooth quadric surface for $i=1,2$;
\item the intersection $C := Q_1 \cap Q_2$ is a smooth conic;
\item $Q_2$ meets $\Gamma_1$ transversely in three distinct points $p_1,p_2,p_3$, $Q_1$ meets $\Gamma_2$ transversely in three distinct points $q_1,q_2,q_3$, and $\Gamma_1 \cap C = \{p_1,p_2,p_3\}$, $\Gamma_2 \cap C = \{q_1,q_2,q_3\}$;
\item for each $i=1,2,3$ there exists a unique line $L_i \subset Q_2$ through $p_i$ meeting $\Gamma_2$ in two points, and a unique line $R_i \subset Q_1$ through $q_i$ meeting $\Gamma_1$ in two points; moreover $L_i$ is not tangent to $\Gamma_2$ and $R_i$ is not tangent to $\Gamma_1$;
\item the curves $\Gamma_1$ and $\Gamma_2$ are disjoint.
\end{enumerate}
\end{Assumption}

\begin{Construction}\label{constr:WVX}
Let $\Theta \colon W := \Bl_{\Gamma_1 \sqcup \Gamma_2}(Q) \to Q$
be the blow-up of $Q$ along $\Gamma_1 \sqcup \Gamma_2$. We denote by $E_1^W$ and $E_2^W$ the exceptional divisors, and by $\widetilde C$, $\widetilde L_i$, $\widetilde R_i$ the strict transforms of $C$, $L_i$, and $R_i$. The seven curves $\widetilde C$, $\widetilde L_1,\widetilde L_2,\widetilde L_3$, $\widetilde R_1,\widetilde R_2,\widetilde R_3$ are flopping curves. Let $\psi \colon W \dashrightarrow V$
be the corresponding sequence of flops. If $F_1,F_2 \subset V$ denote the strict transforms of $Q_1,Q_2$, then $F_1 \simeq \PP^2$, $F_2 \simeq \PP^2$, and there is a divisorial contraction
$\phi \colon V \to X$
contracting $F_1$ and $F_2$ to two terminal points $x_1,x_2 \in X$ of type $\frac12(1,1,1)$. Moreover $\phi^*(-2K_X) \sim 8H_V - 3E_1^V - 3E_2^V$, where $H_V$, $E_1^V$, and $E_2^V$ denote the strict transforms on $V$ of the hyperplane class and the two exceptional divisors on $W$.
\end{Construction}

We now recall the part of the construction that is needed for the K-stability argument. The main invariant center will be the curve $l\subset X$ obtained as the image of the flopped transform $C_V$ of the conic $C=Q_1\cap Q_2$. We study this curve by blowing it up and comparing the resulting model with a second weak Fano model obtained by changing the order of the elementary birational operations. This comparison gives the volume functions needed for the valuative criterion and for the Abban--Zhuang adjunction argument.

\begin{Construction}\label{constr:XlZ}
Let $C_V \subset V$ be the flopped transform of $\widetilde C$. Then $C_V$ is not contracted by $\phi$, it meets each $F_i$ transversely in one smooth point, and $(8H_V - 3E_1^V - 3E_2^V)\cdot C_V = 2$. We set $l := \phi(C_V) \subset X$, and we let
$f \colon Z := \Bl_l(X) \to X$
be the schematic blow-up of $X$ along the ideal sheaf $\I_l$. We denote by $E \subset Z$ the exceptional divisor.
\end{Construction}

\begin{Lemma}\label{lem:l-anticanonical-line}
We have $(-K_X)\cdot l=1$.
\end{Lemma}

\begin{proof}
By Construction~\ref{constr:XlZ}, the curve $C_V$ is not contracted by $\phi$, and its image is $l=\phi(C_V)\subset X$. Since $C_V$ meets each of $F_1$ and $F_2$ transversely in one smooth point, the restriction $\phi|_{C_V}\colon C_V\to l$ is birational. Thus $\phi_*[C_V]=[l]$.

Since $-2K_X$ is Cartier, the projection formula gives $\phi^*(-2K_X)\cdot C_V=(-2K_X)\cdot \phi_*C_V=(-2K_X)\cdot l$. On the other hand, by Construction~\ref{constr:WVX}, we have $\phi^*(-2K_X)\sim 8H_V-3E_1^V-3E_2^V$, and by Construction~\ref{constr:XlZ}, $(8H_V-3E_1^V-3E_2^V)\cdot C_V=2$. Hence $(-2K_X)\cdot l=2$, and therefore $(-K_X)\cdot l=1$.
\end{proof}

\begin{Lemma}\label{lem:canonical-class-Z}
We have $K_Z=f^*K_X+E$.
\end{Lemma}

\begin{proof}
Let $U:=X\setminus\{x_1,x_2\}$. By Construction~\ref{constr:WVX}, the only singular points of $X$ are $x_1$ and $x_2$, hence $U$ is smooth. Moreover, by Construction~\ref{constr:XlZ}, the curve $l$ is the image of $C_V$ and passes through $x_1$ and $x_2$, so $l\cap U=l\setminus\{x_1,x_2\}$ is smooth. Thus $f_U\colon Z_U:=f^{-1}(U)\to U$ is the blow-up of a smooth threefold along a smooth curve.

For the blow-up of a smooth threefold along a smooth curve, the canonical divisor formula gives $K_{Z_U}=f_U^*K_U+E|_{Z_U}$. Since $K_U=K_X|_U$, this gives $K_Z|_{Z_U}=(f^*K_X+E)|_{Z_U}$.

It remains to extend this equality across $f^{-1}(\{x_1,x_2\})$. The local toric description of the blow-up of the germ $\frac12(1,1,1)$ along the image of $C_V$, written explicitly in Lemma~\ref{lem:singularities-Z}, shows that $Z$ is normal near $f^{-1}(x_1)\cup f^{-1}(x_2)$, that each fiber $f^{-1}(x_i)$ is one-dimensional, and that the singularities of $Z$ are quotient singularities. In particular $Z$ is $\QQ$-Gorenstein, so $K_Z$ is $\QQ$-Cartier. The same local model also shows that $E$ is $\QQ$-Cartier near $f^{-1}(x_i)$; away from these fibers it is the exceptional divisor of the blow-up of a smooth curve in a smooth threefold, hence Cartier. Thus both $K_Z$ and $f^*K_X+E$ are $\QQ$-Cartier divisors on the normal variety $Z$. Since $Z\setminus Z_U=f^{-1}(\{x_1,x_2\})$ has codimension at least $2$ in $Z$, equality on $Z_U$ implies equality globally. Hence $K_Z=f^*K_X+E$.
\end{proof}

\begin{Lemma}\label{lem:singularities-Z}
The singular locus of $Z$ consists of two irreducible curves $\Sigma_1,\Sigma_2\subset E$, each isomorphic to $\PP^1$. The general transversal singularity of $Z$ along each $\Sigma_i$ is of type $\frac12(1,1,0)$.
\end{Lemma}

\begin{proof}
Let $U:=X\setminus\{x_1,x_2\}$. By Construction~\ref{constr:WVX}, the variety $X$ is smooth on $U$, and by Construction~\ref{constr:XlZ} the curve $l\cap U=l\setminus\{x_1,x_2\}$ is smooth. Hence $f^{-1}(U)\to U$ is the ordinary blow-up of a smooth threefold along a smooth curve, so $f^{-1}(U)$ is smooth. Thus $\Sing(Z)\subset f^{-1}(x_1)\cup f^{-1}(x_2)\subset E$.

Fix $i\in\{1,2\}$. Since $x_i\in X$ is of type $\frac12(1,1,1)$, there is an analytic neighborhood $U_i$ of $x_i$ and an index-one cover $\pi_i\colon \A^3_{x,y,z}\to U_i\simeq \A^3/\mu_2$, where $\mu_2$ acts by $(x,y,z)\mapsto (-x,-y,-z)$. Since $C_V$ meets $F_i$ transversely in one smooth point, the image curve $l$ is smooth at $x_i$. After an equivariant analytic change of coordinates, we may assume that $\pi_i^{-1}(l\cap U_i)=L:=\{x=y=0\}$.

The germ $(U_i,l\cap U_i)$ is toric. Namely, $U_i$ is the affine toric threefold associated with the cone $\sigma=\Cone(e_1,e_2,e_3)$ in the lattice $N=\ZZ^3+\ZZ\cdot \frac12(1,1,1)$, and $l\cap U_i$ is the orbit closure corresponding to the face $\tau=\Cone(e_1,e_2)\subset \sigma$. Hence the blow-up of $U_i$ along $l\cap U_i$ is the toric morphism given by the star subdivision of $\sigma$ along the ray generated by $\rho=e_1+e_2$. The resulting fan has two maximal cones $\sigma_1=\Cone(e_1,\rho,e_3)$ and $\sigma_2=\Cone(e_2,\rho,e_3)$.

Thus the two affine charts of $f^{-1}(U_i)$ are analytically the quotients of $\widetilde U_{i,1}:=\Spec k[u,v,w]$ and $\widetilde U_{i,2}:=\Spec k[u',v',w']$ by the $\mu_2$-actions $(u,v,w)\mapsto (-u,v,-w)$ and $(u',v',w')\mapsto (u',-v',-w')$, respectively. Up to permuting the coordinates, each chart is a quotient singularity of type $\frac12(1,1,0)$.

On the first chart, the morphism to $U_i$ is induced on the index-one cover by $(x,y,z)=(u,uv,w)$, while on the second chart it is induced by $(x,y,z)=(u'v',v',w')$. Therefore the fiber over $x_i$ is given on the first chart by $u=w=0$, and on the second chart by $v'=w'=0$. These two copies of $\A^1$ are glued on the overlap by $u'=1/v$. Hence $f^{-1}(x_i)\simeq \PP^1$. We denote this curve by $\Sigma_i$. Since $f^{-1}(x_i)$ is contained in the exceptional locus of $f$, we have $\Sigma_i\subset E$.

On the first chart the singular locus is exactly $\{u=w=0\}$, and on the second chart it is exactly $\{v'=w'=0\}$, since in each chart the $\mu_2$-action is nontrivial on exactly two coordinates and trivial on the third. Thus the singular locus of $f^{-1}(U_i)$ is precisely $\Sigma_i$, and the transversal analytic singularity along its general point is $(\A^2/\mu_2)\times \A^1$, that is, of type $\frac12(1,1,0)$.

Since $Z$ is smooth away from $f^{-1}(x_1)\cup f^{-1}(x_2)$, it follows that $\Sing(Z)=\Sigma_1\cup\Sigma_2$, where $\Sigma_1,\Sigma_2\subset E$ are irreducible curves, each isomorphic to $\PP^1$, and the general transversal singularity of $Z$ along each $\Sigma_i$ is of type $\frac12(1,1,0)$.
\end{proof}

\begin{Lemma}\label{lem:E-irred}
The exceptional divisor $E\subset Z$ is irreducible and $Z$ is $\QQ$-factorial.
\end{Lemma}

\begin{proof}
Let $X^\circ:=X\setminus\{x_1,x_2\}$ and $l^\circ:=l\cap X^\circ=l\setminus\{x_1,x_2\}$. By Construction~\ref{constr:WVX}, the variety $X^\circ$ is smooth, and by Construction~\ref{constr:XlZ}, the curve $l^\circ$ is smooth. Hence $f^\circ\colon Z^\circ:=f^{-1}(X^\circ)\to X^\circ$ is the ordinary blow-up of a smooth threefold along a smooth curve. Therefore $E^\circ:=E\cap Z^\circ$ is the exceptional divisor of $f^\circ$, and $E^\circ\simeq \PP(N_{l^\circ/X^\circ})$. In particular, $E^\circ$ is irreducible.

It remains to analyze $E$ near the fibers over $x_1$ and $x_2$. Fix $i\in\{1,2\}$, and let $U_i$ be the analytic neighborhood of $x_i$ used in the proof of Lemma~\ref{lem:singularities-Z}. In the corresponding local model, the blow-up $f^{-1}(U_i)\to U_i$ is the toric morphism given by the star subdivision of $\sigma=\Cone(e_1,e_2,e_3)$ along the ray generated by $\rho=e_1+e_2$. Hence the exceptional divisor on $f^{-1}(U_i)$ is the torus-invariant prime divisor $D_\rho$ corresponding to the new ray $\rho$. In particular, $E\cap f^{-1}(U_i)=D_\rho$ is irreducible.

In the two affine charts described in Lemma~\ref{lem:singularities-Z}, this divisor is given by $\{u=0\}/\mu_2$ and $\{v'=0\}/\mu_2$, so it is locally irreducible near $f^{-1}(x_i)$. Moreover, $D_\rho\setminus f^{-1}(x_i)=E\cap f^{-1}(U_i\setminus\{x_i\})\subset E^\circ$, and this is a dense open subset of the irreducible surface $D_\rho$. Therefore $D_\rho\subset \overline{E^\circ}$. Since $E=E^\circ\cup (E\cap f^{-1}(U_1))\cup (E\cap f^{-1}(U_2))$, and each local piece $E\cap f^{-1}(U_i)$ is contained in $\overline{E^\circ}$, it follows that $E=\overline{E^\circ}$. Hence $E$ is irreducible.

Now we prove that $Z$ is $\QQ$-factorial. The variety $Z^\circ=f^{-1}(X^\circ)$ is smooth, hence factorial. Near each point of $f^{-1}(x_i)$, the variety $Z$ is covered by the two affine toric charts from Lemma~\ref{lem:singularities-Z}, corresponding to the cones $\sigma_1=\Cone(e_1,\rho,e_3)$ and $\sigma_2=\Cone(e_2,\rho,e_3)$. Each of these cones is simplicial, so the corresponding affine toric varieties are $\QQ$-factorial. Therefore $Z$ is $\QQ$-factorial in a neighborhood of $f^{-1}(x_i)$ for $i=1,2$. Since these neighborhoods together with $Z^\circ$ cover $Z$, we conclude that $Z$ is $\QQ$-factorial.
\end{proof}

\begin{Lemma}\label{lem:E-QCartier-index}
The prime divisor $E$ is $\QQ$-Cartier. It is Cartier away from $\Sigma_1\cup\Sigma_2$, and it has Cartier index $2$ at the generic point of each curve $\Sigma_i$.
\end{Lemma}

\begin{proof}
Away from $f^{-1}(x_1)\cup f^{-1}(x_2)$, the morphism $f$ is the blow-up of a smooth threefold along a smooth curve, hence $E$ is Cartier there. We use the local quotient charts from Lemma~\ref{lem:singularities-Z}. On the first chart near $\Sigma_i$ one has $Z\simeq \A^3_{u,v,w}/\mu_2$, with action $(u,v,w)\mapsto (-u,v,-w)$, and $E$ is the image of $\{u=0\}$. The invariant ring is
$$
        \kk[u,v,w]^{\mu_2}=\kk[v,a,b,c]/(ac-b^2),\qquad a=u^2,\ b=uw,\ c=w^2.
$$
Thus the image of $\{u=0\}$ is the prime divisor with ideal $(a,b)$.

Let $\eta_i$ be the generic point of the singular curve in this chart. Then the completed local ring at $\eta_i$ is $\kk(v)[[a,b,c]]/(ac-b^2)$, with maximal ideal $(a,b,c)$. The ideal $(a,b)$ is not principal: $(a,b)/(a,b)(a,b,c)$ has dimension $2$ over $\kk(v)$, so it cannot be generated by one element by Nakayama's lemma. Hence $E$ is not Cartier at $\eta_i$.

On the other hand, at the generic point of $E$ the function $c$ is invertible and $a=b^2/c$. Hence the invariant function $a=u^2$ cuts out $2E$. Thus $2E$ is Cartier in the chart. The same computation in the second chart, where $E$ is the image of $\{v'=0\}$, gives the same conclusion along the other affine piece of $\Sigma_i$. Therefore $E$ is $\QQ$-Cartier of local index $2$ along each $\Sigma_i$, and it is Cartier elsewhere.
\end{proof}

\begin{Proposition}\label{prop:relative-cone-Z}
If $e\subset E$ is a fiber of $E\to l$ over a general point of $l$, then $\NEbar(Z/X)=\RR_{\geq 0}[e]$. In particular, $\rho(Z/X)=1$ and $\rho(Z)=2$.
\end{Proposition}

\begin{proof}
We first compute $\rho(X)$. Since $\rho(Q)=1$ and $\Theta\colon W\to Q$ is the blow-up of the smooth threefold $Q$ along the two disjoint smooth curves $\Gamma_1$ and $\Gamma_2$, we have $\rho(W)=3$. The map $\psi\colon W\dashrightarrow V$ is a sequence of flops, hence $\rho(V)=3$.

Now $\phi\colon V\to X$ contracts exactly the two disjoint prime divisors $F_1$ and $F_2$. Since the points $x_1,x_2\in X$ are quotient singularities of type $\frac12(1,1,1)$, the variety $X$ is $\QQ$-factorial. Thus $N^1(V/X)_\RR$ is generated by the classes of the $\phi$-exceptional divisors. Since $\Exc(\phi)=F_1\sqcup F_2$, the space $N^1(V/X)_\RR$ is generated by $[F_1]$ and $[F_2]$. These classes are linearly independent: if $c_i\subset F_i$ is a curve contracted by $\phi$, then $F_i\cdot c_i<0$ and $F_j\cdot c_i=0$ for $j\neq i$, since $F_1\cap F_2=\varnothing$. Hence $\rho(V/X)=2$, and therefore $\rho(X)=\rho(V)-\rho(V/X)=1$.

We now compute $\rho(Z/X)$. By Lemma~\ref{lem:E-irred}, the exceptional divisor $E$ is irreducible and $Z$ is $\QQ$-factorial. Let $D$ be a $\QQ$-Cartier divisor on $Z$. Since $X$ is $\QQ$-factorial, the pushforward $f_*D$ is $\QQ$-Cartier on $X$, and $D-f^*f_*D$ is an $f$-exceptional $\QQ$-Cartier divisor. Since $E$ is the only $f$-exceptional prime divisor, we have $D-f^*f_*D\sim_{\QQ} aE$ for some $a\in\QQ$. Thus the image of $D$ in $N^1(Z/X)_\RR$ is a real multiple of $[E]$, so $N^1(Z/X)_\RR=\RR\cdot [E]$. Hence $\rho(Z/X)=1$.

Let $e\subset E$ be a fiber of $E\to l$ over a general point of $l$. Then $e$ is contracted by $f$, so $[e]\in\NEbar(Z/X)$. Moreover $[e]\neq 0$, since $-E$ is $f$-ample as a $\QQ$-Cartier divisor; indeed, a positive Cartier multiple, for instance $-2E$, is relatively ample over $X$. Thus $(-E)\cdot e>0$.

Since $N_1(Z/X)_\RR$ is dual to the one-dimensional space $N^1(Z/X)_\RR$, it is one-dimensional. Hence every class in $N_1(Z/X)_\RR$ is of the form $\lambda[e]$ for some $\lambda\in\RR$. If $C\subset Z$ is any integral curve contracted by $f$, then $C\subset E$, and $(-E)\cdot C>0$ since $-E$ is $f$-ample over $X$. Writing $[C]=\lambda[e]$, we get $(-E)\cdot C=\lambda(-E)\cdot e$. Since both intersection numbers are strictly positive, $\lambda>0$. Therefore every effective $f$-contracted curve class is a nonnegative multiple of $[e]$, and hence $\NEbar(Z/X)=\RR_{\geq 0}[e]$. Finally, $\rho(Z)=\rho(X)+\rho(Z/X)=2$.
\end{proof}

\begin{Lemma}\label{lem:class-group-X}
We have $\Cl(X)=\ZZ[-K_X]$.
\end{Lemma}

\begin{proof}
The threefold $V$ is smooth and $\Pic(V)=\ZZ H_V\oplus \ZZ E_1^V\oplus \ZZ E_2^V$. The morphism $\phi\colon V\to X$ contracts exactly the two disjoint planes $F_1$ and $F_2$. Since $X$ is $\QQ$-factorial and $\phi$ is an isomorphism in codimension one outside $F_1\cup F_2$, the push-forward induces an isomorphism $
        \Pic(V)/\langle F_1,F_2\rangle \simeq \Cl(X)$.

On $V$ one has $F_i\sim H_V-E_i^V$. Hence, in the quotient $\Pic(V)/\langle F_1,F_2\rangle$, the three classes $H_V$, $E_1^V$, and $E_2^V$ have the same image. Denote this image by $h$. Then the quotient is freely generated by $h$. Moreover, by Construction~\ref{constr:WVX}, one has $\phi^*(-2K_X)\sim 8H_V-3E_1^V-3E_2^V$, whose image in the quotient is $2h$. Thus the class of $-2K_X$ is $2h$, and the class of $-K_X$ is $h$. Therefore $\Cl(X)=\ZZ[-K_X]$.
\end{proof}

\begin{Construction}\label{constr:auxiliary}
Let $\pi\colon \widehat Q:=\Bl_C(Q)\to Q$ be the blow-up of $Q$ along $C$, and let $\alpha\colon \widehat V:=\Bl_{\widehat\Gamma_1\sqcup\widehat\Gamma_2}(\widehat Q)\to \widehat Q$ be the blow-up of the strict transforms $\widehat\Gamma_1,\widehat\Gamma_2\subset \widehat Q$ of $\Gamma_1,\Gamma_2$. We denote by $H\in \Pic(\widehat V)$ the pullback of the hyperplane class of $Q$, by $\widehat E\subset \widehat V$ the strict transform of the exceptional divisor of $\pi$, by $T_1,T_2$ the exceptional divisors of $\alpha$, and by $G_1,G_2\subset \widehat V$ the strict transforms of $Q_1,Q_2$.
\end{Construction}

\begin{Lemma}\label{lem:quadrics}
For each $i=1,2$, the surface $Q_i$ is isomorphic to $\PP^1\times\PP^1$. After possibly exchanging the two rulings, we have $C\in |\OO_{Q_i}(1,1)|$ and $\Gamma_i\in |\OO_{Q_i}(1,2)|$.
\end{Lemma}

\begin{proof}
By Assumption~\ref{ass:general-setting}(i), the surface $Q_i=Q\cap H_i\subset H_i\simeq \PP^3$ is a smooth quadric surface. Hence $Q_i\simeq \PP^1\times \PP^1$, and the hyperplane class of $Q_i\subset \PP^3$ is $\OO_{Q_i}(1,1)$.

Since $Q_{3-i}=Q\cap H_{3-i}$ and $H_{3-i}\subset \PP^4$ is a hyperplane, we have $C=Q_i\cap Q_{3-i}=Q_i\cap H_{3-i}$. Thus $C$ is a hyperplane section of $Q_i$, so $C\in |\OO_{Q_i}(1,1)|$.

Now write $\Gamma_i\in |\OO_{Q_i}(a,b)|$ with $a,b\geq 0$. Its degree in $\PP^3$ is computed with respect to $\OO_{Q_i}(1,1)$, hence $\deg(\Gamma_i)=a+b$. Since $\Gamma_i$ is a twisted cubic, $a+b=3$. Moreover, $\Gamma_i$ is smooth and rational, so $p_a(\Gamma_i)=0$. On $\PP^1\times\PP^1$, a curve of bidegree $(a,b)$ has arithmetic genus $(a-1)(b-1)$. Therefore $(a-1)(b-1)=0$. Combining this with $a+b=3$ gives $\{a,b\}=\{1,2\}$. After possibly exchanging the two rulings on $Q_i$, we get $\Gamma_i\in |\OO_{Q_i}(1,2)|$.
\end{proof}

\begin{Proposition}\label{prop:classes-Vhat}
We have $\Pic(\widehat V)=\ZZ H\oplus \ZZ \widehat E\oplus \ZZ T_1\oplus \ZZ T_2$. Moreover $-K_{\widehat V}\sim 3H-\widehat E-T_1-T_2$, $G_i\sim H-\widehat E-T_i$, and $M:=|H-\widehat E|$ is a base-point-free pencil on $\widehat V$. For each $i=1,2$, the divisor $G_i+T_i$ is a member of $M$.
\end{Proposition}

\begin{proof}
Since $Q$ is smooth and $C\subset Q$ is a smooth curve, the blow-up $\pi\colon \widehat Q=\Bl_C(Q)\to Q$ is smooth. Moreover $\Pic(Q)=\ZZ[H_Q]$, where $H_Q$ is the hyperplane class on $Q$, and for the blow-up of a smooth variety along a smooth center of codimension $2$ one has $\Pic(\widehat Q)\cong \pi^*\Pic(Q)\oplus \ZZ[E_\pi]$, where $E_\pi\subset \widehat Q$ is the exceptional divisor of $\pi$. Hence $\Pic(\widehat Q)=\ZZ[\pi^*H_Q]\oplus \ZZ[E_\pi]$.

The curves $\widehat\Gamma_1,\widehat\Gamma_2\subset \widehat Q$ are smooth and disjoint. Indeed, each $\Gamma_i$ is smooth, meets $C$ transversely in finitely many points by Assumption~\ref{ass:general-setting}(iii), and the strict transform of a smooth curve under the blow-up of finitely many points is again smooth. Moreover $\Gamma_1\cap\Gamma_2=\varnothing$ by Assumption~\ref{ass:general-setting}(v), so their strict transforms remain disjoint. Thus $\alpha\colon \widehat V=\Bl_{\widehat\Gamma_1\sqcup\widehat\Gamma_2}(\widehat Q)\to \widehat Q$ is the blow-up of a smooth variety along two disjoint smooth curves, with exceptional divisors $T_1$ and $T_2$, and $\Pic(\widehat V)\cong \alpha^*\Pic(\widehat Q)\oplus \ZZ[T_1]\oplus \ZZ[T_2]$. If $H:=\alpha^*\pi^*H_Q$ and $\widehat E:=\alpha_*^{-1}(E_\pi)$, then $\Pic(\widehat V)=\ZZ H\oplus \ZZ \widehat E\oplus \ZZ T_1\oplus \ZZ T_2$, since the centers $\widehat\Gamma_1,\widehat\Gamma_2$ are not contained in $E_\pi$.

We compute the canonical class. Since $Q\subset \PP^4$ is a smooth quadric threefold, one has $-K_Q\sim 3H_Q$. Since $\pi$ is the blow-up of a smooth threefold along a smooth curve, $K_{\widehat Q}=\pi^*K_Q+E_\pi$. Similarly, since $\alpha$ is the blow-up of $\widehat Q$ along the two disjoint smooth curves $\widehat\Gamma_1,\widehat\Gamma_2$, we have $K_{\widehat V}=\alpha^*K_{\widehat Q}+T_1+T_2$. Therefore
$$
-K_{\widehat V}\sim -\alpha^*K_{\widehat Q}-T_1-T_2\sim \alpha^*(3H_Q-E_\pi)-T_1-T_2\sim 3H-\widehat E-T_1-T_2.
$$

Next we determine the class of $G_i$. Let $\widehat Q_i\subset \widehat Q$ be the strict transform of $Q_i$. Since $Q_i$ is a hyperplane section of $Q$ containing $C$, its strict transform satisfies $\widehat Q_i\sim \pi^*H_Q-E_\pi$. Now $\widehat\Gamma_i\subset \widehat Q_i$ since $\Gamma_i\subset Q_i$, whereas $\widehat\Gamma_{3-i}$ is not contained in $\widehat Q_i$ since $\Gamma_{3-i}\not\subset Q_i$ by Assumption~\ref{ass:general-setting}(iii). Hence the strict transform $G_i\subset \widehat V$ satisfies $G_i\sim \alpha^*(\widehat Q_i)-T_i\sim H-\widehat E-T_i$.

Finally, let $\Pi:=\langle C\rangle\subset \PP^4$. Since $C$ is a smooth conic, it spans a plane, and a hyperplane of $\PP^4$ contains $C$ if and only if it contains $\Pi$. Therefore the hyperplanes of $\PP^4$ containing $\Pi$ form a pencil. Restricting to $Q$, we obtain a pencil of hyperplane sections of $Q$ containing $C$, and after blowing up $C$ this becomes the complete linear system $|\pi^*H_Q-E_\pi|$ on $\widehat Q$. This linear system is base-point-free, since the only base locus downstairs is $C$, which is resolved by the blow-up. Pulling back via $\alpha$, we obtain the base-point-free pencil $M:=|H-\widehat E|$ on $\widehat V$.

For each $i=1,2$, the divisor $\widehat Q_i$ is a member of $|\pi^*H_Q-E_\pi|$, since $Q_i$ is the hyperplane section cut out by $H_i=\langle \Gamma_i\rangle$ and $C\subset Q_i$. Its total transform on $\widehat V$ is $\alpha^*\widehat Q_i=G_i+T_i$, since $\widehat\Gamma_i\subset \widehat Q_i$ with multiplicity one. Hence $G_i+T_i\in |H-\widehat E|=M$.
\end{proof}

\begin{Proposition}\label{prop:restrictions-Gi}
For each $i=1,2$, the divisor $G_i$ is isomorphic to $\PP^1\times\PP^1$. Under this identification one has $H|_{G_i}\sim(1,1)$, $\widehat E|_{G_i}\sim(1,1)$, $T_i|_{G_i}\sim(1,2)$, and $T_j|_{G_i}\sim 0$ for $j\neq i$. Consequently $-K_{\widehat V}|_{G_i}\sim(1,0)$ and $G_i|_{G_i}\sim(-1,-2)$.
\end{Proposition}

\begin{proof}
Fix $i\in\{1,2\}$, and let $\widehat Q_i\subset \widehat Q$ be the strict transform of $Q_i$. Since $Q_i\subset Q$ is a smooth surface and $C\subset Q_i$ is a smooth Cartier divisor, the morphism $\pi|_{\widehat Q_i}\colon \widehat Q_i\to Q_i$ is the blow-up of the smooth surface $Q_i$ along the Cartier divisor $C$. Hence $\pi|_{\widehat Q_i}$ is an isomorphism. By Lemma~\ref{lem:quadrics}, we get $\widehat Q_i\simeq Q_i\simeq \PP^1\times\PP^1$. Under this identification, $H|_{\widehat Q_i}\sim \OO_{Q_i}(1,1)$ and $\widehat E|_{\widehat Q_i}\sim C\sim \OO_{Q_i}(1,1)$. Moreover, the strict transform $\widehat\Gamma_i\subset \widehat Q_i$ corresponds to $\Gamma_i\in |\OO_{Q_i}(1,2)|$.

We claim that $\widehat\Gamma_j\cap \widehat Q_i=\varnothing$ for $j\neq i$. By Assumption~\ref{ass:general-setting}(iii), one has $\Gamma_j\cap Q_i=\Gamma_j\cap C$, and the intersection is transverse at each of these three points. Let $p\in \Gamma_j\cap C$. Since $Q$ is smooth at $p$, we may choose analytic coordinates $(x,y,z)$ on $Q$ centered at $p$ such that $Q_i=\{x=0\}$ and $C=\{x=y=0\}$. Since $\Gamma_j$ meets $Q_i$ transversely at $p$, if $t$ is a local parameter on $\Gamma_j$ at $p$, we may write $x=t$ and $y=t h(t)$ for some holomorphic function $h$.

The blow-up of $Q$ along $C$ is the blow-up of the ideal $(x,y)$. In the chart $y=vx$, the strict transform of $\Gamma_j$ is given by $x=t$ and $v=h(t)$, while the strict transform of $Q_i$ is not present, since the total transform of $Q_i$ is contained in the exceptional divisor. In the chart $x=uy$, the strict transform of $Q_i$ is given by $u=0$. If the strict transform of $\Gamma_j$ appears in this chart, then it is given by $y=t h(t)$ and $u=1/h(t)$, so its limit over $p$ is never on $u=0$. Thus the strict transforms of $\Gamma_j$ and $Q_i$ do not meet over $p$. Since this holds at every point of $\Gamma_j\cap C$, we obtain $\widehat\Gamma_j\cap \widehat Q_i=\varnothing$.

Consider now the blow-up $\alpha\colon \widehat V=\Bl_{\widehat\Gamma_1\sqcup\widehat\Gamma_2}(\widehat Q)\to \widehat Q$. Since $\widehat\Gamma_i\subset \widehat Q_i$ is a smooth Cartier divisor on the smooth surface $\widehat Q_i$, the restriction $\alpha|_{G_i}\colon G_i\to \widehat Q_i$ is the blow-up of $\widehat Q_i$ along the Cartier divisor $\widehat\Gamma_i$, hence it is an isomorphism. Therefore $G_i\simeq \widehat Q_i\simeq \PP^1\times\PP^1$.

Under this identification, since $\alpha|_{G_i}$ is an isomorphism and $H,\widehat E$ are pullbacks from $\widehat Q$, we have $H|_{G_i}\sim(1,1)$ and $\widehat E|_{G_i}\sim(1,1)$. Since $T_i$ is the exceptional divisor of the blow-up along $\widehat\Gamma_i$, its restriction to $G_i$ is the center itself, so $T_i|_{G_i}\sim \widehat\Gamma_i\sim(1,2)$. On the other hand, $\widehat\Gamma_j\cap \widehat Q_i=\varnothing$ for $j\neq i$, so $G_i\cap T_j=\varnothing$ and $T_j|_{G_i}\sim 0$.

Finally, by Proposition~\ref{prop:classes-Vhat}, one has $-K_{\widehat V}\sim 3H-\widehat E-T_1-T_2$ and $G_i\sim H-\widehat E-T_i$. Restricting to $G_i$ gives $-K_{\widehat V}|_{G_i}\sim 3(1,1)-(1,1)-(1,2)\sim(1,0)$ and $G_i|_{G_i}\sim(1,1)-(1,1)-(1,2)\sim(-1,-2)$.
\end{proof}

\begin{Corollary}\label{cor:rulings}
Let $r_i$ and $s_i$ denote the two rulings on $G_i$, with $r_i$ of class $(1,0)$ and $s_i$ of class $(0,1)$. Then $-K_{\widehat V}\cdot r_i=0$, $-K_{\widehat V}\cdot s_i=1$, $G_i\cdot r_i=-2$, and $G_i\cdot s_i=-1$. In particular, the only curves contained in $G_i$ on which $-K_{\widehat V}$ has degree $0$ are the fibers of the ruling $r_i$.
\end{Corollary}

\begin{proof}
By Proposition~\ref{prop:restrictions-Gi}, after identifying
$G_i\simeq \PP^1\times \PP^1,$
we have $-K_{\widehat V}|_{G_i}\sim (1,0),
\,
G_i|_{G_i}\sim (-1,-2)$. Let $r_i$ and $s_i$ be the curve classes $(1,0)$ and $(0,1)$, respectively. Using the intersection pairing on $\PP^1\times \PP^1$, namely
$(1,0)^2=(0,1)^2=0,
\qquad
(1,0)\cdot (0,1)=1,$
we obtain
\[
-K_{\widehat V}\cdot r_i
=
\bigl(-K_{\widehat V}|_{G_i}\bigr)\cdot r_i
=
(1,0)\cdot (1,0)
=
0,\quad
-K_{\widehat V}\cdot s_i
=
\bigl(-K_{\widehat V}|_{G_i}\bigr)\cdot s_i
=
(1,0)\cdot (0,1)
=
1,
\]
and similarly
\[
G_i\cdot r_i
=
\bigl(G_i|_{G_i}\bigr)\cdot r_i
=
(-1,-2)\cdot (1,0)
=
-2,\quad
G_i\cdot s_i
=
\bigl(G_i|_{G_i}\bigr)\cdot s_i
=
(-1,-2)\cdot (0,1)
=
-1.
\]
Now let $B\subset G_i$ be an irreducible curve. Since
$G_i\simeq \PP^1\times \PP^1,$
its class can be written as
$B\equiv a\,r_i+b\,s_i
\text{ with } a,b\ge 0.$
Therefore $
-K_{\widehat V}\cdot B
=
\bigl(-K_{\widehat V}|_{G_i}\bigr)\cdot B
=
(1,0)\cdot (a,b)
=
b$. Hence
$
-K_{\widehat V}\cdot B=0
\quad\Longleftrightarrow\quad
b=0.
$ So any irreducible curve $B\subset G_i$ with $-K_{\widehat V}\cdot B=0$ has class
$B\equiv a\,r_i$
for some $a>0$. Since on $\PP^1\times \PP^1$ numerical and linear equivalence coincide, this means
$B\in |\OO_{\PP^1\times \PP^1}(a,0)|.$
But every divisor in $|\OO_{\PP^1\times \PP^1}(a,0)|$ is the pullback of a degree-$a$ divisor on the first factor, hence is a union of $a$ fibers of the ruling of class $(1,0)$. Since $B$ is irreducible, necessarily $a=1$, and therefore $B$ is a fiber of the ruling $r_i$.
\end{proof}

\begin{Lemma}\label{lem:Ti}
For each $i=1,2$, one has $N_{\widehat\Gamma_i/\widehat Q}\cong \OO_{\PP^1}(4)\oplus\OO_{\PP^1}$. Hence $T_i\simeq \F_4$. If $f_i$ denotes a fiber of the ruling $T_i\to\widehat\Gamma_i$ and $S_i:=G_i\cap T_i$, then $S_i$ is the negative section of $T_i$, so $S_i^2=-4$ and $\NEbar(T_i)=\RR_{\geq 0}[S_i]+\RR_{\geq 0}[f_i]$.
\end{Lemma}

\begin{proof}
Fix $i\in\{1,2\}$, and let $\widehat Q_i\subset \widehat Q$ be the strict transform of $Q_i$. As in the proof of Proposition~\ref{prop:restrictions-Gi}, the restriction $\pi|_{\widehat Q_i}\colon \widehat Q_i\to Q_i$ is an isomorphism. Hence $\widehat Q_i\simeq Q_i\simeq \PP^1\times\PP^1$. By Lemma~\ref{lem:quadrics}, under this identification one has $\widehat\Gamma_i\in |\OO_{\widehat Q_i}(1,2)|$. Therefore $N_{\widehat\Gamma_i/\widehat Q_i}\cong \OO_{\widehat\Gamma_i}(\widehat\Gamma_i)\cong \OO_{\PP^1}((1,2)^2)\cong \OO_{\PP^1}(4)$.

We compute the normal bundle of $\widehat Q_i$ in $\widehat Q$ along $\widehat\Gamma_i$. Since $Q_i\subset Q$ is a hyperplane section containing $C$, its strict transform satisfies $\widehat Q_i\sim \pi^*H_Q-E_\pi$. Restricting to $\widehat Q_i\simeq Q_i$, we get $\OO_{\widehat Q_i}(\widehat Q_i)\cong \OO_{Q_i}(1,1)\otimes \OO_{Q_i}(-C)$. By Lemma~\ref{lem:quadrics}, $C\in |\OO_{Q_i}(1,1)|$, hence $\OO_{\widehat Q_i}(\widehat Q_i)\cong \OO_{\widehat Q_i}$. Thus $N_{\widehat Q_i/\widehat Q}|_{\widehat\Gamma_i}\cong \OO_{\PP^1}$. The normal exact sequence for $\widehat\Gamma_i\subset \widehat Q_i\subset \widehat Q$ gives $0\to \OO_{\PP^1}(4)\to N_{\widehat\Gamma_i/\widehat Q}\to \OO_{\PP^1}\to 0$. Since $\Ext^1_{\PP^1}(\OO_{\PP^1},\OO_{\PP^1}(4))=H^1(\PP^1,\OO_{\PP^1}(4))=0$, the sequence splits. Hence $N_{\widehat\Gamma_i/\widehat Q}\cong \OO_{\PP^1}(4)\oplus\OO_{\PP^1}$, and $T_i=\PP(N_{\widehat\Gamma_i/\widehat Q})\simeq \F_4$.

Let $S_i:=G_i\cap T_i$. Since $G_i$ is the strict transform of $\widehat Q_i$ under the blow-up $\alpha\colon \widehat V\to \widehat Q$, and since $\widehat\Gamma_i\subset \widehat Q_i$ is a Cartier divisor on the smooth surface $\widehat Q_i$, the restriction $\alpha|_{G_i}\colon G_i\to \widehat Q_i$ is an isomorphism. Under this identification, $S_i$ corresponds to $\widehat\Gamma_i\subset \widehat Q_i$. Thus, by Proposition~\ref{prop:restrictions-Gi}, the curve $S_i\subset G_i\simeq \PP^1\times\PP^1$ has class $S_i\sim T_i|_{G_i}\sim(1,2)$, while $G_i|_{G_i}\sim(-1,-2)$.

Since $S_i=G_i|_{T_i}$, one has $\OO_{T_i}(S_i)=\OO_{\widehat V}(G_i)|_{T_i}$. Therefore $N_{S_i/T_i}\cong \OO_{T_i}(S_i)|_{S_i}\cong \OO_{\widehat V}(G_i)|_{S_i}$. Taking degrees, we obtain $S_i^2=\deg N_{S_i/T_i}=(G_i|_{G_i})\cdot S_i=(-1,-2)\cdot(1,2)=-4$. Thus $S_i$ is a section of the ruling $T_i\to \widehat\Gamma_i\simeq \PP^1$ with self-intersection $-4$. Since $\F_4$ has a unique section of self-intersection $-4$, the curve $S_i$ is the negative section of $T_i$.

Finally, if $f_i$ denotes a fiber of $T_i\to\widehat\Gamma_i$, then the Mori cone of $T_i\simeq \F_4$ is generated by the negative section and a fiber. Hence $\NEbar(T_i)=\RR_{\geq 0}[S_i]+\RR_{\geq 0}[f_i]$.
\end{proof}

\begin{Lemma}\label{lem:Ehat}
Let $E_\pi\subset \widehat Q$ be the exceptional divisor of $\pi$. Then $N_{C/Q}\cong \OO_{\PP^1}(2)\oplus\OO_{\PP^1}(2)$ and $E_\pi\simeq \PP^1\times\PP^1$. Moreover $\widehat\Gamma_1$ and $\widehat\Gamma_2$ meet $E_\pi$ transversely in six distinct points, three on each of the two disjoint sections $s_1=E_\pi\cap \widehat Q_1$ and $s_2=E_\pi\cap \widehat Q_2$. Thus $\widehat E$ is the blow-up of $E_\pi$ at these six points.
\end{Lemma}

\begin{proof}
By Lemma~\ref{lem:quadrics}, for each $i=1,2$ we have $Q_i\simeq \PP^1\times\PP^1$ and $C\in |\OO_{Q_i}(1,1)|$. Since $C$ is a smooth divisor on the smooth surface $Q_i$, one has $N_{C/Q_i}\cong \OO_C(C)\cong \OO_{\PP^1}(2)$. Moreover $Q_i\subset Q$ is a hyperplane section, so $N_{Q_i/Q}\cong \OO_{Q_i}(1,1)$, and hence $N_{Q_i/Q}|_C\cong \OO_{\PP^1}(2)$.

The normal exact sequence for $C\subset Q_i\subset Q$ gives $0\to N_{C/Q_i}\to N_{C/Q}\to N_{Q_i/Q}|_C\to 0$, that is, $0\to \OO_{\PP^1}(2)\to N_{C/Q}\to \OO_{\PP^1}(2)\to 0$. Since $\Ext^1_{\PP^1}(\OO_{\PP^1}(2),\OO_{\PP^1}(2))=H^1(\PP^1,\OO_{\PP^1})=0$, this sequence splits. Therefore $N_{C/Q}\cong \OO_{\PP^1}(2)\oplus\OO_{\PP^1}(2)$.

Thus the exceptional divisor of $\pi\colon \widehat Q=\Bl_C(Q)\to Q$ is $E_\pi=\PP(N_{C/Q})\cong \PP(\OO_{\PP^1}(2)\oplus\OO_{\PP^1}(2))\cong \PP^1\times\PP^1$.

Let $\widehat Q_i\subset \widehat Q$ be the strict transform of $Q_i$. Since $C\subset Q_i$ is a smooth Cartier divisor on the smooth surface $Q_i$, the restriction $\pi|_{\widehat Q_i}\colon \widehat Q_i\to Q_i$ is the blow-up of $Q_i$ along a Cartier divisor, hence it is an isomorphism. In particular, $s_i:=E_\pi\cap \widehat Q_i$ maps isomorphically onto $C$, so $s_i$ is a section of the ruling $E_\pi\to C$.

We claim that $s_1$ and $s_2$ are disjoint. Indeed, $Q_1$ and $Q_2$ are smooth divisors on the smooth threefold $Q$, and $Q_1\cap Q_2=C$ is smooth of codimension $2$ in $Q$. Hence $Q_1$ and $Q_2$ meet transversely along $C$, so their strict transforms $\widehat Q_1$ and $\widehat Q_2$ in $\Bl_C(Q)$ are disjoint. Thus $s_1\cap s_2=\varnothing$.

We next analyze the intersections of $\widehat\Gamma_i$ with $E_\pi$. Fix $i\in\{1,2\}$ and let $p\in \Gamma_i\cap C$. By Assumption~\ref{ass:general-setting}(iii), the intersection is transverse and the points of $\Gamma_i\cap C$ are distinct. Since $Q$ is smooth at $p$, we may choose analytic coordinates $(u,v,w)$ on $Q$ centered at $p$ such that $Q_i=\{w=0\}$, $C=\{u=w=0\}$, and $\Gamma_i=\{v=w=0\}$. Indeed, inside the smooth surface $Q_i$ the curves $C$ and $\Gamma_i$ are smooth and meet transversely at $p$, so after choosing coordinates $(u,v)$ on $Q_i$ with $C=\{u=0\}$ and $\Gamma_i=\{v=0\}$, one extends them by a transverse coordinate $w$ defining $Q_i$ in $Q$.

Now $\widehat Q=\Bl_C(Q)$ is locally the blow-up of the ideal $(u,w)$. In the chart $w=\lambda u$, the strict transform of $Q_i$ is given by $\lambda=0$, the exceptional divisor is $u=0$, and the strict transform of $\Gamma_i$ is given by $v=\lambda=0$. Therefore $\widehat\Gamma_i$ meets $E_\pi$ at the single point $u=v=\lambda=0$, which lies on the section $s_i=E_\pi\cap \widehat Q_i$. Moreover, since $\widehat\Gamma_i$ is locally cut out by $v=\lambda=0$, while $E_\pi$ is given by $u=0$, the intersection is transverse.

Thus, for each of the three points of $\Gamma_i\cap C$, the curve $\widehat\Gamma_i$ meets $E_\pi$ transversely in one point of $s_i$. Since these three points are distinct, the corresponding intersection points are distinct. Applying this to $i=1,2$, and using $s_1\cap s_2=\varnothing$, we conclude that $\widehat\Gamma_1$ and $\widehat\Gamma_2$ meet $E_\pi$ transversely in six distinct points, three on $s_1$ and three on $s_2$.

Finally, $\alpha\colon \widehat V\to \widehat Q$ is the blow-up of the smooth threefold $\widehat Q$ along the disjoint smooth curves $\widehat\Gamma_1$ and $\widehat\Gamma_2$. Since $E_\pi$ is a smooth divisor on $\widehat Q$ and each $\widehat\Gamma_i$ meets $E_\pi$ transversely in finitely many points, the strict transform $\widehat E\subset \widehat V$ is the blow-up of $E_\pi$ at these intersection points. Hence $\widehat E$ is the blow-up of $E_\pi$ at the six points described above.
\end{proof}

\begin{Proposition}\label{prop:restrictions-TE}
With notation as above, one has $H|_{T_i}\sim 3f_i$, $\widehat E|_{T_i}\sim 3f_i$, $G_i|_{T_i}\sim S_i$, and $T_i|_{T_i}\sim -S_i$, where $f_i$ is a fiber of $T_i\to \widehat\Gamma_i$. Hence $-K_{\widehat V}|_{T_i}\sim S_i+6f_i$.

Moreover, let $\beta\colon \widehat E\to \PP^1\times\PP^1$ be the blow-up map, and let $a,b$ denote the pullbacks of the two rulings on $\PP^1\times\PP^1$, chosen so that $H|_{\widehat E}\sim 2b$ and $(H-\widehat E)|_{\widehat E}\sim a$. If $A_1,A_2,A_3,B_1,B_2,B_3$ are the six exceptional curves, then
$$
H|_{\widehat E}\sim 2b,\quad
\widehat E|_{\widehat E}\sim 2b-a,\quad
T_1|_{\widehat E}\sim \sum_{j=1}^3 A_j,\quad
T_2|_{\widehat E}\sim \sum_{j=1}^3 B_j,\quad
G_1|_{\widehat E}\sim a-\sum_{j=1}^3 A_j,\quad
G_2|_{\widehat E}\sim a-\sum_{j=1}^3 B_j.
$$
Moreover $-K_{\widehat V}|_{\widehat E}\sim a+4b-\sum_{j=1}^3 A_j-\sum_{j=1}^3 B_j$.

Set $\ell_j:=b-A_j$ and $m_j:=b-B_j$ for $j=1,2,3$. Then the six curves $\ell_1,\ell_2,\ell_3,m_1,m_2,m_3$ are pairwise disjoint $(-1)$-curves on $\widehat E$. Contracting them yields a morphism $\eta\colon \widehat E\to \PP^1\times\PP^1$ such that
$$
\eta^*\OO_{\PP^1\times\PP^1}(1,0)\sim a+3b-\sum_{j=1}^3 A_j-\sum_{j=1}^3 B_j,
\qquad
\eta^*\OO_{\PP^1\times\PP^1}(0,1)\sim b.
$$
In particular, $-K_{\widehat V}|_{\widehat E}\sim \eta^*\OO_{\PP^1\times\PP^1}(1,1)$.
\end{Proposition}

\begin{proof}
We first compute the restrictions to $T_i$. Let $p_i\colon T_i\to \widehat\Gamma_i\simeq \PP^1$ be the ruling, and let $f_i$ be a fiber. Since $H$ and $\widehat E$ come from divisors on $\widehat Q$, their restrictions to $T_i=\PP(N_{\widehat\Gamma_i/\widehat Q})$ are pullbacks from the base: $H|_{T_i}\cong p_i^*(H|_{\widehat\Gamma_i})$ and $\widehat E|_{T_i}\cong p_i^*(\widehat E|_{\widehat\Gamma_i})$. Now $\deg H|_{\widehat\Gamma_i}=H\cdot \widehat\Gamma_i=\deg\Gamma_i=3$, since $\Gamma_i\subset Q\subset \PP^4$ is a twisted cubic, and by Lemma~\ref{lem:Ehat} the curve $\widehat\Gamma_i$ meets $\widehat E$ transversely in three points, so $\deg \widehat E|_{\widehat\Gamma_i}=3$. Hence $H|_{T_i}\sim 3f_i$ and $\widehat E|_{T_i}\sim 3f_i$.

By definition $S_i=G_i\cap T_i$, and $G_i$ does not contain $T_i$, hence $G_i|_{T_i}\sim S_i$. Moreover $T_1\cap T_2=\varnothing$, since the centers $\widehat\Gamma_1$ and $\widehat\Gamma_2$ are disjoint, so $T_j|_{T_i}\sim 0$ for $j\neq i$. Proposition~\ref{prop:classes-Vhat} gives $G_i+T_i\sim H-\widehat E$. Restricting to $T_i$ and using $H|_{T_i}\sim \widehat E|_{T_i}$, we obtain $S_i+T_i|_{T_i}\sim 0$, hence $T_i|_{T_i}\sim -S_i$. Finally, since $-K_{\widehat V}\sim 3H-\widehat E-T_1-T_2$, restricting to $T_i$ gives $-K_{\widehat V}|_{T_i}\sim 3(3f_i)-3f_i-(-S_i)\sim S_i+6f_i$.

We now turn to $\widehat E$. By Lemma~\ref{lem:Ehat}, the surface $\widehat E$ is the blow-up of $E_\pi\simeq \PP^1\times\PP^1$ at six points, three on each of the two disjoint sections $E_\pi\cap \widehat Q_1$ and $E_\pi\cap \widehat Q_2$. Let $\beta\colon \widehat E\to E_\pi\simeq \PP^1\times\PP^1$ be this blow-up, and let $a,b\in \Pic(\widehat E)$ be the pullbacks of the two rulings, chosen so that $H|_{\widehat E}\sim 2b$ and $(H-\widehat E)|_{\widehat E}\sim a$. If $A_1,A_2,A_3$ and $B_1,B_2,B_3$ are the exceptional curves of $\beta$, then $T_1|_{\widehat E}\sim \sum_{j=1}^3 A_j$ and $T_2|_{\widehat E}\sim \sum_{j=1}^3 B_j$, since $T_1$ and $T_2$ meet $\widehat E$ exactly along the exceptional curves over the points where $\widehat\Gamma_1$ and $\widehat\Gamma_2$ meet $E_\pi$. Moreover $(H-\widehat E)|_{\widehat E}\sim a$ and $H|_{\widehat E}\sim 2b$, so $\widehat E|_{\widehat E}\sim 2b-a$.

By Proposition~\ref{prop:classes-Vhat}, one has $G_i\sim H-\widehat E-T_i$. Restricting to $\widehat E$, we obtain $G_1|_{\widehat E}\sim a-\sum_{j=1}^3 A_j$ and $G_2|_{\widehat E}\sim a-\sum_{j=1}^3 B_j$. Similarly, $-K_{\widehat V}|_{\widehat E}\sim 3H|_{\widehat E}-\widehat E|_{\widehat E}-T_1|_{\widehat E}-T_2|_{\widehat E}\sim a+4b-\sum_{j=1}^3 A_j-\sum_{j=1}^3 B_j$.

Set $\ell_j:=b-A_j$ and $m_j:=b-B_j$ for $j=1,2,3$. We claim that these are six pairwise disjoint $(-1)$-curves on $\widehat E$. Indeed, on $E_\pi\simeq \PP^1\times \PP^1$, the two sections $E_\pi\cap \widehat Q_1$ and $E_\pi\cap \widehat Q_2$ have class $a$, since their class is $(H-E_\pi)|_{E_\pi}$. The points blown up to obtain the $A_j$ lie on the first such section, and the points blown up to obtain the $B_j$ lie on the second. Since these two sections are disjoint and the six points are distinct, they lie on six distinct fibers of the $b$-ruling. Their strict transforms are precisely $\ell_j$ and $m_j$. Thus $\ell_j^2=(b-A_j)^2=-1$ and $m_j^2=(b-B_j)^2=-1$, and the six curves are pairwise disjoint smooth rational curves. Hence Castelnuovo's criterion gives a morphism $\eta\colon \widehat E\to \Sigma$ contracting exactly these six curves, where $\Sigma$ is smooth.

Let $S_A:=a-\sum_{j=1}^3 A_j=G_1|_{\widehat E}$ and $S_B:=a-\sum_{j=1}^3 B_j=G_2|_{\widehat E}$. These are the strict transforms of the two disjoint sections of class $a$ on $E_\pi$, hence $S_A^2=S_B^2=-3$. Moreover $S_A$ is disjoint from each $\ell_j$ and meets each $m_j$ transversely in one point, while $S_B$ is disjoint from each $m_j$ and meets each $\ell_j$ transversely in one point. Therefore, after contracting the six curves, their images $\overline S_A:=\eta(S_A)$ and $\overline S_B:=\eta(S_B)$ are disjoint smooth rational curves with $\overline S_A^2=\overline S_B^2=0$. Hence $\Sigma$ is a Hirzebruch surface $\F_n$ containing two disjoint sections of self-intersection $0$, so $n=0$. Thus $\Sigma\simeq \PP^1\times\PP^1$.

We identify $\Sigma$ with $\PP^1\times\PP^1$. The curve $\overline S_A$ is one ruling on $\Sigma$, and its total transform is $\eta^*\overline S_A=S_A+m_1+m_2+m_3$. Therefore $\eta^*\OO_{\PP^1\times\PP^1}(1,0)\sim S_A+m_1+m_2+m_3\sim a+3b-\sum_{j=1}^3 A_j-\sum_{j=1}^3 B_j$. Similarly, a general fiber of the $b$-ruling on $E_\pi$ avoids the six blown-up points, so its strict transform on $\widehat E$ is still of class $b$, and after the contraction $\eta$ it becomes a fiber of the second ruling on $\PP^1\times\PP^1$. Hence $\eta^*\OO_{\PP^1\times\PP^1}(0,1)\sim b$. Finally, $-K_{\widehat V}|_{\widehat E}\sim a+4b-\sum_{j=1}^3 A_j-\sum_{j=1}^3 B_j\sim \eta^*\OO_{\PP^1\times\PP^1}(1,0)+\eta^*\OO_{\PP^1\times\PP^1}(0,1)$, so $-K_{\widehat V}|_{\widehat E}\sim \eta^*\OO_{\PP^1\times\PP^1}(1,1)$.
\end{proof}

\begin{Proposition}\label{prop:Vhat-weak-fano}
The threefold $\widehat V$ is weak Fano.
\end{Proposition}

\begin{proof}
We prove first that $-K_{\widehat V}$ is nef. Let $B\subset \widehat V$ be an irreducible curve. If $B\subset G_i$ for some $i$, then the claim follows from Corollary~\ref{cor:rulings}: the only curves in $G_i$ on which $-K_{\widehat V}$ has degree $0$ are the fibers of the ruling $r_i$, and on every irreducible curve in $G_i$ one has $-K_{\widehat V}\cdot B\geq 0$.

If $B\subset T_i$ for some $i$, then by Lemma~\ref{lem:Ti} one has $\NEbar(T_i)=\RR_{\geq 0}[S_i]+\RR_{\geq 0}[f_i]$, and by Proposition~\ref{prop:restrictions-TE}, $-K_{\widehat V}|_{T_i}\sim S_i+6f_i$. Using $S_i^2=-4$, $S_i\cdot f_i=1$, and $f_i^2=0$ on $\F_4$, we get $(S_i+6f_i)\cdot f_i=1$ and $(S_i+6f_i)\cdot S_i=2$. Hence $-K_{\widehat V}|_{T_i}$ is nef, so $-K_{\widehat V}\cdot B\geq 0$ for every $B\subset T_i$.

If $B\subset \widehat E$, then Proposition~\ref{prop:restrictions-TE} gives $-K_{\widehat V}|_{\widehat E}\sim \eta^*\OO_{\PP^1\times\PP^1}(1,1)$. Since $\OO_{\PP^1\times\PP^1}(1,1)$ is ample, its pullback is nef. Therefore $-K_{\widehat V}\cdot B\geq 0$ for every $B\subset \widehat E$, and equality holds exactly on the six curves contracted by $\eta$, namely $\ell_1,\ell_2,\ell_3,m_1,m_2,m_3$.

It remains to consider the case where $B\not\subset G_1\cup G_2\cup T_1\cup T_2\cup \widehat E$. Set $d:=H\cdot B$, $m_C:=\widehat E\cdot B$, and $m_i:=T_i\cdot B$ for $i=1,2$. Since $H$ is the pullback to $\widehat V$ of the hyperplane class of $Q$, it is nef, hence $d\geq 0$. Moreover, since $B$ is not contained in $\widehat E,T_1,T_2,G_1,G_2$, the intersections with these effective Cartier divisors are nonnegative: $m_C\geq 0$, $m_1\geq 0$, $m_2\geq 0$, $G_1\cdot B\geq 0$, and $G_2\cdot B\geq 0$. By Proposition~\ref{prop:classes-Vhat}, $G_i\sim H-\widehat E-T_i$, so $G_i\cdot B=d-m_C-m_i\geq 0$ for $i=1,2$. Using again Proposition~\ref{prop:classes-Vhat}, $-K_{\widehat V}\sim 3H-\widehat E-T_1-T_2$, hence
$$
        -K_{\widehat V}\cdot B
        =d+G_1\cdot B+G_2\cdot B+m_C\geq 0.
$$
Therefore $-K_{\widehat V}$ is nef.

We now prove that $-K_{\widehat V}$ is big by showing $(-K_{\widehat V})^3>0$. First consider $\pi\colon \widehat Q=\Bl_C(Q)\to Q$. Let $H_Q$ be the hyperplane class on $Q$, and still denote by $H$ its pullback to $\widehat Q$. Since $Q$ is a smooth quadric threefold, $-K_Q\sim 3H_Q$ and $H_Q^3=2$. By Lemma~\ref{lem:Ehat}, $N_{C/Q}\cong \OO_{\PP^1}(2)\oplus \OO_{\PP^1}(2)$, so $\deg N_{C/Q}=4$. For the blow-up of a smooth threefold along a smooth curve, the standard intersection formulas give $H^2\cdot E_\pi=0$, $H\cdot E_\pi^2=-H_Q\cdot C$, and $E_\pi^3=-\deg N_{C/Q}$. Since $C\subset Q$ is a conic, $H_Q\cdot C=2$. Therefore $H^2E_\pi=0$, $HE_\pi^2=-2$, and $E_\pi^3=-4$. Since $-K_{\widehat Q}\sim 3H-E_\pi$, we obtain $(-K_{\widehat Q})^3=(3H-E_\pi)^3=27\cdot 2+9(-2)-(-4)=40$.

Next we blow up the two disjoint curves $\widehat\Gamma_1$ and $\widehat\Gamma_2$. By Lemma~\ref{lem:Ehat}, each $\widehat\Gamma_i$ meets $E_\pi$ transversely in three points, hence $E_\pi\cdot\widehat\Gamma_i=3$. Since $\Gamma_i$ is a twisted cubic, $H\cdot\widehat\Gamma_i=3$. Thus $-K_{\widehat Q}\cdot \widehat\Gamma_i=(3H-E_\pi)\cdot \widehat\Gamma_i=6$. Moreover, by Lemma~\ref{lem:Ti}, $N_{\widehat\Gamma_i/\widehat Q}\cong \OO_{\PP^1}(4)\oplus \OO_{\PP^1}$, so $\deg N_{\widehat\Gamma_i/\widehat Q}=4$.

Let $\beta_1\colon \widehat V_1:=\Bl_{\widehat\Gamma_1}(\widehat Q)\to \widehat Q$ be the first blow-up, with exceptional divisor still denoted by $T_1$. Then $-K_{\widehat V_1}\sim \beta_1^*(-K_{\widehat Q})-T_1$, and the usual blow-up formulas yield
$$
        (-K_{\widehat V_1})^3
        =(-K_{\widehat Q})^3-3((-K_{\widehat Q})\cdot\widehat\Gamma_1)+\deg N_{\widehat\Gamma_1/\widehat Q}
        =40-3\cdot 6+4=26.
$$
Since $\widehat\Gamma_1$ and $\widehat\Gamma_2$ are disjoint, the strict transform of $\widehat\Gamma_2$ on $\widehat V_1$ is still isomorphic to $\widehat\Gamma_2$, has the same normal bundle, and has anticanonical degree $6$ with respect to $-K_{\widehat V_1}$. Blowing it up, we get $\widehat V=\Bl_{\widehat\Gamma_2}(\widehat V_1)$ and similarly $(-K_{\widehat V})^3=26-3\cdot 6+4=12$. Hence $(-K_{\widehat V})^3=12>0$. Thus $-K_{\widehat V}$ is nef and big, so $\widehat V$ is weak Fano.
\end{proof}

\begin{Proposition}\label{prop:anticanonical-contraction}
Let $\mu\colon \widehat V\to V^a:=\Proj \bigoplus_{m\geq 0} H^0(\widehat V,-mK_{\widehat V})$ be the anticanonical contraction. The exceptional locus of $\mu$ is $G_1\cup G_2\cup \bigcup_{j=1}^3(\ell_j\cup m_j)$. More precisely, $\mu$ contracts each $G_i$ onto a rational curve $\Gamma_i^a\subset V^a$, and it contracts each of the six curves $\ell_j,m_j\subset \widehat E$ to a point.
\end{Proposition}

\begin{proof}
By Proposition~\ref{prop:Vhat-weak-fano}, the divisor $-K_{\widehat V}$ is nef and big. Hence, by the Base Point Free Theorem, it is semiample. Therefore there exist an integer $m>0$ and an ample divisor $A$ on $V^a$ such that $-mK_{\widehat V}\sim \mu^*A$. Thus, for every irreducible curve $B\subset \widehat V$, the curve $B$ is contracted by $\mu$ if and only if $(-K_{\widehat V})\cdot B=0$.

It is enough to determine all irreducible curves on $\widehat V$ with anticanonical degree zero. Let $B\subset \widehat V$ be an irreducible curve. If $B\subset G_i$ for some $i$, then Corollary~\ref{cor:rulings} shows that $-K_{\widehat V}\cdot B=0$ if and only if $B$ is a fiber of the ruling on $G_i$ of class $r_i$.

If $B\subset T_i$ for some $i$, then by Lemma~\ref{lem:Ti}, one has $\NEbar(T_i)=\RR_{\geq 0}[S_i]+\RR_{\geq 0}[f_i]$, and by Proposition~\ref{prop:restrictions-TE}, $-K_{\widehat V}|_{T_i}\sim S_i+6f_i$. Hence $(-K_{\widehat V}|_{T_i})\cdot f_i=1$ and $(-K_{\widehat V}|_{T_i})\cdot S_i=2$, so $-K_{\widehat V}|_{T_i}$ is strictly positive on every nonzero class in $\NEbar(T_i)$. Therefore $-K_{\widehat V}\cdot B>0$ for every irreducible curve $B\subset T_i$.

If $B\subset \widehat E$, then Proposition~\ref{prop:restrictions-TE} gives $-K_{\widehat V}|_{\widehat E}\sim \eta^*\OO_{\PP^1\times\PP^1}(1,1)$. Since $\OO_{\PP^1\times\PP^1}(1,1)$ is ample, its pullback is nef, and its degree is zero exactly on the curves contracted by $\eta$. Therefore $-K_{\widehat V}\cdot B=0$ if and only if $B\in\{\ell_1,\ell_2,\ell_3,m_1,m_2,m_3\}$.

It remains to consider the case where $B\not\subset G_1\cup G_2\cup T_1\cup T_2\cup \widehat E$. Set $d:=H\cdot B$, $m_C:=\widehat E\cdot B$, and $m_i:=T_i\cdot B$ for $i=1,2$. Since $H$ is the pullback to $\widehat V$ of the hyperplane class on $Q$, it is nef. Moreover, since $B$ is not contained in the exceptional locus of $\widehat V\to Q$, the curve $B$ is not contracted by $\widehat V\to Q$, hence $d>0$. Also, since $B$ is not contained in $\widehat E,T_1,T_2,G_1,G_2$, we have $m_C\geq 0$, $m_1\geq 0$, $m_2\geq 0$, $G_1\cdot B\geq 0$, and $G_2\cdot B\geq 0$. By Proposition~\ref{prop:classes-Vhat}, $G_i\sim H-\widehat E-T_i$, so $G_i\cdot B=d-m_C-m_i\geq 0$ for $i=1,2$. Again by Proposition~\ref{prop:classes-Vhat}, $-K_{\widehat V}\sim 3H-\widehat E-T_1-T_2$, hence $
        -K_{\widehat V}\cdot B
        =d+G_1\cdot B+G_2\cdot B+m_C>0$.

We have proved that the irreducible curves contracted by $\mu$ are exactly the fibers of the ruling of class $r_i$ on each $G_i$, together with the six curves $\ell_1,\ell_2,\ell_3,m_1,m_2,m_3$. In particular, $\Exc(\mu)=G_1\cup G_2\cup \bigcup_{j=1}^3(\ell_j\cup m_j)$.

It remains to describe the images of these components. Fix $i\in\{1,2\}$, and let $p_i\colon G_i\simeq \PP^1\times\PP^1\to \PP^1$ be the ruling whose fibers are the curves of class $r_i$. Since $\mu$ contracts precisely these fibers on $G_i$, the Rigidity Lemma implies that $\mu|_{G_i}$ factors through $p_i$, say $\mu|_{G_i}=\nu_i\circ p_i$ for some morphism $\nu_i\colon \PP^1\to V^a$. The image is not a point, since a curve of the other ruling on $G_i$ has positive anticanonical degree. Hence $\mu(G_i)$ is an irreducible rational curve. We denote it by $\Gamma_i^a:=\mu(G_i)\subset V^a$.

Finally, each of the six curves $\ell_j,m_j$ has anticanonical degree zero, hence is contracted by $\mu$. Since each of them is irreducible and proper, its image is a point.
\end{proof}

\begin{Construction}\label{constr:Y}
Fix $0<\varepsilon\ll 1$. Run the relative $(K_{\widehat V}+\varepsilon G_1+\varepsilon G_2)$-MMP over $V^a$. We denote by $\nu\colon \widehat V\dashrightarrow Y$ the resulting birational map and by $\tau\colon Y\to V^a$ the induced morphism. Thus $\mu=\tau\circ\nu$ as rational maps.
\end{Construction}

\begin{Proposition}\label{prop:Y-model}
The only curves that are $(K_{\widehat V}+\varepsilon G_1+\varepsilon G_2)$-negative over $V^a$ are the fibers of the rulings on $G_1$ and $G_2$. Consequently, $\nu$ contracts exactly $G_1$ and $G_2$ divisorially onto rational curves $\Lambda_1,\Lambda_2\subset Y$. Moreover, $Y$ is normal, projective, and $\QQ$-factorial, while $-K_Y$ is nef and big. If $\overline\ell_j,\overline m_j\subset Y$ denote the images of $\ell_j,m_j$, then these six curves are pairwise disjoint, $K_Y$-trivial, and contracted by $\tau$.
\end{Proposition}

\begin{proof}
We work over $V^a$. Since $\mu\colon \widehat V\to V^a$ is the anticanonical contraction, an irreducible curve on $\widehat V$ is contracted by $\mu$ if and only if it has anticanonical degree $0$. By Proposition~\ref{prop:anticanonical-contraction}, the $\mu$-contracted curves are precisely the fibers $r_i\subset G_i$ of the ruling of class $(1,0)$ on $G_i$, for $i=1,2$, together with the six curves $\ell_1,\ell_2,\ell_3,m_1,m_2,m_3\subset \widehat E$.

We first determine which of these curves are $(K_{\widehat V}+\varepsilon G_1+\varepsilon G_2)$-negative. If $C=r_i\subset G_i$, then by Corollary~\ref{cor:rulings}, one has $K_{\widehat V}\cdot r_i=0$, $G_i\cdot r_i=-2$, and $G_j\cdot r_i=0$ for $j\neq i$, since $G_1\cap G_2=\varnothing$. Hence $(K_{\widehat V}+\varepsilon G_1+\varepsilon G_2)\cdot r_i=-2\varepsilon<0$.

Now let $C=\ell_j$. By Proposition~\ref{prop:restrictions-TE}, one has $G_1|_{\widehat E}\sim a-\sum_{k=1}^3 A_k$, $G_2|_{\widehat E}\sim a-\sum_{k=1}^3 B_k$, and $\ell_j\sim b-A_j$. Using the intersection pairing on the blow-up of $\PP^1\times\PP^1$ at the six points, we get $G_1\cdot \ell_j=(a-\sum_{k=1}^3 A_k)\cdot(b-A_j)=0$ and $G_2\cdot \ell_j=(a-\sum_{k=1}^3 B_k)\cdot(b-A_j)=1$. Since $\ell_j$ is $\mu$-contracted, we have $K_{\widehat V}\cdot\ell_j=0$. Therefore $(K_{\widehat V}+\varepsilon G_1+\varepsilon G_2)\cdot \ell_j=\varepsilon>0$. Similarly, for $C=m_j$, one has $G_1\cdot m_j=1$, $G_2\cdot m_j=0$, and $K_{\widehat V}\cdot m_j=0$, hence $(K_{\widehat V}+\varepsilon G_1+\varepsilon G_2)\cdot m_j=\varepsilon>0$.

Thus the only $\mu$-contracted curves on which $K_{\widehat V}+\varepsilon G_1+\varepsilon G_2$ is negative are the fibers of the two rulings on $G_1$ and $G_2$. Hence the only $(K_{\widehat V}+\varepsilon G_1+\varepsilon G_2)$-negative extremal rays in $\NEbar(\widehat V/V^a)$ are the rays generated by these fibers.

Since these curves cover $G_1$ and $G_2$, the corresponding extremal contractions are divisorial and contract $G_1$ and $G_2$ onto curves. Moreover $G_1$ and $G_2$ are disjoint, so the two contractions do not interfere with each other. Therefore the relative $(K_{\widehat V}+\varepsilon G_1+\varepsilon G_2)$-MMP over $V^a$ consists of the two divisorial contractions of $G_1$ and $G_2$, and no flips occur. Thus $\nu\colon \widehat V\to Y$ is a morphism contracting exactly $G_1$ and $G_2$.

Let $\Lambda_i:=\nu(G_i)\subset Y$. Since $G_i\simeq \PP^1\times\PP^1$ and the contracted curves on $G_i$ are precisely the fibers of the ruling $r_i$, the restriction $\nu|_{G_i}$ factors through the corresponding projection $p_i\colon G_i\simeq \PP^1\times\PP^1\to \PP^1$. The image is not a point, since a curve of the other ruling has positive anticanonical degree. Hence $\Lambda_i$ is an irreducible rational curve.

Since $\widehat V$ is smooth and $(\widehat V,\varepsilon G_1+\varepsilon G_2)$ is klt, standard properties of the relative MMP imply that $Y$ is normal, projective, and $\QQ$-factorial.

We next show that $\nu$ is crepant. Since $\nu$ is a birational morphism whose exceptional divisors are exactly $G_1$ and $G_2$, we may write $K_{\widehat V}=\nu^*K_Y+a_1G_1+a_2G_2$ for some $a_1,a_2\in\QQ$. Intersecting with a ruling fiber $r_i\subset G_i$, and using $K_{\widehat V}\cdot r_i=0$, $\nu^*K_Y\cdot r_i=0$, $G_i\cdot r_i=-2$, and $G_j\cdot r_i=0$ for $j\neq i$, we get $0=-2a_i$. Thus $a_i=0$ for $i=1,2$, and therefore $K_{\widehat V}=\nu^*K_Y$.

Let $A$ be an ample $\QQ$-Cartier divisor on $V^a$ such that $-K_{\widehat V}\sim_{\QQ}\mu^*A$. Since $\mu=\tau\circ\nu$ and $K_{\widehat V}=\nu^*K_Y$, we have $\nu^*(\tau^*A)\sim_{\QQ}\nu^*(-K_Y)$. Pushing forward by $\nu$, it follows that $-K_Y\sim_{\QQ}\tau^*A$. Hence $-K_Y$ is nef. Since $\tau$ is birational and $A$ is ample, $\tau^*A$ is big; therefore $-K_Y$ is big.

Now let $\overline\ell_j:=\nu(\ell_j)$ and $\overline m_j:=\nu(m_j)$. These are curves, since $\ell_j$ and $m_j$ are not contracted by $\nu$. Since $\nu$ is crepant and is an isomorphism at the generic point of each of these curves, the projection formula gives $K_Y\cdot \overline\ell_j=K_{\widehat V}\cdot \ell_j=0$ and $K_Y\cdot \overline m_j=K_{\widehat V}\cdot m_j=0$. Thus all six curves are $K_Y$-trivial.

It remains to prove that they are pairwise disjoint on $Y$. First consider $\overline\ell_1,\overline\ell_2,\overline\ell_3$. The curves $\ell_1,\ell_2,\ell_3$ are pairwise disjoint on $\widehat V$, and each of them is disjoint from $G_1$. Hence the contraction of $G_1$ does not affect them. Each $\ell_j$ meets $G_2$ transversely in exactly one point, since $G_2\cdot \ell_j=1$. Moreover, on $G_2\simeq \PP^1\times\PP^1$, one has $\widehat E|_{G_2}\sim(1,1)$ by Proposition~\ref{prop:restrictions-Gi}, and the contraction of $G_2$ is the projection $p_2\colon G_2\to \PP^1$ with fibers of class $(1,0)$. The restriction of $p_2$ to $\widehat E\cap G_2$ has degree $1$. The three points $\ell_j\cap G_2$ lie on three distinct fibers of this projection, hence they have distinct images on $\Lambda_2$. Therefore $\overline\ell_1,\overline\ell_2,\overline\ell_3$ remain pairwise disjoint on $Y$.

The same argument, exchanging the roles of $G_1$ and $G_2$, shows that $\overline m_1,\overline m_2,\overline m_3$ are pairwise disjoint. Finally, each $\ell_j$ is disjoint from every $m_k$ on $\widehat V$. Moreover, each $\ell_j$ is disjoint from $G_1$ and meets only $G_2$, while each $m_k$ is disjoint from $G_2$ and meets only $G_1$. Since $G_1$ and $G_2$ are disjoint, the two divisorial contractions can be performed successively, and each step is an isomorphism in a neighborhood of the curves belonging to the other family. Hence no new intersections are created between the images of the $\ell_j$'s and the $m_k$'s. Therefore every $\overline\ell_j$ is disjoint from every $\overline m_k$.

Lastly, since $\mu=\tau\circ\nu$, and the curves $\ell_j,m_j$ are contracted by $\mu$ but not by $\nu$, their images $\overline\ell_j,\overline m_j\subset Y$ are contracted by $\tau$.
\end{proof}

\begin{Notation}\label{not:MY-EY}
Set $M:=H-\widehat E$, and let $\mathcal M:=|M|=|H-\widehat E|$ be the pencil introduced in Proposition~\ref{prop:classes-Vhat}. We denote by $E_Y\subset Y$ the birational transform of $\widehat E$. Since $M$ descends to $Y$, we denote the descended divisor by $M_Y$.
\end{Notation}

\begin{Lemma}\label{lem:pullbackEY}
We have $\nu^*M_Y=M$ and $\nu^*E_Y=\widehat E+\frac12G_1+\frac12G_2$.
\end{Lemma}

\begin{proof}
By Proposition~\ref{prop:Y-model}, the morphism $\nu\colon \widehat V\to Y$ contracts exactly the divisors $G_1$ and $G_2$, and on each $G_i\simeq \PP^1\times\PP^1$ the contracted curves are precisely the fibers $r_i$ of the ruling of class $(1,0)$.

We first prove that $M$ descends to $Y$. Since $M\sim H-\widehat E$, Proposition~\ref{prop:restrictions-Gi} gives $M|_{G_i}\sim (H-\widehat E)|_{G_i}\sim (1,1)-(1,1)\sim 0$ for $i=1,2$. Thus $\OO_{\widehat V}(M)$ is trivial on every fiber of $\nu$. Since $\nu$ is a proper morphism with connected fibers and $Y$ is normal, the descent theorem for line bundles gives a line bundle on $Y$ whose pullback is $\OO_{\widehat V}(M)$. By definition this line bundle is $\OO_Y(M_Y)$, and therefore $\nu^*M_Y=M$.

Now let $E_Y\subset Y$ be the birational transform of $\widehat E$. Since $Y$ is $\QQ$-factorial by Proposition~\ref{prop:Y-model}, the divisor $E_Y$ is $\QQ$-Cartier. The morphism $\nu$ is an isomorphism at the generic point of $\widehat E$, and its exceptional divisors are exactly $G_1$ and $G_2$. Thus there exist $a_1,a_2\in\QQ$ such that $\nu^*E_Y=\widehat E+a_1G_1+a_2G_2$.

Fix $i\in\{1,2\}$ and intersect with a ruling fiber $r_i\subset G_i$. Since $r_i$ is contracted by $\nu$, we have $\nu^*E_Y\cdot r_i=0$. Moreover, by Corollary~\ref{cor:rulings}, $G_i\cdot r_i=-2$ and $G_j\cdot r_i=0$ for $j\neq i$, while Proposition~\ref{prop:restrictions-Gi} gives $\widehat E\cdot r_i=\widehat E|_{G_i}\cdot r_i=(1,1)\cdot(1,0)=1$. Therefore $0=(\widehat E+a_iG_i)\cdot r_i=1-2a_i$, so $a_i=\frac12$. Since this holds for $i=1,2$, we conclude that $\nu^*E_Y=\widehat E+\frac12G_1+\frac12G_2$.
\end{proof}

\begin{Proposition}\label{prop:intersections-Y}
On $Y$ we have $-K_Y\sim_{\QQ} M_Y+2E_Y$, $E_Y^3=-\frac32$, $M_Y\cdot E_Y^2=2$, and $M_Y^3=M_Y^2\cdot E_Y=0$.
\end{Proposition}

\begin{proof}
By Proposition~\ref{prop:classes-Vhat} and Lemma~\ref{lem:pullbackEY}, we have $-K_{\widehat V}\sim M+2\widehat E+G_1+G_2\sim \nu^*M_Y+2(\nu^*E_Y-\frac12G_1-\frac12G_2)+G_1+G_2=\nu^*(M_Y+2E_Y)$. On the other hand, by Proposition~\ref{prop:Y-model}, the morphism $\nu$ is crepant, so $K_{\widehat V}=\nu^*K_Y$. Hence $\nu^*(-K_Y)=\nu^*(M_Y+2E_Y)$. Pushing forward via $\nu$, we obtain $-K_Y\sim_{\QQ}M_Y+2E_Y$.

We now compute the intersection numbers. By Lemma~\ref{lem:pullbackEY}, we have $\nu^*E_Y=\widehat E+\frac12G_1+\frac12G_2=H-\frac12T_1-\frac12T_2$, where the last equality follows from Proposition~\ref{prop:classes-Vhat}, since $G_i\sim H-\widehat E-T_i$. Moreover $\nu^*M_Y=M=H-\widehat E$.

Since $\alpha\colon \widehat V\to \widehat Q$ is the blow-up of the smooth threefold $\widehat Q$ along the two disjoint smooth curves $\widehat\Gamma_1,\widehat\Gamma_2$, the standard blow-up formulas give, for each $i=1,2$, $H^2\cdot T_i=H\cdot\widehat E\cdot T_i=\widehat E^2\cdot T_i=0$, $T_1\cdot T_2=0$, $H\cdot T_i^2=-H\cdot\widehat\Gamma_i=-3$, $\widehat E\cdot T_i^2=-\widehat E\cdot\widehat\Gamma_i=-3$, and $T_i^3=-\deg N_{\widehat\Gamma_i/\widehat Q}=-4$. Here we used Lemma~\ref{lem:Ti} and Lemma~\ref{lem:Ehat}. We also have $H^3=2$ and $H^2\cdot\widehat E=0$. Therefore
$$
E_Y^3=(\nu^*E_Y)^3=\left(H-\frac12T_1-\frac12T_2\right)^3
=2+\frac34(-3-3)-\frac18(-4-4)=-\frac32.
$$

Similarly,
$$
M_Y\cdot E_Y^2
=\nu^*M_Y\cdot(\nu^*E_Y)^2
=(H-\widehat E)\cdot\left(H-\frac12T_1-\frac12T_2\right)^2
=2.
$$

It remains to prove that $M_Y^3=M_Y^2\cdot E_Y=0$. By Proposition~\ref{prop:classes-Vhat}, the linear system $|M|$ is a base-point-free pencil on $\widehat V$. Since $\nu^*M_Y=M$ by Lemma~\ref{lem:pullbackEY}, and since $\nu_*\OO_{\widehat V}=\OO_Y$, the pullback map gives $H^0(Y,\OO_Y(M_Y))\cong H^0(\widehat V,\OO_{\widehat V}(M))$. Hence $|M_Y|$ is again a base-point-free pencil. Let $\pi_Y\colon Y\to \PP^1$ be the associated morphism. Then $M_Y\sim \pi_Y^*\OO_{\PP^1}(1)$, so $M_Y^2\equiv 0$. Therefore $M_Y^3=0$ and $M_Y^2\cdot E_Y=0$.
\end{proof}

\begin{Corollary}\label{cor:Y-weak-fano}
The variety $Y$ is weak Fano.
\end{Corollary}

\begin{proof}
By Proposition~\ref{prop:Y-model}, the variety $Y$ is normal, projective, and $\QQ$-factorial, and the divisor $-K_Y$ is nef and big. Hence $Y$ is weak Fano.
\end{proof}

We summarize the birational constructions introduced above in Figure~\ref{fig:global-birational-picture}. It collects in a single picture the two routes starting from the smooth quadric threefold $Q$: on the left, the construction leading to $X$ and then to $Z=\Bl_l(X)$, and on the right, the construction leading to the model $Y$. The figure also records the divisors and curves used later in the birational comparison between $Z$ and $Y$.

\begin{figure}[htbp]
\centering
\scalebox{0.66}{
\begin{tikzpicture}[
 x=1cm,y=1cm,
 line cap=round,
 line join=round,
 >=Latex,
 var/.style={draw, rounded corners=3pt, minimum width=4.8cm, minimum height=2.8cm, inner sep=2pt},
 map/.style={-Latex, thick},
 flop/.style={dashed, {Latex[length=2mm]}-{Latex[length=2mm]}, thick},
 bir/.style={dashed, {Latex[length=2mm]}-{Latex[length=2mm]}, thick},
 lab/.style={font=\scriptsize},
 smalllab/.style={font=\tiny},
 divA/.style={draw=blue!70!black, fill=blue!8},
 divB/.style={draw=teal!70!black, fill=teal!10},
 surf/.style={draw=violet!70!black, fill=violet!7},
 curve/.style={red!80!black, line width=.9pt},
 point/.style={circle, fill=black, inner sep=1.2pt}
]

\node[var] (X) at (-5.8, 7.2) {};
\node[var] (Z) at ( 0.0, 7.2) {};
\node[var] (Y) at ( 5.8, 7.2) {};

\node[var] (Vmz) at (-5.8, 2.9) {};
\node[var] (Vhat) at ( 5.8, 2.9) {};

\node[var] (W) at (-5.8, -1.4) {};
\node[var] (Q) at ( 0.0, -1.4) {};
\node[var] (Qhat) at ( 5.8, -1.4) {};

\node at ([yshift=-0.17cm]X.north) {\small $X$};
\node at ([yshift=-0.17cm]Z.north) {\small $Z=\Bl_l(X)$};
\node at ([yshift=-0.17cm]Y.north) {\small $Y$};

\node at ([yshift=-0.17cm]Vmz.north) {\small $V$};
\node at ([yshift=-0.31cm]Vhat.north) {\small $\widehat V=\Bl_{\widehat\Gamma_1\sqcup\widehat\Gamma_2}(\widehat Q)$};

\node at ([yshift=-0.17cm]W.north) {\small $W=\Bl_{\Gamma_1\sqcup\Gamma_2}(Q)$};
\node at ([yshift=-0.17cm]Q.north) {\small $Q$};
\node at ([yshift=-0.27cm]Qhat.north) {\small $\widehat Q=\Bl_C(Q)$};

\draw[map] (Z.west) -- node[above,lab] {$\Bl_l$} (X.east);
\draw[bir] (Z.east) -- node[above,lab] {$\chi$} (Y.west);

\draw[map] (Vmz.north) -- node[left,lab] {contract $F_1,F_2$} (X.south);
\draw[map] (Vhat.north) -- node[right,lab] {contract $G_1,G_2$} (Y.south);

\draw[flop] (W.north) -- node[left,lab] {flops of $\widetilde C,\widetilde L_i,\widetilde R_i$} (Vmz.south);
\draw[map] (Vhat.south) -- node[right,lab] {$\Bl_{\widehat\Gamma_1\sqcup\widehat\Gamma_2}$} (Qhat.north);

\draw[map] (W.east) -- node[above,lab] {$\Bl_{\Gamma_1\sqcup\Gamma_2}$} (Q.west);
\draw[map] (Qhat.west) -- node[above,lab] {$\Bl_C$} (Q.east);

\begin{scope}[shift={(X.center)}]
 \node[point,label={[lab]below:$x_1$}] at (-1.10,0) {};
 \node[point,label={[lab]below:$x_2$}] at ( 1.10,0) {};

 \draw[curve] (-1.10,0) .. controls (0,0.82) .. (1.10,0);
 \node[lab] at (0,0.92) {$l$};

 \draw[curve] (-1.10,0) -- (-1.95, 0.45);
 \draw[curve] (-1.10,0) -- (-2.00, 0.00);
 \draw[curve] (-1.10,0) -- (-1.92,-0.42);

 \draw[curve] (1.10,0) -- (1.95, 0.45);
 \draw[curve] (1.10,0) -- (2.00, 0.00);
 \draw[curve] (1.10,0) -- (1.92,-0.42);

 \node[smalllab] at (-1.55,-0.66) {images of $R_i$};
 \node[smalllab] at ( 1.55,-0.66) {images of $L_i$};
\end{scope}

\begin{scope}[shift={(Z.center)}]
 \filldraw[divB] (-0.88,-0.82) rectangle (0.88,0.82);
 \node[lab] at (0,0) {$E$};

 \draw[curve] (-0.66,0.66) -- (-0.66,-0.66);
 \draw[curve] ( 0.66,0.66) -- ( 0.66,-0.66);
 \node[smalllab] at (-0.66,0.90) {$\Sigma_1$};
 \node[smalllab] at ( 0.66,0.90) {$\Sigma_2$};

 \draw[curve] (-1.82, 0.42) -- (-0.88, 0.42);
 \draw[curve] (-1.82, 0.00) -- (-0.88, 0.00);
 \draw[curve] (-1.82,-0.38) -- (-0.88,-0.38);

 \draw[curve] (1.82, 0.42) -- (0.88, 0.42);
 \draw[curve] (1.82, 0.00) -- (0.88, 0.00);
 \draw[curve] (1.82,-0.38) -- (0.88,-0.38);

 \node[smalllab] at (-1.68,-0.62) {$\widetilde R_1,\widetilde R_2,\widetilde R_3$};
 \node[smalllab] at ( 1.68,-0.62) {$\widetilde L_1,\widetilde L_2,\widetilde L_3$};
\end{scope}

\begin{scope}[shift={(Y.center)}]
 \filldraw[divA] (-1.95,-0.85) rectangle (-1.22,0.85);
 \filldraw[divA] ( 1.22,-0.85) rectangle ( 1.95,0.85);
 \node[rotate=90,lab] at (-1.58,0) {$E_1$};
 \node[rotate=90,lab] at ( 1.58,0) {$E_2$};

 \filldraw[divB] (-0.52,-0.85) rectangle (0.52,0.85);
 \node[rotate=90,lab] at (0,0) {$E_Y$};

 \draw[curve] (-0.25,0.62) -- (-0.25,-0.62);
 \draw[curve] ( 0.25,0.62) -- ( 0.25,-0.62);

 \node[lab] at (-0.25,0.73) {$\Lambda_1$};
 \node[lab] at ( 0.25,0.73) {$\Lambda_2$};
\end{scope}

\begin{scope}[shift={(Vmz.center)}]
 \filldraw[divA] (-1.95,-0.85) rectangle (-1.22,0.85);
 \filldraw[divA] ( 1.22,-0.85) rectangle ( 1.95,0.85);
 \node[rotate=90,lab] at (-1.58,0) {$E_1^V$};
 \node[rotate=90,lab] at ( 1.58,0) {$E_2^V$};

 \draw[surf] (-0.82,0.0) circle (0.46);
 \draw[surf] ( 0.82,0.0) circle (0.46);
 \node[lab] at (-0.82,0.0) {$F_1$};
 \node[lab] at ( 0.82,0.0) {$F_2$};

 \draw[curve] (-0.35,0.22) .. controls (0,0.78) .. (0.35,0.22);
 \node[lab] at (0,0.86) {$C_V$};
\end{scope}

\begin{scope}[shift={(Vhat.center)}]
 \filldraw[divA] (-1.95,-0.85) rectangle (-1.22,0.85);
 \filldraw[divA] ( 1.22,-0.85) rectangle ( 1.95,0.85);
 \node[rotate=90,lab] at (-1.58,0) {$T_1$};
 \node[rotate=90,lab] at ( 1.58,0) {$T_2$};

 \filldraw[divB] (-0.40,-0.85) rectangle (0.40,0.85);
 \node[rotate=90,lab] at (0,0) {$\widehat E$};

 \draw[surf] (-0.84,0.0) ellipse (0.58 and 0.40);
 \draw[surf] ( 0.84,0.0) ellipse (0.58 and 0.40);
 \node[lab] at (-0.84,0.0) {$G_1$};
 \node[lab] at ( 0.84,0.0) {$G_2$};

 \draw[curve] (-1.30,0.16)--(-0.40,0.21);
 \draw[curve] ( 0.40,0.21)--( 1.30,0.16);

 \node[smalllab] at (-0.82,0.61) {$\widehat E\cap G_1$};
 \node[smalllab] at ( 0.82,0.61) {$\widehat E\cap G_2$};
\end{scope}

\begin{scope}[shift={(W.center)}]
 \filldraw[divA] (-1.95,-0.85) rectangle (-1.22,0.85);
 \filldraw[divA] ( 1.22,-0.85) rectangle ( 1.95,0.85);
 \node[rotate=90,lab] at (-1.58,0) {$E_1^W$};
 \node[rotate=90,lab] at ( 1.58,0) {$E_2^W$};

 \draw[curve] (-0.55,0.60) .. controls (0,0.88) .. (0.55,0.60);
 \draw[curve] (-0.55,0.60) .. controls (0,0.30) .. (0.55,0.60);
 \node[lab] at (0,0.93) {$\widetilde C$};

 \draw[curve] (-1.00, 0.30)--(-0.32, 0.40);
 \draw[curve] (-1.00, 0.00)--(-0.32, 0.00);
 \draw[curve] (-1.00,-0.30)--(-0.32,-0.40);

 \draw[curve] (0.32, 0.40)--(1.00, 0.30);
 \draw[curve] (0.32, 0.00)--(1.00, 0.00);
 \draw[curve] (0.32,-0.40)--(1.00,-0.30);

 \node[smalllab] at (-0.88,-0.66) {$\widetilde R_i$};
 \node[smalllab] at ( 0.88,-0.66) {$\widetilde L_i$};
\end{scope}

\begin{scope}[shift={(Q.center)}]
 \draw[surf] (-1.15,0.03) ellipse (1.22 and 0.62);
 \draw[surf] ( 1.15,0.03) ellipse (1.22 and 0.62);

 \draw[curve] (-0.34,0.03) .. controls (0,0.42) .. (0.34,0.03);
 \draw[curve] (-0.34,0.03) .. controls (0,-0.15) .. (0.34,0.03);
 \node[lab] at (0,0.53) {$C$};

 \draw[curve] (-1.78,0.14) .. controls (-1.40,0.74) and (-0.94,-0.05) .. (-0.68,0.40);
 \draw[curve] ( 1.78,0.14) .. controls ( 1.40,0.74) and ( 0.94,-0.05) .. ( 0.68,0.40);

 \coordinate (p1) at (0.13,0.14);
 \coordinate (p2) at (0.20,0.03);
 \coordinate (p3) at (0.13,-0.07);

 \coordinate (q1) at (-0.13,0.14);
 \coordinate (q2) at (-0.20,0.03);
 \coordinate (q3) at (-0.13,-0.07);

 \draw[curve] (p1) -- (1.83,0.55);
 \draw[curve] (p2) -- (1.88,0.08);
 \draw[curve] (p3) -- (1.80,-0.28);

 \draw[curve] (q1) -- (-1.83,0.55);
 \draw[curve] (q2) -- (-1.88,0.08);
 \draw[curve] (q3) -- (-1.80,-0.28);

 \node[lab] at (-1.15,-0.81) {$Q_1$};
 \node[lab] at ( 1.15,-0.81) {$Q_2$};
 \node[lab] at (-1.45,0.82) {$\Gamma_1$};
 \node[lab] at ( 1.45,0.82) {$\Gamma_2$};
 \node[smalllab] at (-1.29,-0.40) {$R_1,R_2,R_3$};
 \node[smalllab] at ( 1.29,-0.40) {$L_1,L_2,L_3$};
\end{scope}

\begin{scope}[shift={(Qhat.center)}]
 \filldraw[divB] (-0.38,-0.85) rectangle (0.38,0.85);
 \node[rotate=90,lab] at (0,0) {$E_\pi$};

 \draw[surf] (-1.28,0.07) ellipse (0.84 and 0.46);
 \draw[surf] ( 1.28,0.07) ellipse (0.84 and 0.46);

 \draw[curve] (-1.74,0.08) .. controls (-1.44,0.64) and (-0.98,-0.03) .. (-0.78,0.35);
 \draw[curve] ( 1.74,0.08) .. controls ( 1.44,0.64) and ( 0.98,-0.03) .. ( 0.78,0.35);

 \node[lab] at (-1.28,-0.58) {$\widehat Q_1$};
 \node[lab] at ( 1.28,-0.58) {$\widehat Q_2$};
 \node[lab] at (-1.46,0.74) {$\widehat\Gamma_1$};
 \node[lab] at ( 1.46,0.74) {$\widehat\Gamma_2$};
\end{scope}

\end{tikzpicture}}
\caption{The two birational constructions starting from $Q$ and leading respectively to $Z=\Bl_l(X)$ and to $Y$.}
\label{fig:global-birational-picture}
\end{figure}

Figure~\ref{fig:global-birational-picture} summarizes the two birational routes starting from $Q$. The left-hand side gives $X$ and then $Z=\Bl_l(X)$; the right-hand side gives the auxiliary model $Y$. The dashed arrow $\chi\colon Z\dashrightarrow Y$ denotes the small birational comparison used below. The notation for the divisors and curves in the figure is the notation fixed in the preceding constructions.

\begin{Proposition}\label{prop:birational-input}
There is a small crepant birational map $\chi\colon Z\dashrightarrow Y$. Under the induced identification of divisor class groups in codimension $1$, the divisor $E$ corresponds to $E_Y$, and $f^*(-K_X)$ corresponds to $-K_Y+E_Y$.

Moreover, the indeterminacy locus of $\chi$ on $Z$ consists of the six pairwise disjoint curves
$\widetilde L_1,\widetilde L_2,\widetilde L_3,\widetilde R_1,\widetilde R_2,\widetilde R_3$,
and each of these curves is $K_Z$-trivial. The indeterminacy locus of $\chi^{-1}$ on $Y$ consists of the six pairwise disjoint curves
$\overline\ell_1,\overline\ell_2,\overline\ell_3,\overline m_1,\overline m_2,\overline m_3$,
where $\overline\ell_j:=\nu(\ell_j)$ and $\overline m_j:=\nu(m_j)$, and each of these curves is $K_Y$-trivial.
\end{Proposition}

\begin{proof}
We compare the two constructions starting from the same data $(Q,\Gamma_1,\Gamma_2)$. Let
\[
        g\colon \widetilde Z:=\Bl_{C_V}(V)\longrightarrow V
\]
be the blow-up of the flopped curve $C_V\subset V$, and let $E_V:=\Exc(g)$. The morphism $\phi\circ g$ factors through $f\colon Z=\Bl_l(X)\to X$; away from $x_1,x_2$ this follows from the universal property of the blow-up, and near $x_1,x_2$ from the local model of Lemma~\ref{lem:singularities-Z}. We denote the induced morphism by $\rho\colon \widetilde Z\to Z$.

We first identify the divisors corresponding to $E$ and $E_Y$. Since $\rho(V)=3$ and $\widetilde Z$ is the blow-up of $V$ along a smooth curve, $\rho(\widetilde Z)=4$, while $\rho(Z)=2$ by Proposition~\ref{prop:relative-cone-Z}. Hence $\rho(\widetilde Z/Z)=2$. The two $\rho$-exceptional divisors are the strict transforms $\widetilde F_1$ and $\widetilde F_2$, which are contracted onto $\Sigma_1$ and $\Sigma_2$. Thus $E_V$ is not $\rho$-exceptional. Since $g(E_V)=C_V$ and $\phi(C_V)=l$, one has $\rho(E_V)\subset f^{-1}(l)=E$; as $E$ is irreducible, $E_V$ maps dominantly onto $E$.

On the right-hand side, over the complement in $C$ of the six points $p_i,q_i$, blowing up $C\subset Q$ is the same as blowing up the corresponding curve $\widetilde C\subset W$. This is the common resolution of the flop of $\widetilde C$. On the left-hand side the same common resolution is the blow-up of $C_V\subset V$, with exceptional divisor $E_V$. Therefore $E_V$ is identified in codimension $1$ with $\widehat E\subset \widehat V$. Since $\nu(\widehat E)=E_Y$ birationally and $\rho(E_V)=E$, the divisors $E$ and $E_Y$ correspond in codimension $1$. The same comparison identifies the divisors over $\Gamma_1,\Gamma_2$ on the two sides, while the transforms of $Q_1,Q_2$ are contracted in both constructions. Hence $\chi$ contracts no divisor in either direction; in particular, it is small.

We now determine the indeterminacy loci. The only small modifications not already accounted for are those associated with the six lines $L_i$ and $R_i$. Thus the only possible indeterminacy curves on $Z$ are the strict transforms $\widetilde L_i$ and $\widetilde R_i$. There are no further curves contained in $E$: on the model $\widetilde Z$, the map to $Y$ is regular at the generic points of $\widetilde F_1$ and $\widetilde F_2$ and is constant on the fibers contracted by $\rho$, hence it descends near the generic points of $\Sigma_1$ and $\Sigma_2$. Conversely, along the generic point of each $\widetilde L_i$ and $\widetilde R_i$, the two constructions choose different small modifications. Therefore
\[
        \Indet(\chi)=
        \bigcup_{i=1}^3(\widetilde L_i\cup \widetilde R_i).
\]
These six curves are pairwise disjoint: the $L_i$ are fibers of one ruling on $Q_2$, the $R_i$ are fibers of one ruling on $Q_1$, and after the blow-up of $l$ their distinct tangent directions at the singular points are separated.

Let $b\subset X$ be the image of one of these curves, let $\widetilde b\subset Z$ be its strict transform, and let $c\subset V$ be the corresponding flopped curve. Since $-K_V\cdot c=0$ and
\[
        K_V=\phi^*K_X+\frac12F_1+\frac12F_2,
\]
one gets $\phi^*(-K_X)\cdot c=\frac12$: for curves coming from $L_i$ the curve $c$ meets $F_2$ once and is disjoint from $F_1$, and for curves coming from $R_i$ the roles of $F_1$ and $F_2$ are exchanged. Hence $f^*(-K_X)\cdot \widetilde b=\frac12$. Moreover $b$ meets $l$ transversely at one of the two index-two points; on the index-one cover the intersection with the exceptional divisor of the blow-up is $1$, and after quotienting it becomes $E\cdot\widetilde b=\frac12$. Since $-K_Z=f^*(-K_X)-E$, we get $-K_Z\cdot\widetilde b=0$. Thus the curves $\widetilde L_i,\widetilde R_i$ are $K_Z$-trivial.

We next consider $\chi^{-1}$. On $\widehat V$ the only curves along which the two constructions differ by a small modification are
\[
        \ell_1,\ell_2,\ell_3,m_1,m_2,m_3\subset \widehat E.
\]
They are not contracted by $\nu$; write $\overline\ell_j:=\nu(\ell_j)$ and $\overline m_j:=\nu(m_j)$. The only possible additional curves in $E_Y$ would be the images $\Lambda_i=\nu(G_i)$, but the map to $Z$ is constant on the fibers of $G_i\to\Lambda_i$ and therefore descends at their generic points. Hence
\[
        \Indet(\chi^{-1})=
        \bigcup_{j=1}^3(\overline\ell_j\cup \overline m_j).
\]
These six curves are pairwise disjoint by Proposition~\ref{prop:Y-model}. Furthermore, by Proposition~\ref{prop:restrictions-TE},
\[
        -K_{\widehat V}|_{\widehat E}
        \sim a+4b-\sum_{i=1}^3A_i-\sum_{i=1}^3B_i.
\]
Since $\ell_j\sim b-A_j$ and $m_j\sim b-B_j$, we have $-K_{\widehat V}\cdot\ell_j=-K_{\widehat V}\cdot m_j=0$. The morphism $\nu$ is crepant by Proposition~\ref{prop:Y-model}, so $-K_Y\cdot\overline\ell_j=-K_Y\cdot\overline m_j=0$.

Thus $\chi$ is a small birational map whose indeterminacy curves on both sides are pairwise disjoint and $K$-trivial. It is therefore the composition of the six corresponding flops. Since flops are crepant, $\chi$ is crepant. Finally, $E$ corresponds to $E_Y$ and, because $\chi$ is small, $K_Z$ corresponds to $K_Y$. From $K_Z=f^*K_X+E$ we obtain $f^*(-K_X)=-K_Z+E$, and hence $f^*(-K_X)$ corresponds to $-K_Y+E_Y$.
\end{proof}

\begin{Proposition}\label{prop:Zn-equals-Y}
Fix $\delta\in\QQ\cap(1,3)$ and set $A_\delta:=f^*(-K_X)-\delta E$ on $Z$. Let
$$
R(Z,A_\delta):=\bigoplus_{m\geq 0}H^0\bigl(Z,\OO_Z(\lfloor mA_\delta\rfloor)\bigr)
$$
be its divisorial section ring, and let $Z_n:=\Proj R(Z,A_\delta)$ be the ample model of $A_\delta$. Then $Z_n\simeq Y$.
\end{Proposition}

\begin{proof}
Since $\chi\colon Z\dashrightarrow Y$ is small, it identifies Weil divisor classes in codimension $1$. By Proposition~\ref{prop:birational-input}, the divisor $E$ corresponds to $E_Y$, while $f^*(-K_X)$ corresponds to $-K_Y+E_Y$. Hence $A_\delta$ corresponds to $B_\delta:=-K_Y+(1-\delta)E_Y$.

Choose a positive integer $r$ such that $rA_\delta$ and $rB_\delta$ are Cartier. Let $U\subset Z$ and $U'\subset Y$ be the maximal open subsets on which $\chi$ is an isomorphism. Since $\chi$ is small, the complements of $U$ and $U'$ have codimension at least $2$. Moreover, under the isomorphism $U\simeq U'$, the restrictions of $rA_\delta$ and $rB_\delta$ correspond. Since $Z$ and $Y$ are normal, sections of rank-one reflexive sheaves extend uniquely across subsets of codimension at least $2$. Thus, for every $k\geq 0$, we have natural isomorphisms
$$
H^0\bigl(Z,\OO_Z(krA_\delta)\bigr)
\cong
H^0\bigl(U,\OO_U(krA_\delta)\bigr)
\cong
H^0\bigl(U',\OO_{U'}(krB_\delta)\bigr)
\cong
H^0\bigl(Y,\OO_Y(krB_\delta)\bigr).
$$
These isomorphisms are compatible with multiplication of sections. Therefore the $r$-th Veronese subrings $R(Z,A_\delta)^{(r)}$ and $R(Y,B_\delta)^{(r)}$ are isomorphic, where
$$
R(Z,A_\delta)^{(r)}
=
\bigoplus_{k\geq 0}H^0\bigl(Z,\OO_Z(krA_\delta)\bigr),
\qquad
R(Y,B_\delta)^{(r)}
=
\bigoplus_{k\geq 0}H^0\bigl(Y,\OO_Y(krB_\delta)\bigr).
$$
Here we used that $rA_\delta$ and $rB_\delta$ are Cartier, so $\lfloor krA_\delta\rfloor=krA_\delta$ and $\lfloor krB_\delta\rfloor=krB_\delta$.

Passing from a positively graded ring to a Veronese subring does not change its projective spectrum. Indeed, for any positively graded ring $R$ and any $r>0$, the natural inclusion $R^{(r)}\subset R$ induces an isomorphism $\Proj R\simeq \Proj R^{(r)}$: homogeneous prime ideals of $R$ not containing $R_+$ correspond naturally to homogeneous prime ideals of $R^{(r)}$ not containing $R^{(r)}_+$, and the standard affine charts agree after replacing a homogeneous element by a suitable power of degree divisible by $r$. Hence $\Proj R(Z,A_\delta)\simeq \Proj R(Y,B_\delta)$.

Set $t:=3-\delta\in(0,2)$. By Proposition~\ref{prop:intersections-Y}, one has $-K_Y\sim_{\QQ}M_Y+2E_Y$, hence $B_\delta\sim_{\QQ}M_Y+tE_Y$. Since $-K_Y\sim_{\QQ}M_Y+2E_Y$, we can rewrite this as
$$
B_\delta
\sim_{\QQ}
\left(1-\frac t2\right)M_Y+\frac t2(-K_Y).
$$
Both coefficients are strictly positive. The divisor $M_Y$ is nef and spans one boundary ray of $\Nef(Y)$, since $|M_Y|$ is a base-point-free pencil and $M_Y^2\equiv 0$. The divisor $-K_Y$ spans the other boundary ray of $\Nef(Y)$, since it is nef and big but not ample, as it has degree $0$ on the six curves contracted by $\tau$. Since $\rho(Y)=2$ and $M_Y$ is not proportional to $-K_Y$, these two rays are the two boundary rays of $\Nef(Y)$. Thus $B_\delta$ lies in the interior of $\Nef(Y)$. By Kleiman's criterion, $B_\delta$ is ample.

Since $B_\delta$ is ample and $\QQ$-Cartier, its section ring gives back $Y$, namely $\Proj R(Y,B_\delta)\simeq Y$. Therefore $Z_n=\Proj R(Z,A_\delta)\simeq \Proj R(Y,B_\delta)\simeq Y$.
\end{proof}

\begin{Corollary}\label{cor:Z-weak-fano}
The variety $Z$ is weak Fano.
\end{Corollary}

\begin{proof}
By Corollary~\ref{cor:Y-weak-fano}, the divisor $-K_Y$ is nef and big. By Proposition~\ref{prop:birational-input}, the birational map
$\chi\colon Z\dashrightarrow Y$
is a composition of flops. In particular, it is small and crepant. Let
\stepcounter{thm}
\begin{equation}\label{eq:common-resolution-ZY}
\tag{\thethm}
\begin{tikzcd}
& R \arrow[dl,"p"'] \arrow[dr,"q"] & \\
Z \arrow[rr,dashed,"\chi"] && Y
\end{tikzcd}
\end{equation}
be a common resolution of $\chi$. Since $\chi$ is crepant, we have $p^*K_Z=q^*K_Y$, and hence $p^*(-K_Z)=q^*(-K_Y)$.

Since $-K_Y$ is nef, its pullback $q^*(-K_Y)$ is nef on $R$. Thus $p^*(-K_Z)$ is nef. We claim that this implies that $-K_Z$ is nef. Let $C\subset Z$ be an irreducible curve. Choose an irreducible curve $\widetilde C\subset R$ which is not $p$-exceptional and maps onto $C$. Then $p_*\widetilde C=dC$ for some integer $d>0$. By the projection formula, $d(-K_Z)\cdot C=p^*(-K_Z)\cdot \widetilde C\geq 0$. Hence $(-K_Z)\cdot C\geq 0$, and therefore $-K_Z$ is nef.

Similarly, since $-K_Y$ is big, its pullback $q^*(-K_Y)$ is big. Hence $p^*(-K_Z)$ is big. Bigness is invariant under birational pullback for $\QQ$-Cartier divisors, so $-K_Z$ is big.

Finally, $Z$ is normal and projective, and it is $\QQ$-factorial by Lemma~\ref{lem:E-irred}. Therefore $-K_Z$ is nef and big on the normal projective $\QQ$-factorial threefold $Z$, and $Z$ is weak Fano.
\end{proof}

\begin{Remark}
Although $Z$ is singular, all computations below are legitimate. Indeed, by Lemma~\ref{lem:singularities-Z}, the variety $Z$ is normal, $\QQ$-factorial and klt, with quotient singularities of type $\frac12(1,1,0)$ along $\Sigma_1\cup\Sigma_2$. The divisors $K_Z$, $E$ and $f^*(-K_X)$ are $\QQ$-Cartier, so the volumes and intersection numbers of $D(u):=f^*(-K_X)-uE$ are well-defined. If an integral divisor is not Cartier, $\OO_Z(D)$ denotes the corresponding reflexive divisorial sheaf; after passing to sufficiently divisible indices this agrees with the usual line-bundle notation and does not affect the asymptotic invariants. Finally, adjunction to $E$ is understood via the different,
\[
        (K_Z+E)|_E\sim_{\QQ}K_E+\Delta_E,
\]
which in our local quotient model gives $\Delta_E=\frac12\Sigma_1+\frac12\Sigma_2$.
\end{Remark}

\begin{Lemma}\label{lem:small-birational-sections}
Let $D$ be a $\QQ$-Cartier $\QQ$-divisor on $Z$, and let $D_Y$ be its strict transform on $Y$. Then, for every sufficiently divisible positive integer $m$, there is a natural isomorphism $H^0(Z,\OO_Z(mD))\cong H^0(Y,\OO_Y(mD_Y))$. In particular, $\Vol_Z(D)=\Vol_Y(D_Y)$.
\end{Lemma}

\begin{proof}
By Proposition~\ref{prop:birational-input}, the birational map $\chi\colon Z\dashrightarrow Y$ is small. Hence there exist open subsets $U\subset Z$ and $U_Y\subset Y$ such that $\chi|_U\colon U\xrightarrow{\sim}U_Y$ is an isomorphism, while $\codim_Z(Z\setminus U)\geq 2$ and $\codim_Y(Y\setminus U_Y)\geq 2$.

Since $D$ is $\QQ$-Cartier and $Y$ is $\QQ$-factorial, we may choose a positive integer $r$ such that both $rD$ and $rD_Y$ are Cartier. Then, for every $k\geq 1$, the divisors $krD$ and $krD_Y$ are Cartier, and their restrictions to the isomorphic open sets $U\simeq U_Y$ correspond. Thus $\OO_U(krD|_U)\cong \OO_{U_Y}(krD_Y|_{U_Y})$.

Let $j\colon U\hookrightarrow Z$ and $j_Y\colon U_Y\hookrightarrow Y$ be the open immersions. Since $Z$ and $Y$ are normal, and since $\OO_Z(krD)$ and $\OO_Y(krD_Y)$ are invertible, hence reflexive, the complements of $U$ and $U_Y$ having codimension at least $2$ imply $\OO_Z(krD)\cong j_*\OO_U(krD|_U)$ and $\OO_Y(krD_Y)\cong j_{Y*}\OO_{U_Y}(krD_Y|_{U_Y})$. Taking global sections and using $U\simeq U_Y$, we obtain $H^0(Z,\OO_Z(krD))\cong H^0(Y,\OO_Y(krD_Y))$ for every $k\geq 1$. This proves the first assertion.

It remains to compare the volumes. Since $\dim Z=\dim Y=3$, homogeneity gives $\Vol_Z(D)=r^{-3}\Vol_Z(rD)$ and $\Vol_Y(D_Y)=r^{-3}\Vol_Y(rD_Y)$. The equality of sections above gives $h^0(Z,\OO_Z(krD))=h^0(Y,\OO_Y(krD_Y))$ for every $k\geq 1$, hence $\Vol_Z(rD)=\Vol_Y(rD_Y)$. Therefore $\Vol_Z(D)=\Vol_Y(D_Y)$.
\end{proof}

\begin{Proposition}\label{prop:volume-on-Y}
For $u\in[1,3]$, the divisor $-K_Y+(1-u)E_Y$ is nef and
$\Vol(-K_Y+(1-u)E_Y)=6(3-u)^2-\frac32(3-u)^3$.
\end{Proposition}

\begin{proof}
By Proposition~\ref{prop:intersections-Y}, one has $-K_Y\sim_{\QQ}M_Y+2E_Y$. Therefore $-K_Y+(1-u)E_Y\sim_{\QQ}M_Y+(3-u)E_Y$. Set $t:=3-u$. Since $u\in[1,3]$, we have $t\in[0,2]$, and
$$
        M_Y+tE_Y
        =
        \left(1-\frac t2\right)M_Y+\frac t2(M_Y+2E_Y)
        \sim_{\QQ}
        \left(1-\frac t2\right)M_Y+\frac t2(-K_Y).
$$

Now $-K_Y$ is nef by Corollary~\ref{cor:Y-weak-fano}. Moreover, $M_Y$ is nef since $|M_Y|$ is base-point-free by the proof of Proposition~\ref{prop:intersections-Y}. Since $1-\frac t2\geq 0$ and $\frac t2\geq 0$ for $t\in[0,2]$, the divisor $M_Y+tE_Y$ is nef.

Since $Y$ is projective and $\QQ$-factorial, and since $M_Y+tE_Y$ is nef and $\QQ$-Cartier, its volume is its top self-intersection. Using Proposition~\ref{prop:intersections-Y}, namely $M_Y^3=0$, $M_Y^2\cdot E_Y=0$, $M_Y\cdot E_Y^2=2$, and $E_Y^3=-\frac32$, we get $
        (M_Y+tE_Y)^3
        =
        3t^2\,M_Y\cdot E_Y^2+t^3E_Y^3
        =
        6t^2-\frac32t^3$. Substituting $t=3-u$, we obtain $\Vol(-K_Y+(1-u)E_Y)=6(3-u)^2-\frac32(3-u)^3$.
\end{proof}

\begin{Proposition}\label{prop:volume-on-Z}
For $u\in[0,1]$, the divisor $f^*(-K_X)-uE$ is nef and $\Vol(f^*(-K_X)-uE)=15-3u^2$.
\end{Proposition}

\begin{proof}
Set $A:=f^*(-K_X)$. We first prove that $A-uE$ is nef for every $u\in[0,1]$. As shown in the proof of Proposition~\ref{prop:relative-cone-Z}, one has $\rho(X)=1$. By Lemma~\ref{lem:l-anticanonical-line}, $(-K_X)\cdot l=1>0$. Since $\NEbar(X)$ is one-dimensional, it follows that $-K_X$ has positive degree on every nonzero effective curve class on $X$. Hence $-K_X$ is ample by Kleiman's criterion, and therefore $A=f^*(-K_X)$ is nef on $Z$.

On the other hand, by Corollary~\ref{cor:Z-weak-fano}, the divisor $-K_Z$ is nef. By Lemma~\ref{lem:canonical-class-Z}, one has $K_Z=f^*K_X+E$, hence $-K_Z=A-E$. Therefore, for $u\in[0,1]$, one has $A-uE=u(A-E)+(1-u)A=u(-K_Z)+(1-u)A$, which is a convex combination of nef divisors. Thus $A-uE$ is nef for every $u\in[0,1]$. Since $Z$ is projective and $\QQ$-factorial, we get $\Vol(A-uE)=(A-uE)^3$.

We compute first $(-K_X)^3$. Let $H_Q$ be the hyperplane class on $Q$. Since $Q\subset\PP^4$ is a smooth quadric threefold, one has $-K_Q\sim 3H_Q$ and $H_Q^3=2$, so $(-K_Q)^3=54$. Recall that $\Theta\colon W=\Bl_{\Gamma_1\sqcup\Gamma_2}(Q)\to Q$ is the blow-up of $Q$ along the two disjoint smooth twisted cubics $\Gamma_1,\Gamma_2$. For the blow-up of a smooth threefold along a smooth curve $B$, one has $(-K_{\Bl_B X})^3=(-K_X)^3+2K_X\cdot B-2+2g(B)$. Since $g(\Gamma_i)=0$ and $K_Q\cdot\Gamma_i=-3H_Q\cdot\Gamma_i=-9$, each summand is $2(-9)-2=-20$. Hence $(-K_W)^3=54-20-20=14$. Since $\psi\colon W\dashrightarrow V$ is a sequence of flops, it is small and crepant, so $(-K_V)^3=(-K_W)^3=14$.

Write $K_V=\phi^*K_X+a_1F_1+a_2F_2$ for some $a_1,a_2\in\QQ$. Since $\phi$ contracts $F_i\simeq\PP^2$ to a point of type $\frac12(1,1,1)$, one has $F_i|_{F_i}\simeq\OO_{\PP^2}(-2)$. By adjunction, $K_V|_{F_i}=K_{F_i}-F_i|_{F_i}\simeq\OO_{\PP^2}(-1)$. Since $\phi^*K_X|_{F_i}\sim 0$, restricting to $F_i$ gives $\OO_{\PP^2}(-1)\simeq a_iF_i|_{F_i}\simeq\OO_{\PP^2}(-2a_i)$, hence $a_i=\frac12$. Thus $K_V=\phi^*K_X+\frac12F_1+\frac12F_2$. Since $\phi^*(-K_X)|_{F_i}\sim 0$, all mixed intersections involving $\phi^*(-K_X)$ and $F_i$ vanish. Moreover $F_i^3=(c_1(\OO_{F_i}(F_i)))^2=(-2h)^2=4$, where $h$ is the hyperplane class on $\PP^2$. Since $F_1\cap F_2=\varnothing$, we get $14=(-K_V)^3=(-K_X)^3-\frac18(F_1^3+F_2^3)=(-K_X)^3-1$. Hence $(-K_X)^3=15$, and therefore $A^3=15$.

We have $A^2\cdot E=0$ by the projection formula, since $E$ is $f$-exceptional. It remains to compute $A\cdot E^2$ and $E^3$. Let $g\colon \widetilde Z:=\Bl_{C_V}(V)\to V$ be the blow-up of $V$ along $C_V$, with exceptional divisor $D$, and let $\widetilde F_i\subset \widetilde Z$ be the strict transform of $F_i$. Since $C_V$ is a flopping curve on the smooth threefold $V$, one has $N_{C_V/V}\cong\OO_{\PP^1}(-1)\oplus\OO_{\PP^1}(-1)$. Hence $D\simeq\PP(N_{C_V/V})\simeq\PP^1\times\PP^1$ and $D^3=-\deg N_{C_V/V}=2$.

Since $C_V$ meets each $F_i\simeq\PP^2$ transversely in one point $p_i$, the surface $\widetilde F_i$ is the blow-up of $F_i$ at $p_i$, hence $\widetilde F_i\simeq\F_1$. Let $\pi_i\colon \widetilde F_i\to F_i\simeq\PP^2$ be the blow-up map, and let $e_i:=D\cap\widetilde F_i$ be the exceptional section. Since $g^*F_i=\widetilde F_i$ as divisors and $F_i|_{F_i}\simeq\OO_{\PP^2}(-2)$, we have $\widetilde F_i|_{\widetilde F_i}\sim -2\pi_i^*\OO_{\PP^2}(1)$. Therefore $D^2\cdot\widetilde F_i=e_i^2=-1$, $D\cdot\widetilde F_i^2=e_i\cdot(-2\pi_i^*\OO_{\PP^2}(1))=0$, and $\widetilde F_i^3=4$.

We claim that $\phi\circ g\colon \widetilde Z\to X$ factors through $f\colon Z\to X$. Away from $x_1$ and $x_2$, this is clear, since both morphisms are the ordinary blow-up of the smooth curve $l$. Near each $x_i$, we are in the toric local model of Lemma~\ref{lem:singularities-Z}: the contraction $V\to X$ is the toric extraction corresponding to the ray generated by $\frac12(1,1,1)$, while $g$ is the further blow-up of the strict transform of the invariant curve corresponding to $l$, that is, the star subdivision along the ray generated by $e_1+e_2$. Forgetting the ray $\frac12(1,1,1)$ gives exactly the toric blow-up $Z\to X$ of the curve $l$. Thus there exists a birational morphism $\rho\colon \widetilde Z\to Z$ such that $f\circ\rho=\phi\circ g$. The only $\rho$-exceptional divisors are $\widetilde F_1$ and $\widetilde F_2$, each contracted onto $\Sigma_i$.

Hence we may write $\rho^*E=D+a_1\widetilde F_1+a_2\widetilde F_2$ for some $a_1,a_2\in\QQ$. Let $\gamma_i\subset\widetilde F_i$ be a fiber of the ruling $\widetilde F_i\simeq\F_1\to\PP^1$ contracted by $\rho$. Then $D\cdot\gamma_i=e_i\cdot\gamma_i=1$ and $\widetilde F_i\cdot\gamma_i=(-2\pi_i^*\OO_{\PP^2}(1))\cdot\gamma_i=-2$. Since $\rho^*E\cdot\gamma_i=0$, we get $0=D\cdot\gamma_i+a_i\widetilde F_i\cdot\gamma_i=1-2a_i$, so $a_i=\frac12$. Therefore $\rho^*E=D+\frac12\widetilde F_1+\frac12\widetilde F_2$.

Moreover $\rho^*A=(f\circ\rho)^*(-K_X)=(\phi\circ g)^*(-K_X)=g^*\phi^*(-K_X)$. Since $\phi^*(-K_X)|_{F_i}\sim 0$, every term involving $\widetilde F_i$ vanishes when intersected with $\rho^*A$. Hence $A\cdot E^2=\rho^*A\cdot(\rho^*E)^2=g^*\phi^*(-K_X)\cdot D^2$. For the blow-up of a smooth threefold along a smooth curve, one has $g^*L\cdot D^2=-L\cdot C_V$ for every $\QQ$-Cartier divisor $L$ on $V$. Applying this to $L=\phi^*(-K_X)$, we obtain $A\cdot E^2=-\phi^*(-K_X)\cdot C_V=-(-K_X)\cdot l=-1$, where the last equality follows from Lemma~\ref{lem:l-anticanonical-line}. Finally,
$$
E^3=(\rho^*E)^3
=\left(D+\frac12\widetilde F_1+\frac12\widetilde F_2\right)^3
=2+\frac32(-1-1)+\frac34(0+0)+\frac18(4+4)=0.
$$
Thus $A^3=15$, $A^2\cdot E=0$, $A\cdot E^2=-1$, and $E^3=0$. Hence $\Vol(f^*(-K_X)-uE)=\Vol(A-uE)=(A-uE)^3=15-3u^2$.
\end{proof}

\begin{Definition}\label{def:dreamy-discrepancy-S}
Let $X$ be a normal projective variety, let $L$ be a big $\QQ$-Cartier divisor on $X$, and let $F$ be a prime divisor over $X$. Choose a birational morphism $\pi\colon Y\to X$ from a normal variety $Y$ such that $F$ is a divisor on $Y$.

We say that $F$ is \emph{dreamy over $X$ with respect to $L$} if, for some positive integer $r$ such that $rL$ is Cartier, the bi-graded algebra
$$
        \bigoplus_{m,j\geq 0}
        H^0\bigl(Y,\OO_Y(mr\pi^*L-jF)\bigr)
$$
is finitely generated. In the Fano case, when $L=-K_X$, we simply say that $F$ is dreamy over $X$.

The \emph{log discrepancy} of $F$ with respect to $X$ is
$$
        A_X(F):=1+\operatorname{ord}_F(K_Y-\pi^*K_X).
$$
This number is independent of the choice of $Y$.

The expected vanishing order, or $S$-invariant, of $F$ with respect to $L$ is
$$
        S_L(F)
        :=
        \frac{1}{L^n}
        \int_0^{+\infty}
        \Vol_Y(\pi^*L-uF)\,du,
        \qquad n:=\dim X.
$$
In the Fano case $L=-K_X$, we write $S_X(F):=S_{-K_X}(F)$.
\end{Definition}

\begin{Remark}\label{rem:why-dreamy-S-important}
These invariants are the valuative ingredients of K-stability. The quantity $\frac{A_X(F)}{S_X(F)}$ measures whether the divisor $F$ can destabilize $X$. In particular, a divisor with $A_X(F)\leq S_X(F)$ is a possible destabilizing divisor, while the inequality $A_X(F)>S_X(F)$ excludes $F$ as a destabilizing divisor. Dreaminess is the finite generation condition which allows one to associate to $F$ a test configuration and to compute the relevant models by taking Proj of section rings. In our situation, the divisor $E$ is the exceptional divisor of $f\colon Z=\Bl_l(X)\to X$, and the goal of Theorem~\ref{thm:E-beta-positive} is to show that this natural divisorial valuation does not destabilize $X$.
\end{Remark}

\begin{thm}\label{thm:E-beta-positive}
In the setting of Assumption~\ref{ass:general-setting}, the divisor $E$ defines a dreamy prime divisor over $X$ with log discrepancy $A_X(E)=2$, and $S(-K_X;E)=\frac85<2$. In particular, $A_X(E)>S(-K_X;E)$.
\end{thm}

\begin{proof}
By Lemma~\ref{lem:singularities-Z}, the singularities of $Z$ are quotient singularities of type $\frac12(1,1,0)$ along the curves $\Sigma_1,\Sigma_2$, hence $Z$ is klt. By Corollary~\ref{cor:Z-weak-fano}, the divisor $-K_Z$ is nef and big. Thus $Z$ is of Fano type: indeed, choosing a sufficiently small ample $\QQ$-divisor $A$, bigness of $-K_Z$ gives an effective $\QQ$-divisor $\Delta\sim_{\QQ}-K_Z-A$, and for a general choice of $\Delta$ the pair $(Z,\Delta)$ is klt, while $-(K_Z+\Delta)\sim_{\QQ}A$ is ample. Hence $Z$ is a Mori dream space by \cite[Corollary~1.3.2]{BCHM}.

Choose an integer $r>0$ divisible by the Cartier indices of $K_X$ and $E$. Then $f^*(-rK_X)$ and $rE$ are Cartier on $Z$ by Lemma~\ref{lem:E-QCartier-index}. Since $Z$ is a $\QQ$-factorial Mori dream space, its Cox ring is finitely generated. Consequently, every divisorial algebra obtained from a finitely generated subgroup of $\WCl(Z)$ is finitely generated. 

Applying this to the subgroup generated by $f^*(-rK_X)$ and $E$, we get finite generation of 
$$
        \bigoplus_{m,j\geq 0}
        H^0\bigl(Z,\OO_Z(mf^*(-rK_X)-jE)\bigr).
$$ 
Here $\OO_Z(D)$ denotes the reflexive divisorial sheaf associated with the integral Weil divisor $D$. Let us spell out the relation with the usual Cartier formulation of dreaminess. Since $rE$ is Cartier, the subalgebra obtained by keeping only the summands with $j$ divisible by $r$ is 
$$
        \bigoplus_{m,q\geq 0}
        H^0\bigl(Z,\OO_Z(mf^*(-rK_X)-qrE)\bigr).
$$
All divisors appearing in this latter algebra are Cartier. This is a Veronese-type subalgebra of the preceding finitely generated multigraded algebra, hence it is finitely generated. Conversely, passing between the full divisorial algebra and this Cartier subalgebra does not affect the associated Proj used in the definition of the model, since homogeneous prime ideals and standard affine charts are unchanged after replacing homogeneous sections by sufficiently divisible powers. Thus the finite generation above is precisely the dreaminess condition for the prime divisor $E$ over $X$ with respect to $-K_X$.

By Lemma~\ref{lem:canonical-class-Z}, one has $K_Z=f^*K_X+E$. Therefore $A_X(E)=1+\ord_E(K_Z-f^*K_X)=2$. We compute $S(-K_X;E)$. For $u\in[0,1]$, Proposition~\ref{prop:volume-on-Z} gives $\Vol(f^*(-K_X)-uE)=15-3u^2$. For $u\geq 1$, set $D_u:=f^*(-K_X)-uE=-K_Z+(1-u)E$, and let $D_u^Y$ be its strict transform on $Y$. Since $\chi$ is small and crepant by Proposition~\ref{prop:birational-input}, and since $E$ corresponds to $E_Y$, we have $D_u^Y=-K_Y+(1-u)E_Y$. For rational $u\geq 1$, Lemma~\ref{lem:small-birational-sections} gives $\Vol_Z(D_u)=\Vol_Y(D_u^Y)$; by continuity of the volume function, the same equality holds for real $u\geq 1$.

If $u\in[1,3]$, Proposition~\ref{prop:volume-on-Y} yields $\Vol(f^*(-K_X)-uE)=6(3-u)^2-\frac32(3-u)^3$. If $u>3$, then $D_u^Y\sim_{\QQ}M_Y-(u-3)E_Y$. Since $E_Y$ is effective, monotonicity of volume gives $0\leq \Vol_Y(D_u^Y)\leq \Vol_Y(M_Y)=M_Y^3=0$ by Proposition~\ref{prop:intersections-Y}. Hence the pseudo-effective threshold is $\tau(-K_X;E)=3$.

Finally, Proposition~\ref{prop:volume-on-Z} at $u=0$ gives $(-K_X)^3=15$, and therefore
$$
\begin{aligned}
S(-K_X;E)
&=
\frac1{15}\left(
\int_0^1 (15-3u^2)\,du
+
\int_1^3\left(6(3-u)^2-\frac32(3-u)^3\right)\,du
\right)  
=
\frac1{15}(14+10)
=
\frac85.
\end{aligned}
$$
Thus $A_X(E)=2>\frac85=S(-K_X;E)$.
\end{proof}

We now study only the local stability threshold along the curve $l\subset X$.

\begin{Notation}\label{not:refinement-by-E}
Let $V_\bullet:=\bigoplus_{m\geq 0}V_m$, where $V_m:=H^0(X,\OO_X(-mK_X))$; if $mK_X$ is not Cartier, $\OO_X(-mK_X)$ means the rank-one reflexive sheaf associated with the Weil divisor $-mK_X$. The divisor $E$ defines the filtration
\[
        \mathcal F_E^\lambda V_m:=\{s\in V_m\mid \ord_E(s)\geq \lambda\}.
\]
Since the value group of $\ord_E$ is $\ZZ$, only integral jumps occur. Thus, for $j\geq 0$,
\[
        \mathcal F_E^jV_m
        =\{s\in H^0(X,\OO_X(-mK_X))\mid \ord_E(s)\geq j\}.
\]
Equivalently, $\mathcal F_E^jV_m$ is the image of
\[
        H^0\bigl(Z,\OO_Z(f^*(-mK_X)-jE)\bigr)
        \longrightarrow H^0\bigl(X,\OO_X(-mK_X)\bigr),
\]
where the sheaf on $Z$ is reflexive divisorial. The refinement of $V_\bullet$ by $E$ is the $\ZZ_{\geq 0}^2$-graded series
\[
        W_{m,j}:=\mathcal F_E^jV_m/\mathcal F_E^{j+1}V_m.
\]
Its slice with ratio $u=j/m$ is governed asymptotically, after pulling back to the smooth model $\widetilde E$, by the positive part of $D(u):=f^*(-K_X)-uE$.
\end{Notation}

\begin{Definition}\label{def:delta-refined-series}
Let $(E,\Delta_E)$ be the pair obtained by adjunction from $(Z,E)$, and let $W_\bullet$ be the refinement above. For an irreducible subvariety $B\subset E$, set
\[
        \delta_B(E,\Delta_E;W_\bullet):=
        \inf_F \frac{A_{(E,\Delta_E)}(F)}{S(W_\bullet;F)},
\]
where $F$ runs over all prime divisors over $E$ whose center contains $B$. We use the Abban--Zhuang slice formula \cite[Theorem~1.106]{ACCFKMSV21}, in the simplified form \cite[Corollary~1.108]{ACCFKMSV21}, to compute $S(W_\bullet;F)$. In the case needed below, for a curve $B\subset E$ and its strict transform $\widetilde B\subset\widetilde E$, this gives
\[
        S(W_\bullet;B)
        =
        \frac{
        \displaystyle
        \int_0^3\int_0^\infty
        \Vol_{\widetilde E}\bigl(P(u)|_{\widetilde E}-v\widetilde B\bigr)\,dv\,du
        }{
        \displaystyle
        \int_0^3
        \bigl(P(u)|_{\widetilde E}\bigr)^2\,du
        },
\]
where $P(u)$ is the positive part of the Zariski decomposition of the pull-back of $D(u)$ to the common resolution.
\end{Definition}

\begin{Remark}\label{rem:negative-part-AZ}
In the general Abban--Zhuang formula, if $p^*D(u)=P(u)+N(u)$, there is also a contribution
\[
        \bigl(P(u)^{n-1}\cdot Y\bigr)
        \ord_F\bigl(N(u)|_Y\bigr)
\]
from the negative part. In our computation this term vanishes for every curve $B\subset E$ dominating $l$. For $0\leq u\leq 1$, the divisor $D(u)$ is nef, so $N(u)=0$. For $1\leq u\leq 3$, the model $Y$ is the nef model of $D(u)$; using the common resolution \eqref{eq:common-resolution-ZY}, the support of $N(u)$ lies over the curves where $\chi$ is not an isomorphism. After restriction to $\widetilde E\simeq\widehat E$, this support is contained in
\[
        \ell_1\cup\ell_2\cup\ell_3\cup m_1\cup m_2\cup m_3.
\]
These curves are vertical for the ruling $\widetilde E\to\PP^1$, whereas the strict transform of a curve dominating $l$ is horizontal. Hence $\ord_{\widetilde B}(N(u)|_{\widetilde E})=0$ for all $u$, and only the surface-volume term remains.
\end{Remark}

\begin{thm}[Abban--Zhuang adjunction]\label{thm:AZ-adjunction-along-E}
In the notation above, assume that $E$ is dreamy over $X$, and let $W_\bullet$ be the refinement of the anticanonical series by $E$. Then
$$
        \delta_l(X)
        \geq
        \min\left\{
        \frac{A_X(E)}{S(-K_X;E)},\
        \inf_{f(B)=l}\delta_B(E,\Delta_E;W_\bullet)
        \right\},
$$
where the infimum is taken over all irreducible curves $B\subset E$ dominating $l$. 
\end{thm}

This is the form of \cite[Theorem~3.2]{AZ20} used below.

\begin{Remark}\label{rem:AZ-use}
Thus, since Theorem~\ref{thm:E-beta-positive} gives $A_X(E)/S(-K_X;E)=5/4>1$, to prove $\delta_l(X)>1$ it remains to show that $\delta_B(E,\Delta_E;W_\bullet)>1$ for every irreducible curve $B\subset E$ with $f(B)=l$.
\end{Remark}

\begin{Proposition}\label{prop:l-not-destabilizing-final}
We have $\delta_l(X)>1$. In particular, the curve $l\subset X$ is not a destabilizing center.
\end{Proposition}

\begin{proof}
Let $V_\bullet:=\bigoplus_{m\geq 0}H^0(X,\OO_X(-mK_X))$, and let $W_\bullet$ be the refinement of $V_\bullet$ by $E$. By Theorem~\ref{thm:E-beta-positive}, we have $A_X(E)=2$ and $S(-K_X;E)=\frac85$. Hence $A_X(E)/S(-K_X;E)=5/4>1$. Therefore, by Theorem~\ref{thm:AZ-adjunction-along-E}, it is enough to prove that $\delta_B(E,\Delta_E;W_\bullet)>1$ for every irreducible curve $B\subset E$ such that $f(B)=l$, where $\Delta_E:=\Diff_E(0)$ is defined by adjunction, that is, by $(K_Z+E)|_E\sim_{\QQ}K_E+\Delta_E$.

We compute $\Delta_E$. By Lemma~\ref{lem:singularities-Z}, near the generic point of $\Sigma_i$ the threefold $Z$ is analytically $(\A^3_{u,v,w}/\mu_2,E=\{u=0\}/\mu_2)$, where $\mu_2$ acts by $(u,v,w)\mapsto(-u,v,-w)$. Restricting to $E$ gives $\A^2_{v,w}/\mu_2$, where $(v,w)\mapsto(v,-w)$. This quotient is smooth, with coordinates $x=v$ and $z=w^2$, and the quotient map is branched exactly along $\{z=0\}$, the image of $\Sigma_i$. Since $K_{\A^2_{v,w}}=q^*(K_{\A^2_{x,z}}+\frac12\{z=0\})$, adjunction gives $\Diff_E(0)=\frac12\Sigma_1+\frac12\Sigma_2$. Thus $\Delta_E=\frac12\Sigma_1+\frac12\Sigma_2$.

Let $B\subset E$ be an irreducible curve with $f(B)=l$. Then $B$ dominates $l$, whereas $\Sigma_1$ and $\Sigma_2$ are fibers of $E\to l$. Hence $B\neq \Sigma_1,\Sigma_2$, so $B$ is not a component of $\Delta_E$. Therefore $A_{(E,\Delta_E)}(B)=1-\operatorname{coeff}_B(\Delta_E)=1$. Consequently $\delta_B(E,\Delta_E;W_\bullet)=1/S(W_\bullet;B)$ for such curves $B$. It remains to prove that $S(W_\bullet;B)<1$ for every irreducible curve $B\subset E$ dominating $l$.

We pass to a smooth model $\widetilde E$. Let $R$ be as in the common resolution diagram~\eqref{eq:common-resolution-ZY}, and let $\widetilde E\subset R$ be the strict transform of $E$. By the construction of $\chi$, the surface $\widetilde E$ is naturally identified with the surface $\widehat E$ of Proposition~\ref{prop:restrictions-TE}. Hence $\widetilde E$ is the blow-up of $\PP^1\times\PP^1$ at six points, three on each of two distinct fibers of the first projection.

We keep the notation of Proposition~\ref{prop:restrictions-TE}: the classes $a,b\in\Pic(\widetilde E)$ are the pullbacks of the two rulings, and $A_1,A_2,A_3,B_1,B_2,B_3$ are the exceptional curves. Set $s_1:=a-\sum_{j=1}^3A_j$, $s_2:=a-\sum_{j=1}^3B_j$, $\ell_j:=b-A_j$, $m_j:=b-B_j$, $r:=a+3b-\sum_{j=1}^3A_j-\sum_{j=1}^3B_j$, and $s:=b$. Then $r^2=s^2=0$ and $r\cdot s=1$. Moreover, by Proposition~\ref{prop:restrictions-TE}, one has $r=\eta^*\OO_{\PP^1\times\PP^1}(1,0)$ and $s=\eta^*\OO_{\PP^1\times\PP^1}(0,1)$. In particular, $|r|$ defines a morphism $\pi\colon\widetilde E\to\PP^1$ whose general fiber has class $r$, and $s$ is a section of $\pi$. Since $r=s_1+m_1+m_2+m_3=s_2+\ell_1+\ell_2+\ell_3$, the two reducible fibers of $\pi$ are exactly $s_1+m_1+m_2+m_3$ and $s_2+\ell_1+\ell_2+\ell_3$.

Let $\widetilde B\subset\widetilde E$ be the strict transform of $B$. Then $\widetilde B$ is horizontal for $\pi$. Set $d:=\widetilde B\cdot r$. This is the degree of $\widetilde B\to\PP^1$, so $d\geq 1$.

We now consider the slice divisors. The support of the refinement is $u\in[0,3]$: indeed, the volume computation in the proof of Theorem~\ref{thm:E-beta-positive} shows that the pseudo-effective threshold of $E$ with respect to $-K_X$ is $\tau(-K_X;E)=3$. For $u\in[0,3]$, set $D(u):=f^*(-K_X)-uE$, and let $P(u)$ denote the positive part of the Zariski decomposition of $p^*D(u)$ on $R$.

If $0\leq u\leq 1$, then Proposition~\ref{prop:volume-on-Z} gives that $D(u)$ is nef on $Z$. Hence $P(u)=p^*D(u)$. Restricting to $\widetilde E$, we get $P(u)|_{\widetilde E}=r+us$. Indeed, $f^*(-K_X)|_E$ is the pullback of $\OO_l(1)$ and has class $r$, while $(-K_Z)|_{\widetilde E}=(f^*(-K_X)-E)|_{\widetilde E}=r+s$, so $(-E)|_{\widetilde E}=s$.

If $1\leq u\leq 3$, set $t:=3-u$. Let $D(u)^Y$ be the strict transform of $D(u)$ on $Y$. By Proposition~\ref{prop:birational-input}, the divisor $E$ corresponds to $E_Y$ and $f^*(-K_X)$ corresponds to $-K_Y+E_Y$, hence $D(u)^Y=-K_Y+(1-u)E_Y=M_Y+tE_Y$. For $1<u<3$, Proposition~\ref{prop:Zn-equals-Y} identifies $Y$ with the ample model of $D(u)$; at the endpoints the same formula follows by continuity. Thus on the common resolution one has a divisorial Zariski decomposition $p^*D(u)=q^*(M_Y+tE_Y)+N(u)$, with $N(u)$ effective and fixed over $Y$. Since $M_Y+tE_Y$ is nef by Proposition~\ref{prop:volume-on-Y}, the positive part is $P(u)=q^*(M_Y+tE_Y)$. Restricting to $\widetilde E\simeq\widehat E$, and using Lemma~\ref{lem:pullbackEY} together with Proposition~\ref{prop:restrictions-TE}, we obtain $(q^*M_Y)|_{\widetilde E}=a$ and $(q^*E_Y)|_{\widetilde E}=2b-\frac12\sum_{j=1}^3A_j-\frac12\sum_{j=1}^3B_j$. Hence $P(u)|_{\widetilde E}=L_t$, where
$$
        L_t:=a+2tb-\frac t2\sum_{j=1}^3A_j-\frac t2\sum_{j=1}^3B_j.
$$
A direct computation gives $L_t^2=4t-\frac32t^2$.

By Definition~\ref{def:delta-refined-series}, the expected vanishing order of $B$ with respect to the refinement is
$$
S(W_\bullet;B)=
\frac{\displaystyle\int_0^3\int_0^\infty
\Vol\bigl(P(u)|_{\widetilde E}-v\widetilde B\bigr)\,dv\,du}
{\displaystyle\int_0^3 \bigl(P(u)|_{\widetilde E}\bigr)^2\,du}.
$$
We first compute the denominator. Since $(r+us)^2=2u$ and $L_t^2=4t-\frac32t^2$, we get
$$
        \int_0^3 \bigl(P(u)|_{\widetilde E}\bigr)^2\,du
        =
        \int_0^1 2u\,du+\int_0^2\left(4t-\frac32t^2\right)\,dt
        =
        5.
$$

Suppose first that $d\geq 2$. For $0\leq u\leq 1$, put $N_u:=r+us$. Since $r$ is nef and $N_u\cdot r=u$, the divisor $N_u-v\widetilde B$ is not pseudo-effective for $v>u/d$. Hence, by monotonicity of volume,
$$
        \int_0^1\int_0^\infty \Vol(N_u-v\widetilde B)\,dv\,du
        \leq
        \int_0^1\int_0^{u/d}2u\,dv\,du
        =
        \frac{2}{3d}.
$$
For $1\leq u\leq 3$, write $t=3-u$ and put $N_u:=L_t$. Since $L_t\cdot r=3-t$ and $L_t^2=4t-\frac32t^2$, the same nef-curve test gives
$$
        \int_1^3\int_0^\infty \Vol(N_u-v\widetilde B)\,dv\,du
        \leq
        \frac1d\int_0^2(3-t)\left(4t-\frac32t^2\right)\,dt
        =
        \frac{22}{3d}.
$$
Thus, if $d\geq 2$, the numerator of $S(W_\bullet;B)$ is at most $4$, while the denominator is $5$. Hence $S(W_\bullet;B)\leq 4/5<1$.

It remains to consider the case $d=1$. Then $\widetilde B$ is a section of $\pi$. We use the following elementary comparison on $\widetilde E$: for every slice divisor $N_u=P(u)|_{\widetilde E}$ and every $v\geq 0$, one has $\Vol(N_u-v\widetilde B)\leq \Vol(N_u-vs)$. Indeed, after applying $\eta$, a degree-one horizontal curve becomes a curve of bidegree $(\alpha,1)$ on $\PP^1\times\PP^1$. Away from the two reducible fibers, its numerical difference from the section $s$ is a nonnegative multiple of the fiber class $r$. If the section passes through one of the six points blown up by $\beta$, the discrepancy is accounted for by one of the vertical $(-1)$-curves $\ell_j$ or $m_j$; these curves are contained in the negative part of the Zariski decomposition of $N_u-v\widetilde B$. Removing such vertical negative components can only decrease the positive part, and hence gives the displayed volume inequality. Consequently $S(W_\bullet;B)\leq S(W_\bullet;s)$ for $d=1$. It remains to prove $S(W_\bullet;s)<1$.

For $0\leq u\leq 1$, one has $P(u)|_{\widetilde E}-vs=r+(u-v)s$. If $0\leq v\leq u$, this divisor is nef and has volume $2(u-v)$; if $v>u$, its intersection with the nef class $r$ is negative, hence it is not pseudo-effective. Thus
$$
        \int_0^1\int_0^\infty
        \Vol(P(u)|_{\widetilde E}-vs)\,dv\,du
        =
        \int_0^1\int_0^u2(u-v)\,dv\,du
        =
        \frac13.
$$

We now classify the negative curves on $\widetilde E$. The only irreducible curves on $\widetilde E$ with negative self-intersection are $A_1,A_2,A_3$, $B_1,B_2,B_3$, $\ell_1,\ell_2,\ell_3$, $m_1,m_2,m_3$, $s_1$, and $s_2$. Indeed, all these curves are irreducible and have negative self-intersection. Conversely, let $C\subset\widetilde E$ be an irreducible curve not among them. If $C$ is $\beta$-exceptional, it is one of the $A_i$ or $B_i$, a contradiction. Hence $C$ is the strict transform of an irreducible curve $\overline C\sim xa+yb$ on $\PP^1\times\PP^1$, with $x,y\geq 0$. Write
$$
        C\sim xa+yb-\sum_{i=1}^3\alpha_iA_i-\sum_{i=1}^3\beta_iB_i.
$$
Since $C\neq s_1,s_2,\ell_j,m_j$, intersecting with these curves gives $\sum_i\alpha_i\leq y$, $\sum_i\beta_i\leq y$, $\alpha_j\leq x$, and $\beta_j\leq x$. Hence $\sum_i\alpha_i^2\leq x\sum_i\alpha_i\leq xy$ and $\sum_i\beta_i^2\leq x\sum_i\beta_i\leq xy$. Therefore $C^2=2xy-\sum_i\alpha_i^2-\sum_i\beta_i^2\geq 0$.

For $1\leq u\leq 3$, let $t:=3-u\in[0,2]$ and set $D_{t,v}:=L_t-vs$. Thus
$$
        D_{t,v}
        =
        a+(2t-v)b-\frac t2\sum_{j=1}^3A_j-\frac t2\sum_{j=1}^3B_j.
$$
One computes $D_{t,v}\cdot A_j=D_{t,v}\cdot B_j=t/2$, $D_{t,v}\cdot \ell_j=D_{t,v}\cdot m_j=1-t/2$, $D_{t,v}\cdot s_1=D_{t,v}\cdot s_2=t/2-v$, $D_{t,v}\cdot r=3-t-v$, $D_{t,v}\cdot a=2t-v$, and $D_{t,v}^2=4t-\frac32t^2-2v$.

First assume $0\leq v\leq t/2$. We claim that $D_{t,v}$ is nef. Its intersection with all negative curves listed above is nonnegative. If $C$ is any other irreducible curve, write $C\sim xa+yb-\sum_i\alpha_iA_i-\sum_i\beta_iB_i$ as above. Then $M:=\sum_i\alpha_i+\sum_i\beta_i\leq \min\{2y,6x\}$, and $D_{t,v}\cdot C=y+(2t-v)x-\frac t2M$. If $y\leq 3x$, then $D_{t,v}\cdot C\geq (1-t)y+(2t-v)x$, which is nonnegative for $0\leq t\leq 1$, and for $1\leq t\leq 2$ is at least $(3-t-v)x\geq 0$. If $y\geq 3x$, then $D_{t,v}\cdot C\geq y-(t+v)x\geq (3-t-v)x\geq 0$. Hence $D_{t,v}$ is nef, and $\Vol(D_{t,v})=D_{t,v}^2=4t-\frac32t^2-2v$.

Now assume $t/2\leq v\leq \min\{2t,3-t\}$. Set $\alpha:=(v-t/2)/3$ and $Q_{t,v}:=D_{t,v}-\alpha(s_1+s_2)$. Then $Q_{t,v}\cdot s_1=Q_{t,v}\cdot s_2=0$. Moreover $Q_{t,v}\cdot A_j=Q_{t,v}\cdot B_j=(2t-v)/3\geq 0$, $Q_{t,v}\cdot \ell_j=Q_{t,v}\cdot m_j=1-(t+v)/3\geq 0$, and $Q_{t,v}^2=\frac23(2t-v)(3-t-v)\geq 0$. The same estimate against an arbitrary irreducible curve $C$ not among the negative curves shows that $Q_{t,v}$ is nef. Since the intersection matrix of $s_1,s_2$ is negative definite and $Q_{t,v}\cdot S_i=0$, the decomposition $D_{t,v}=Q_{t,v}+\alpha(s_1+s_2)$ is the Zariski decomposition. Hence $\Vol(D_{t,v})=Q_{t,v}^2=\frac23(2t-v)(3-t-v)$.

If $v>\min\{2t,3-t\}$, then either $D_{t,v}\cdot a<0$ or $D_{t,v}\cdot r<0$. Since $a$ and $r$ are nef, $D_{t,v}$ is not pseudo-effective, and its volume is $0$. Thus
$$
\Vol(L_t-vs)=
\begin{cases}
4t-\frac32t^2-2v, & 0\leq v\leq \frac t2,\\[2mm]
\frac23(2t-v)(3-t-v), & \frac t2\leq v\leq \min\{2t,3-t\},\\[2mm]
0, & v\geq \min\{2t,3-t\}.
\end{cases}
$$

Integrating, if $0\leq t\leq 1$, then $\min\{2t,3-t\}=2t$, and
$$
        \int_0^\infty \Vol(L_t-vs)\,dv
        =
        \frac{t^2(16-9t)}{4}.
$$
If $1\leq t\leq 2$, then $\min\{2t,3-t\}=3-t$, and
$$
        \int_0^\infty \Vol(L_t-vs)\,dv
        =
        \frac34t^3-5t^2+9t-3.
$$
Therefore
$$
        \int_1^3\int_0^\infty
        \Vol(P(u)|_{\widetilde E}-vs)\,dv\,du
        =
        \int_0^1\frac{t^2(16-9t)}{4}\,dt
        +
        \int_1^2\left(\frac34t^3-5t^2+9t-3\right)\,dt
        =
        \frac{29}{12}.
$$
Combining this with the contribution over $0\leq u\leq 1$, we obtain
$$
        \int_0^3\int_0^\infty
        \Vol(P(u)|_{\widetilde E}-vs)\,dv\,du
        =
        \frac13+\frac{29}{12}
        =
        \frac{11}{4}.
$$
Since the denominator is $5$, we get $S(W_\bullet;s)=\frac{11/4}{5}=\frac{11}{20}<1$.

Thus $S(W_\bullet;B)<1$ for every irreducible curve $B\subset E$ dominating $l$, and hence $\delta_B(E,\Delta_E;W_\bullet)>1$ for every such $B$. Applying Theorem~\ref{thm:AZ-adjunction-along-E}, we conclude that
$$
        \delta_l(X)
        \geq
        \min\left\{
        \frac54,\
        \inf_{f(B)=l}\delta_B(E,\Delta_E;W_\bullet)
        \right\}
        >
        1.
$$
Therefore $\delta_l(X)>1$, and $l$ is not a destabilizing center.
\end{proof}

\section{The $D_6$-symmetric member}\label{Sec5}

We now choose the symmetric member used to prove non-emptiness of the K-stable locus. We keep the convention that $D_6$ denotes the dihedral group of order $12$.

\begin{Construction}\label{constr:D6-model}
Let $\omega$ be a primitive cube root of unity. In $\PP^4$ with coordinates $[x_0:x_1:x_2:x_3:x_4]$, let $\Gamma_1$ and $\Gamma_2$ be the closures of the affine parametrized curves
$$
        \Gamma_1=\{[1+t^3:t:1-t^3:t^2:0]\},
        \qquad
        \Gamma_2=\{[0:t^2:1-t^3:t:1+t^3]\}.
$$
The quadrics containing $\Gamma_1\cup\Gamma_2$ are precisely
$$
        Q_\lambda
        =
        \{x_0^2-4x_1x_3-x_2^2+x_4^2+\lambda x_0x_4=0\}.
$$
The quadric $Q_\lambda$ is smooth if and only if $\lambda\neq\pm2$. The spans of the two twisted cubics are $H_1=\{x_4=0\}$ and $H_2=\{x_0=0\}$. Hence
$$
        Q_1=\{x_4=0,\ x_0^2-4x_1x_3-x_2^2=0\},
        \qquad
        Q_2=\{x_0=0,\ -4x_1x_3-x_2^2+x_4^2=0\}.
$$
Both $Q_1$ and $Q_2$ are smooth quadric surfaces, independently of $\lambda$. Their intersection is the smooth conic $C=\{x_0=x_4=0,\ 4x_1x_3+x_2^2=0\}$. Moreover $\Gamma_i\cap C$ is given by $1+t^3=0$, hence consists of three reduced points.

A direct substitution in the two rulings of $Q_1$ and $Q_2$ gives the six lines required in Assumption~\ref{ass:general-setting}(iv). The corresponding transversality and non-tangency conditions reduce to the non-vanishing of finitely many first derivatives at the three parameters satisfying $1+t^3=0$, and these non-vanishing conditions hold for the displayed parametrizations. Finally, the curves $\Gamma_1$ and $\Gamma_2$ are disjoint by inspection of the two parametrizations. Thus, for every $\lambda\neq\pm2$, the triple $(Q_\lambda,\Gamma_1,\Gamma_2)$ satisfies Assumption~\ref{ass:general-setting}. We fix such a value of $\lambda$.

We denote by $X_{D_6}$ the threefold obtained from $(Q_\lambda,\Gamma_1,\Gamma_2)$ by Construction~\ref{constr:WVX}. Thus $X_{D_6}$ has two terminal points $x_1,x_2$ of type $\frac12(1,1,1)$.
\end{Construction}

\begin{Lemma}\label{lem:D6-action}
The threefold $X_{D_6}$ admits an action of a group $G\simeq D_6$, and this action has no $G$-fixed point.
\end{Lemma}

\begin{proof}
The pair $(Q_\lambda,\Gamma_1\cup\Gamma_2)$ is preserved by the projective transformations
$$
        \rho_\omega\colon
        [x_0:x_1:x_2:x_3:x_4]
        \mapsto
        [x_4:\omega x_1:-x_2:\omega^2x_3:x_0],
$$
and
$$
        \tau\colon
        [x_0:x_1:x_2:x_3:x_4]
        \mapsto
        [x_4:x_3:x_2:x_1:x_0].
$$
Indeed, both transformations preserve the equation of $Q_\lambda$. Moreover $\tau$ exchanges $\Gamma_1$ and $\Gamma_2$, while $\rho_\omega$ sends $\Gamma_1$ to $\Gamma_2$ and $\Gamma_2$ to $\Gamma_1$. More explicitly, on the affine parameter of $\Gamma_1$, the map $\rho_\omega$ sends $t$ to $\omega^2/t$ on $\Gamma_2$. The transformations satisfy $\rho_\omega^6=\tau^2=1$ and $\tau\rho_\omega\tau=\rho_\omega^{-1}$. Hence they generate a group $G\simeq D_6$.

Since the whole construction is functorial with respect to automorphisms preserving $(Q_\lambda,\Gamma_1\cup\Gamma_2)$, the action lifts to $W$, preserves the set of seven flopping curves, descends through the flops, and finally descends to $X_{D_6}$.

We now prove that the induced action on $X_{D_6}$ has no fixed point. First, $G$ has no fixed point on $Q_\lambda$. Indeed, a $G$-fixed point in $\PP^4$ spans a common eigenline for the matrices defining $\rho_\omega$ and $\tau$. The common eigenlines are
$$
        [1:0:0:0:1],
        \qquad
        [1:0:0:0:-1],
        \qquad
        [0:0:1:0:0].
$$
The values of the quadratic form defining $Q_\lambda$ at these three points are respectively $2+\lambda$, $2-\lambda$, and $-1$. Since $\lambda\neq\pm2$, none of these points lies on $Q_\lambda$. Thus $Q_\lambda$ has no $G$-fixed point.

The morphism $W\to Q_\lambda$ is $G$-equivariant. Hence a fixed point of $G$ on $W$ would map to a fixed point on $Q_\lambda$, which does not exist. Thus $W$ has no $G$-fixed point.

Consider now the passage from $W$ to $V$. Away from the flopping curves, the birational map $W\dashrightarrow V$ is an isomorphism, so no fixed point can appear there. The six flopping curves lying over the trisecant lines form a single $G$-orbit, so their flopped images also form a single $G$-orbit and cannot contain a common $G$-fixed point. The remaining flopping curve $\widetilde C$ is $G$-invariant. The induced action of $G$ on $\widetilde C\simeq\PP^1$ is the standard dihedral action, hence has no common fixed point. The flop of $\widetilde C$ is analytically the ordinary flop of a node; the two small resolutions carry the corresponding induced dihedral actions, and the action on the flopped curve $C_V\simeq\PP^1$ is again the standard dihedral action. Hence $C_V$ contains no $G$-fixed point.

Finally, the contraction $V\to X_{D_6}$ is $G$-equivariant and contracts only the two divisors $F_1$ and $F_2$. Their images are the two non-Gorenstein points $x_1$ and $x_2$, and these two points are exchanged by $\tau$. Thus neither of them is fixed by $G$. Since no other point can become fixed under the contraction, the action of $G$ on $X_{D_6}$ has no fixed point.
\end{proof}

\begin{Lemma}\label{lem:flopping-orbits-D6}
Among the seven flopping curves on $V$, the only $G$-invariant irreducible one is $C_V$. The remaining six flopping curves form a single $G$-orbit.
\end{Lemma}

\begin{proof}
The conic $C=Q_1\cap Q_2$ is preserved by $G$: indeed, the action preserves the unordered pair $\{Q_1,Q_2\}$, hence preserves their intersection. Therefore its flopped transform $C_V$ is $G$-invariant.

It remains to consider the six flopping curves coming from the lines $L_1,L_2,L_3,R_1,R_2,R_3$. The subgroup generated by $\rho_\omega^2$ cyclically permutes the three curves $R_1,R_2,R_3$, and also cyclically permutes the three curves $L_1,L_2,L_3$. Moreover, the involution $\tau$ exchanges the two configurations, sending the $R_i$'s to the $L_i$'s. Hence the six curves $L_1,L_2,L_3,R_1,R_2,R_3$ form a single $G$-orbit.

The flops are $G$-equivariant, so the same orbit decomposition holds for their flopped transforms on $V$. Thus none of these six curves is invariant as an irreducible curve. Consequently the only $G$-invariant irreducible flopping curve on $V$ is $C_V$.
\end{proof}

\begin{Lemma}\label{lem:Aut-D6}
One has $\Aut(X_{D_6})=G\simeq D_6$.
\end{Lemma}

\begin{proof}
The points $x_1,x_2$ are the only non-Gorenstein points of $X_{D_6}$, hence every automorphism preserves the set $\{x_1,x_2\}$. The Kawamata extraction of a terminal point of type $\frac12(1,1,1)$ is unique, so every automorphism of $X_{D_6}$ lifts uniquely to the extraction $\phi\colon V\to X_{D_6}$ and preserves the unordered pair of planes $F_1,F_2$.

The seven $K_V$-trivial curves are exactly the curves contracted by the anticanonical morphism of $V$. Thus every automorphism of $V$ preserves this set of seven curves. Flopping them back recovers $W=\Bl_{\Gamma_1\sqcup\Gamma_2}(Q_\lambda)$. Contracting the two exceptional divisors of $W\to Q_\lambda$ then recovers the pair $(Q_\lambda,\Gamma_1\cup\Gamma_2)$. Consequently every automorphism of $X_{D_6}$ is induced by an automorphism of $Q_\lambda$ preserving $\Gamma_1\cup\Gamma_2$.

We compute this stabilizer. Since the linear span of $\Gamma_i$ is $H_i$, an automorphism preserving $\Gamma_1\cup\Gamma_2$ either preserves both hyperplanes $H_1,H_2$ or exchanges them.

Suppose first that the automorphism preserves $H_1$ and $H_2$. Then it preserves $C=Q_\lambda\cap H_1\cap H_2$. Hence its restriction to $\Gamma_1\simeq \PP^1$ preserves the set $\Gamma_1\cap C$, which in the affine parameter is $\{t\mid t^3=-1\}$. The stabilizer of this set in $\PGL_2$ is the symmetric group on three letters, generated by $t\mapsto \omega t$ and $t\mapsto 1/t$.

These two transformations are realized by the projective transformations
$$
        \rho_\omega^2:
        [x_0:x_1:x_2:x_3:x_4]
        \mapsto
        [x_0:\omega^2x_1:x_2:\omega x_3:x_4]
$$
and
$$
        \sigma:
        [x_0:x_1:x_2:x_3:x_4]
        \mapsto
        [x_0:x_3:-x_2:x_1:x_4].
$$
Indeed, $\rho_\omega^2$ induces $t\mapsto \omega^2t$ on $\Gamma_1$, and $\sigma$ induces $t\mapsto 1/t$ on both $\Gamma_1$ and $\Gamma_2$. Both transformations preserve $Q_\lambda$, $\Gamma_1$, and $\Gamma_2$. Hence the subgroup preserving both $\Gamma_1$ and $\Gamma_2$ contains the group generated by $\rho_\omega^2$ and $\sigma$, which is isomorphic to $S_3$.

There are no further elements preserving both $\Gamma_1$ and $\Gamma_2$. Indeed, restriction to $\Gamma_1$ gives an injective homomorphism from the stabilizer of the ordered pair $(\Gamma_1,\Gamma_2)$ to $\Aut(\Gamma_1)\simeq \PGL_2$. To see injectivity, let an element act trivially on $\Gamma_1$. Since $\Gamma_1$ spans $H_1\simeq \PP^3$, it acts trivially on $H_1$. In particular it acts trivially on $H_1\cap H_2$. The curve $\Gamma_2$ meets $H_1\cap H_2$ in the three points $\Gamma_2\cap C$, so the induced automorphism of $\Gamma_2\simeq\PP^1$ fixes three points. Thus it is the identity on $\Gamma_2$. Since $\Gamma_2$ spans $H_2$, the automorphism is also the identity on $H_2$. Finally, $H_1$ and $H_2$ span $\PP^4$, so the automorphism is the identity on $\PP^4$. Therefore the ordered stabilizer injects into the stabilizer of $\{t^3=-1\}\subset\PP^1$, which has order $6$, and hence it is exactly the $S_3$ generated by $\rho_\omega^2$ and $\sigma$.

The involution $\tau$ exchanges $\Gamma_1$ and $\Gamma_2$. Therefore the full stabilizer of the unordered pair $\Gamma_1\cup\Gamma_2$ is obtained by adjoining $\tau$ to the ordered stabilizer. This gives a group of order $12$. Since this group contains the element $\rho_\omega$ of order $6$ and the involution $\tau$ satisfies $\tau\rho_\omega\tau=\rho_\omega^{-1}$, it is the dihedral group $D_6$. This stabilizer is precisely the group $G$ acting on the construction. Hence $\Aut(X_{D_6})=G\simeq D_6$.
\end{proof}

\begin{Lemma}\label{lem:general-anticanonical-section-D6}
Let $S\in |-K_{X_{D_6}}|$ be very general, where $|-K_{X_{D_6}}|$ denotes the anticanonical linear system. Then $S$ is a normal K3 surface with Du Val singularities. More precisely, $S$ is smooth away from $x_1,x_2$, it has two $A_1$ singularities at $x_1,x_2$, and $\Cl(S)=\ZZ[-K_{X_{D_6}}|_S]$.
\end{Lemma}

\begin{proof}
The strict transform map identifies the Weil linear system $|-K_{X_{D_6}}|$ with the Cartier linear system $|-K_V|$. Indeed, if $S\in |-K_{X_{D_6}}|$, then its divisorial pull-back is $\phi^*S=\widetilde S+\frac12(F_1+F_2)$. Since $K_V=\phi^*K_{X_{D_6}}+\frac12(F_1+F_2)$, it follows that $\widetilde S\sim -K_V$. Thus the strict transform $\widetilde S$ of a very general $S$ is a very general member of $|-K_V|$.

Let $Y^a$ be the anticanonical model of $V$. The morphism $V\to Y^a\subset \PP^9$ is induced by $|-K_V|$, is birational, and contracts exactly the seven flopping curves to ordinary double points. A very general hyperplane section $S_{Y^a}\subset Y^a$ avoids these nodes. Hence its strict transform on $V$ is isomorphic to $S_{Y^a}$ and equals $\widetilde S$. By Bertini, $\widetilde S$ is smooth.

The surface $S$ is obtained from $\widetilde S$ by contracting the two curves $\widetilde S\cap F_1$ and $\widetilde S\cap F_2$. Since $-K_V|_{F_i}\sim \OO_{\PP^2}(1)$, each curve $\widetilde S\cap F_i$ is a line on $F_i\simeq \PP^2$. The contraction of such a line to the point of type $\frac12(1,1,1)$ gives an $A_1$ singularity on the surface section. Thus $S$ is smooth away from $x_1,x_2$ and has precisely two $A_1$ singularities at these points.

Moreover, $\widetilde S\in |-K_V|$ is smooth and $K_{\widetilde S}\sim 0$ by adjunction. Since the morphism $\widetilde S\to S$ contracts two $(-2)$-curves and is crepant, the singularities of $S$ are Du Val and $K_S\sim 0$. Also $h^1(\OO_S)=h^1(\OO_{\widetilde S})=0$. Hence $S$ is a normal K3 surface with Du Val singularities.

We now compute $\Cl(S)$. We apply the Ravindra--Srinivas Noether--Lefschetz theorem to the anticanonical model $Y^a\subset\PP^9$; see \cite[Theorem~2]{RS09}. Since $S_{Y^a}$ avoids the nodes of $Y^a$, the small morphism $V\to Y^a$ is an isomorphism over $S_{Y^a}$, so $\widetilde S\simeq S_{Y^a}$. The theorem gives that the restriction map $\Cl(Y^a)\to \Pic(S_{Y^a})$ is an isomorphism. Since $V\to Y^a$ is small, we identify $\Cl(Y^a)$ with $\Pic(V)$. Therefore $\Pic(\widetilde S)\simeq \Pic(V)$.

Finally, $\widetilde S\to S$ is the minimal resolution of the two $A_1$ singularities, with exceptional curves $\widetilde S\cap F_1$ and $\widetilde S\cap F_2$. For a normal surface with rational double points, the class group is the quotient of the Picard group of the minimal resolution by the lattice generated by the exceptional curves. Hence $\Cl(S)\simeq \Pic(\widetilde S)/\langle F_1|_{\widetilde S},F_2|_{\widetilde S}\rangle$. Under the identification $\Pic(\widetilde S)\simeq \Pic(V)$, this quotient is the same as $\Pic(V)/\langle F_1,F_2\rangle$. By Lemma~\ref{lem:class-group-X}, this group is generated by the descent of $-K_V$, namely by $-K_{X_{D_6}}|_S$. Therefore $\Cl(S)=\ZZ[-K_{X_{D_6}}|_S]$.
\end{proof}

\begin{Definition}\label{def:local-delta-polarized}
Let $X$ be a normal projective variety, let $L$ be a big $\QQ$-Cartier divisor on $X$, and let $Z\subset X$ be a subvariety. We define
$$
        \delta_Z(X;L)
        :=
        \inf_F
        \frac{A_X(F)}{S_L(F)},
$$
where the infimum is taken over all prime divisors $F$ over $X$ such that $Z\subset \Center_X(F)$. Here $A_X(F)$ is the log discrepancy of $F$ with respect to $X$, and
$$
        S_L(F)
        :=
        \frac{1}{L^n}
        \int_0^{+\infty}
        \Vol(\pi^*L-uF)\,du,
        \qquad n:=\dim X,
$$
where $\pi\colon Y\to X$ is any birational model on which $F$ appears.

Similarly, if $S$ is a normal projective variety, $L_S$ is a big $\QQ$-Cartier divisor on $S$, and $Z_0\subset S$ is a subvariety, we set
$$
        \delta_{Z_0}(S;L_S)
        :=
        \inf_F
        \frac{A_S(F)}{S_{L_S}(F)},
$$
where the infimum is taken over all prime divisors $F$ over $S$ such that $Z_0\subset \Center_S(F)$.
\end{Definition}

\begin{Definition}\label{def:S-divisor-expected-vanishing}
Let $X$ be a normal projective variety of dimension $n$, let $L$ be a big $\QQ$-Cartier divisor on $X$, and let $S\subset X$ be a prime divisor. The expected vanishing order of $S$ with respect to $L$ is
$$
        S(L;S)
        :=
        \frac{1}{L^n}
        \int_0^{+\infty}
        \Vol(L-uS)\,du.
$$
If $L\sim_{\QQ} rS$ with $r>0$ and $S$ is ample, then
$$
        S(L;S)
        =
        \frac{1}{(rS)^n}
        \int_0^r (r-u)^nS^n\,du
        =
        \frac{r}{n+1}.
$$
\end{Definition}

\begin{Lemma}\label{lem:AZ-QCartier-adjunction}
Let $X$ be a normal projective variety of dimension $n$, and suppose that $S\subset X$ is an ample $\QQ$-Cartier prime divisor, Cartier is codimension 2, such that $(X,S)$ is plt. Let $Z\subset X$ be a subvariety, and let $Z_0$ be an irreducible component of $Z\cap S$. Then:
$$
        \delta_Z(X;S)
        \geq
        \min\left\{
        n+1,\,
        \frac{n+1}{n}\delta_{Z_0}(S;S|_S)
        \right\}.
$$
\end{Lemma}

\begin{proof}
Choose $r>0$ such that $L:=\OO_X(rS)$ is a line bundle. To obtain the above inequality, apply \cite[Corollary~3.4]{AZ20} to the complete graded linear series of $L$ and to its refinement, $W_{\bullet,\bullet}$, induced by the divisor $S$ (cf. Notation \ref{not:refinement-by-E}). For this to work, one must show that the linear series $W_{\bullet,\bullet}$ has zero aymptotic fixed-part, denoted by $F(W_{\bullet,\bullet}),$ and is almost complete. We refer the reader to \cite[Definitions 2.25, 2.27]{AZ20} for this terminology. To this end, by applying Serre vanishing to the reflexive sheaves $\OO_X(-(j+1)S)$ for $j=0,1,\ldots,r-1$ one finds a positive integer $m_0$ such that $\forall m\geq m_0$: 
$$H^1\big(X,mL-rS\big)=H^1\big(X,mL-(r-1)S\big)=\cdots=H^1\big(X,mL-S\big)= 0.$$
It follows then that for all $m\geq m_0$ and $j\leq r(m-m_0)$, the restriction maps 
$$H^0\big(X,mL-jS\big)\to H^0\big(S,mL|_S-jS|_S\big)$$
are surjective, and hence that the fixed-part, $F_{m,j}$, of the linear system $|W_{m,j}|$ has no fixed-part for $m$ and $j$ in this range since $S_S$ is ample for $n\gg0$. Therefore the sequence of $\QQ$-divisors $F_m$ given by:
$$F_m=F_m(W_{\bullet,\bullet}):=\frac{\sum_{j=r(m-m_0)+1}^{rm}\mathrm{dim}(W_{m,j})F_{m,j}}{m\sum_{j=0}^{rm}\mathrm{dim}(W_{m,j})}$$
converges to zero as $m\to\infty$ since the numerator grows linearly with ; so $F(W_{\bullet,\bullet})=0$. Moreover, for each $j\in\Q$ such that $W_{1,j}\neq0$, we have that 
$$W_{m,mj}=H^0\Big(S,\Big(1-\frac{j}{r}\Big)mL|_S\Big)$$
for $m\gg0$ and divisible by $1-\frac{j}{r}$. Thus, the conditions of \cite[Definition 2.27]{AZ20} are satisfied and $W_{\bullet,\bullet}$ is almost complete. \\ 

To finish the proof, by the discussion in \cite[\textsection 3]{AZ20} and the fact that $(X,S)$ is plt, the first Chern class $c_1$ of $W_{\bullet,\bullet}$ is equal $L|_S-S(L;S)S|_S,$ which by Definition \ref{def:S-divisor-expected-vanishing} is equal to $\frac{n}{n+1}L|_S$. Since $S$ is assumed to be Cartier in codimension 2, the different $\Delta_S$ of $S$ vanishes. Thus \cite[Corollary~3.4]{AZ20} gives
$$
        \delta_{Z}(X;L)
        \geq
        \min\left\{
        \frac{1}{S(L;S)},
        \delta_{Z_0}\left(S;\frac{n}{n+1}L|_S\right)
        \right\}.
$$
Since $L=rS$, the first term is $(n+1)/r$. The local delta invariant is homogeneous in the polarization: for $a>0$, one has $\delta_{Z_0}(S;aL_S)=a^{-1}\delta_{Z_0}(S;L_S)$. Therefore, after multiplying by the common factor $r$, the inequality becomes
$$
        \delta_{Z}(X;S)
        \geq
        \min\left\{
        n+1,\,
        \frac{n+1}{n}\delta_{Z_0}(S;S|_S)
        \right\}.
$$
This is the desired statement.
\end{proof}

\begin{Lemma}\label{lem:point-threshold-rank-one-K3}
Let $X$ be a $\QQ$-Fano threefold, Gorenstein in codimension 2, with $(-K_X)^3<16$, and let $x\in X$ be a smooth point. Suppose that $S\in |-K_X|$ is a Du Val K3 surface, smooth at $x$, and $\Cl(S)=\ZZ[-K_X|_S]$. Then $\delta_x(X)>1$.
\end{Lemma}

\begin{proof}
Set $H:=-K_X|_S$. Since $x$ is a smooth point of both $X$ and $S$, the pair $(X,S)$ is plt near $x$. Applying Lemma~\ref{lem:AZ-QCartier-adjunction} locally at $x$, with $n=3$, gives
$$
        \delta_x(X)
        \geq
        \min\left\{4,\frac43\delta_x(S;H)\right\}.
$$
Thus it is enough to prove that $\delta_x(S;H)>3/4$.

Let $\epsilon_x(H)$ be the Seshadri constant of $H$ at $x$, and let $\tau_x(H)$ be the pseudo-effective threshold of the exceptional divisor over $x$ on the blow-up of $S$ at $x$. Equivalently,
$$
        \tau_x(H)
        =
        \sup\{\,\mult_x(D)\mid D\geq 0,\ D\sim_{\QQ}H\,\}.
$$
By \cite[Lemma~3.2]{AZZ23}, one has
$$
        \delta_x(S;H)\geq \frac{3}{H^2}\epsilon_x(H).
$$
Since $\Cl(S)=\ZZ[H]$ has rank $1$, \cite[Lemma~2.4]{AZZ23} gives $\epsilon_x(H)\tau_x(H)=H^2$. Hence
$$
        \delta_x(S;H)\geq \frac{3}{\tau_x(H)}.
$$

We claim that $\tau_x(H)\leq 4$. Suppose otherwise. Then there exists an effective $\QQ$-divisor $D\sim_{\QQ}H$ with $\mult_x(D)>4$. Write $D=\sum a_iC_i$, where $a_i>0$ and $C_i$ are integral curves. Since $\Cl(S)=\ZZ[H]$, we may write $C_i\sim m_iH$ with $m_i\in\ZZ_{\geq 1}$. Since $D\sim_{\QQ}H$, one has $\sum a_im_i=1$. If every component satisfied $\mult_x(C_i)\leq 4m_i$, then $\mult_x(D)\leq \sum a_i4m_i=4$, a contradiction. Hence one component, say $C$, satisfies $C\sim mH$ and $\mult_x(C)>4m$. Since $\mult_x(C)$ is an integer, $\mult_x(C)\geq 4m+1$.

Since $S$ has only Du Val singularities and $x$ is smooth, the usual genus estimate may be applied on the minimal resolution. Thus
$$
        C^2\geq \mult_x(C)(\mult_x(C)-1)-2.
$$
On the other hand, $C^2=m^2H^2=m^2(-K_X)^3<16m^2$. Therefore
$$
        16m^2
        >
        C^2
        \geq
        \mult_x(C)(\mult_x(C)-1)-2
        \geq
        (4m+1)4m-2
        >
        16m^2,
$$
a contradiction. Hence $\tau_x(H)\leq 4$, and consequently $\delta_x(S;H)\geq 3/4$.

It remains to exclude equality. If $\delta_x(S;H)=3/4$, then all inequalities above are equalities, so $\tau_x(H)=4$ and equality holds in the estimate of \cite[Lemma~3.2]{AZZ23}. The equality cases in that lemma imply either $H^2=\tau_x(H)^2=16$, contradicting $H^2=(-K_X)^3<16$, or $H\equiv4C_0$ for some curve $C_0\subset S$ computing the threshold. The latter is impossible since $H$ is the primitive generator of $\Cl(S)=\ZZ[H]$. Thus $\delta_x(S;H)>3/4$.

Therefore $\frac43\delta_x(S;H)>1$, and the adjunction inequality gives $\delta_x(X)>1$.
\end{proof}

\begin{Proposition}\label{prop:D6-centers-except-l}
Let $Z\subset X_{D_6}$ be a $G$-invariant irreducible subvariety with $Z\neq l$. Then $\delta_Z(X_{D_6})>1$.
\end{Proposition}

\begin{proof}
The center $Z$ is not a point by Lemma~\ref{lem:D6-action}. Suppose first that $Z$ is a divisor. By Lemma~\ref{lem:class-group-X}, one has $Z\sim -rK_{X_{D_6}}$ for some $r\in\ZZ_{>0}$. Since $Z$ is a prime divisor on $X_{D_6}$ itself, $A_{X_{D_6}}(Z)=1$. Moreover, for $L:=-K_{X_{D_6}}$, the divisor $Z$ is $\QQ$-linearly equivalent to $rL$, hence
$$
        S_L(Z)
        =
        \frac{1}{L^3}
        \int_0^{1/r}(L-uZ)^3\,du
        =
        \frac{1}{L^3}
        \int_0^{1/r}(1-ru)^3L^3\,du
        =
        \frac{1}{4r}.
$$
Therefore $\delta_Z(X_{D_6})=A_{X_{D_6}}(Z)/S_L(Z)=4r>1$.

Thus we may assume that $Z$ is a curve. Let $S\in |-K_{X_{D_6}}|$ be very general. If $S$ meets $Z$ at a smooth point $x\in X_{D_6}$, then $S$ is smooth at $x$ by Lemma~\ref{lem:general-anticanonical-section-D6}. Since $x\in Z$, the local thresholds satisfy $\delta_Z(X_{D_6})\geq \delta_x(X_{D_6})$. Lemmas~\ref{lem:general-anticanonical-section-D6} and~\ref{lem:point-threshold-rank-one-K3} give $\delta_x(X_{D_6})>1$, hence $\delta_Z(X_{D_6})>1$.

We may therefore assume that a very general $S$ meets $Z$ only at the two singular points $x_1,x_2$. Let $\widetilde Z,\widetilde S\subset V$ be the strict transforms. If $\widetilde Z$ meets $F_i$, it does so in finitely many points, since $Z$ is a curve on $X_{D_6}$ and is not one of the singular points. The curves $\widetilde S\cap F_i$ vary in $|\OO_{F_i}(1)|$. Hence, for $S$ very general, they avoid the finite set $\widetilde Z\cap(F_1\cup F_2)$. Since $S\cap Z\subset\{x_1,x_2\}$, this implies $\widetilde Z\cap\widetilde S=\varnothing$.

Thus $(-K_V)\cdot\widetilde Z=\widetilde S\cdot\widetilde Z=0$. Since $-K_V$ is nef and semiample, $\widetilde Z$ is contracted by the anticanonical morphism of $V$. The exceptional locus of this morphism is exactly the union of the seven flopping curves. Hence $\widetilde Z$ is one of these seven curves. By Lemma~\ref{lem:flopping-orbits-D6}, the only $G$-invariant irreducible flopping curve is $C_V$, whose image is $l$. This contradicts $Z\neq l$.

Therefore $\delta_Z(X_{D_6})>1$.
\end{proof}

\begin{thm}\label{thm:D6-member-K-stable}
The Fano threefold $X_{D_6}$ is K-stable.
\end{thm}

\begin{proof}
By Proposition~\ref{prop:D6-centers-except-l} and Proposition~\ref{prop:l-not-destabilizing-final}, every $G$-invariant irreducible center $Z\subset X_{D_6}$ satisfies $\delta_Z(X_{D_6})>1$.

We claim first that $X_{D_6}$ is K-polystable. Indeed, if $X_{D_6}$ were not K-polystable, then the equivariant valuative criterion, together with Zhuang's equivalence between equivariant and ordinary K-polystability \cite{Zhuang21}, would give a $G$-invariant prime divisor $F$ over $X_{D_6}$ such that $A_{X_{D_6}}(F)\leq S_{X_{D_6}}(F)$. Its center $Z:=\Center_{X_{D_6}}(F)$ would be a $G$-invariant irreducible subvariety, and the inequality $A_{X_{D_6}}(F)\leq S_{X_{D_6}}(F)$ would imply $\delta_Z(X_{D_6})\leq 1$. This contradicts the preceding paragraph. Hence $X_{D_6}$ is K-polystable.

By Lemma~\ref{lem:Aut-D6}, one has $\Aut(X_{D_6})=G\simeq D_6$, in particular $\Aut(X_{D_6})$ is finite. Thus $X_{D_6}$ has no non-trivial one-parameter subgroup of automorphisms. Since the equality case in K-polystability is allowed only for product test configurations, and non-trivial product test configurations are induced by non-trivial one-parameter subgroups of $\Aut(X_{D_6})$, every product test configuration of $X_{D_6}$ is trivial. Therefore K-polystability is the same as K-stability for $X_{D_6}$. Hence $X_{D_6}$ is K-stable.
\end{proof}

\begin{Corollary}\label{cor:general-member-K-stable}
A general genus $8$ Fano threefold with two singularities of type $\frac12(1,1,1)$ in the family considered above is K-stable.
\end{Corollary}

\begin{proof}
The threefold $X_{D_6}$ belongs to the family and is K-stable by Theorem~\ref{thm:D6-member-K-stable}. Since $\Aut(X_{D_6})$ is finite, K-stability of $X_{D_6}$ implies uniform K-stability, that is, $\delta(X_{D_6})>1$. By the openness of uniform K-stability in $\QQ$-Gorenstein families of $\QQ$-Fano varieties, or equivalently by the lower semicontinuity of the $\delta$-invariant \cite{BLX22}, the locus of fibers with $\delta>1$ is a Zariski open subset of the base. This open subset is nonempty since it contains the point corresponding to $X_{D_6}$. Hence a general member of the family is K-stable.
\end{proof}
\bibliographystyle{amsalpha}
\bibliography{Biblio}
\end{document}